\documentclass{amsart}

\makeatletter
\newcommand\notsotiny{\@setfontsize\notsotiny\@vipt\@viipt}
\makeatother

 \usepackage{mlmodern}

\usepackage[T1]{fontenc}

\usepackage[cal=boondoxo]{mathalfa}
\DeclareSymbolFontAlphabet{\mathcalorig} {symbols}

\usepackage{amssymb}
\usepackage{amsmath}
\usepackage{amsthm}
\usepackage{amssymb}
\usepackage{amsfonts}
\usepackage{latexsym}
\usepackage{float}

\usepackage{longtable}
\usepackage{afterpage}
\usepackage{array}
\usepackage{verbatim}
\usepackage{graphicx} 
\usepackage[svgnames,x11names]{xcolor}
\usepackage{mysects}  
\usepackage{fancyhdr} 
\usepackage[labelformat=simple]{caption}
\usepackage{setspace}
\usepackage{bm}

\usepackage{enumitem}

\usepackage{microtype}
\usepackage{lineno}
\usepackage{xcolor}
 \usepackage[normalem]{ulem}
 \newcommand\rsout{\bgroup\markoverwith{\textcolor{red}{\rule[0.5ex]{2pt}{1pt}}}\ULon}

\usepackage{hhline}
\def\thickhline{\noalign{\hrule height.8pt}}

\definecolor{refkey}{rgb}{1,1,1}   
\definecolor{labelkey}{rgb}{0,1,0} 
\definecolor{citekey}{rgb}{0,0,0}

\usepackage[multiple]{footmisc}
\usepackage[hidelinks]{hyperref}
\hypersetup{
    colorlinks=false, 
    linktoc=all,     
    linkcolor=DarkBlue,  
	citecolor=Black,
	}

\usepackage{appendix}
\newenvironment{psmatrix}
{\left[\begin{smallmatrix}}
{\end{smallmatrix}\right]}

\usepackage[
doi=false,isbn=false,url=true,eprint=true,style=alphabetic,sorting=nyt,maxbibnames=99,maxcitenames=99]{biblatex}
\renewbibmacro{in:}{}
\AtBeginBibliography{\footnotesize}
\renewcommand{\url}[1]{\nolinkurl{#1}}

\newtheorem{theorem}{Theorem}[section]

\newtheorem*{theorem*}{Theorem}
\newtheorem{proposition}[theorem]{Proposition}
\newtheorem*{proposition*}{Proposition}
\newtheorem{lemma}[theorem]{Lemma}
\newtheorem*{lemma*}{Lemma}
\newtheorem{corollary}[theorem]{Corollary}

\newtheorem*{question*}{Question}
\theoremstyle{definition}
\newtheorem{definition}[theorem]{Definition}

\newtheorem{remark}[theorem]{Remark}

\newtheorem*{acknowledgements}{Acknowledgements}

\newtheorem{notation}[theorem]{Notation}

\newcommand{\Aut}{\operatorname{Aut}}

\newcommand{\GL}{\operatorname{GL}}

\newcommand{\SL}{\operatorname{SL}}

\newcommand{\Vol}{\operatorname{Vol}}

\newcommand{\meas}{\operatorname{\meas}}

\newcommand{\half}{\frac{1}{2}}

\newcommand{\nc}{\newcommand}

\nc{\la}{\langle} \nc{\ra}{\rangle}
 \nc{\CA}{\cal A}
 \nc{\CBB}{\cal B}

\nc{\CDD}{\cal D}
\nc{\CE}{\cal E}
\nc{\CF}{\cal F} \nc{\CG}{\cal
G} \nc{\CH}{\cal H} \nc{\CI}{\cal I} \nc{\CJ}{\cal J}
\nc{\CK}{\cal K} \nc{\CL}{\cal L} \nc{\CM}{\cal M} \nc{\CN}{\cal
N} \nc{\CO}{\cal O} \nc{\CP}{\cal P} \nc{\CQ}{\cal Q}
\nc{\CR}{\cal R} \nc{\CS}{\cal S} \nc{\CT}{\cal T} \nc{\CU}{\cal
U} \nc{\CV}{\cal V} \nc{\CW}{\cal W} \nc{\CZ}{\cal Z}

\makeatletter
\g@addto@macro\bfseries{\boldmath}
\makeatother

\nc{\fa}{  \mathfrak a} \nc{\fg}{  \Gen} \nc{\fk}{  \mathfrak k}
\nc{\fh}{  \mathfrak h} \nc{\fm}{  \mathfrak m} \nc{\fn}{  \mathfrak n}
\nc{\fA}{  \mathfrak A} \nc{\fC}{  \Cl} \nc{\fI}{  \mathfrak I}
\nc{\fL}{  \mathfrak L} \nc{\fS}{  \mathfrak S}
\nc{\fz}{  \mathfrak z}

\nc{\nen}{\newenvironment} \nc{\ol}{\overline}
\nc{\ul}{\underline} \nc{\lra}{\longrightarrow}
\nc{\lla}{\longleftarrow} \nc{\Lra}{\Longrightarrow}
\nc{\Lla}{\Longleftarrow} \nc{\Llra}{\Longleftrightarrow}
\nc{\hra}{\hookrightarrow} 

\newcommand{\rt}{\textup{rt}}

\def\C{\mathbb C}

\def\Q{\mathbb Q}
\def\R{\mathbb R}
\def\Z{\mathbb Z}

\def\Zp{\mathbb{Z}_p}

\definecolor{deepmagenta}{rgb}{0.8, 0.0, 0.8}
 	\definecolor{earthyellow}{rgb}{0.88, 0.66, 0.37}
 	\definecolor{deepp}{rgb}{0.5, 0.0, 0.13}
    \definecolor{darkgreen}{rgb}{0.13, 0.8, 0.13}
    \definecolor{rpurple}{rgb}{1.02,0.414,0.421 }
    \definecolor{brown}{rgb}{0.788, 0.294, 0.0}
    
\newcommand*{\rom}[1]{\expandafter\@slowromancap\romannumeral #1@}
\newcommand{\RM}[1]{%
  \textup{\uppercase\expandafter{\romannumeral#1}}%
}

\newcommand{\Mod}[1]{\ (\text{mod}\ #1)}

\newcommand{\Cl}{\textup{Cl}}
\newcommand{\Ge}{\textup{Ge}}
\newcommand{\Gen}{\textup{Gen}}
\newcommand{\Packet}{\textup{Pkt}}
\newcommand{\rPKT}{\mathfrak R}
\newcommand{\iso}{\textup{iso}}
\newcommand{\prim}{\textup{prim}}

\DeclareMathOperator{\Sq}{Sq}

\DeclareMathOperator{\Supp}{Supp}
\DeclareMathOperator{\Sym}{Sym}

\makeatletter
\@addtoreset{equation}{section}
\makeatother
\numberwithin{equation}{subsection}
 
 \nc{\bq}{\mathbb Q}
 \nc{\br}{\mathbb R}
 \nc{\bz}{\mathbb Z}
 \nc{\bc}{\mathbb C}
 \nc{\bn}{\mathbb N}
 \nc{\bX}{\mathbb{X}}
 \nc{\cL}{ \mathcal{L}}
 \nc{\G}{\Gamma}
 \nc{\sm}{\setminus}
 \nc{\sub}{\subset}
 \nc{\lm}{\lambda}
  \nc{\Lm}{\Lambda}
 \nc{\al}{\alpha}
 \nc{\bt}{\beta}
 \nc{\om}{\omega}
 \nc{\dl}{\delta}
 \nc{\g}{\gamma}
 \nc{\Dl}{\Delta}
 \nc{\Om}{\Omega}
 \nc{\s}{\sigma}
 \nc{\ro}{\rho}
 \nc{\te}{\theta}
 \nc{\SLZ}{\operatorname{SL}_2(\bz)}
 \nc{\SLZp}{\operatorname{SL}_2(\bz[\frac{1}{p}])}
 \nc{\SLZl}{\operatorname{SL}_2(\bz[\frac{1}{\ell}])}
 \nc{\SLR}{\operatorname{SL}_2(\br)}
 \nc{\GLR}{\operatorname{GL}_2(\br)}
 \nc{\PGLR}{\operatorname{PGL}_2(\br)}
 \nc{\PSLR}{\operatorname{PSL}_2(\br)}
 \nc{\PSLZ}{\operatorname{PSL}_2(\bz)}
 
 \nc{\SLC}{\operatorname{SL}_2(\bc)}
 \nc{\uH}{\mathbb H}
 \nc{\fD}{ \mathcal{D}}
 \nc{\fE}{ \mathcal{E}}
 \nc{\fO}{ \mathcal{O}}
 \nc{\haf}{\frac{1}{2}}
 \nc{\qtr}{\frac{1}{4}}
 \nc{\shaf}{{\scriptstyle\frac{1}{2}}}
 \nc{\hlm}{{\scriptstyle\frac{\lambda}{2}}}

 \nc{\8}{\infty}
 \nc{\7}{{-\infty}}
 \nc{\inv}{^{-1}}
 \nc{\eps}{\varepsilon}
 \nc{\aG}{\mathbf{G}}
 \nc{\spn}{\operatorname{Span}}
 \nc{\Cm}{\operatorname{CM}}

 \nc{\sgn}{\operatorname{sgn}}

\def\ve {\varepsilon}

\nc{\new}{{\color{red}{\label{ATTENTION ******}}}}

\nc{\upar}{\upshape\arabic*.\ }
\nc{\upal}{itemsep=0.3cm,label=\upshape\alph*.}

\newcounter{daggerfootnote}

\usepackage{chngcntr}

\begin{document}


\renewcommand{\footnoterule}{{\hrule}\vspace{3.5pt}} 
\renewcommand{\thefootnote}{(\arabic{footnote})}

\title{On indefinite integral  ternary quadratic forms}

\author[A. Gamburd]{Alexander Gamburd}
\address{CUNY Graduate Center, NY, NY 10016, USA \vspace*{.3em}}
\email{agamburd@gmail.com}

\author[A. Ghosh]{Amit Ghosh}
\address{Oklahoma State University, Stillwater, OK 74078, USA \vspace*{.3em}}
\email{ghosh@okstate.edu}

\author[P. Sarnak]{Peter Sarnak}
\address{Institute for Advanced Study \&  Princeton University, Princeton, NJ 08540, USA \vspace*{.3em}}
\email{sarnak@math.ias.edu \vspace{.4em}}

\author[J. P. Whang]{Junho Peter Whang}
\address{\parbox{11cm}{Department of Mathematical Sciences \& RIMS, Seoul National University,
Seoul  08826, S. Korea \vspace*{.3em}}}
\email{jwhang@snu.ac.kr}

\begin{abstract}We resolve two problems pertaining to indefinite integral ternary quadratic
forms, one highlighted by Margulis and the other initiated by Serre, both from 1990.  To do so we develop tools for dealing with high ramification in problems involving sums over classes of such forms weighted by  their diophantine invariants.
\end{abstract}

\maketitle

{\footnotesize \tableofcontents}


\setcounter{footnote}{0}

\pagestyle{myheadings}



\renewcommand{\thefootnote}{(\arabic{footnote}) } 



\section{Introduction}\label{intro}\ 

We answer two well studied questions concerning indefinite  ternary quadratic forms. The first was initiated by Markoff \cite{markoff1902}, who examined the values 
\begin{equation}\label{mu}
\mu(F)= \frac{1}{|D(F)|}\left(\inf_{\substack{{\bf x}\in \mathbb{Z}^3\backslash\{{\bf 0}\}}}|F({\bf x})|\right)^3,
\end{equation}
as $F=F(x_1,x_2,x_3)$ varies over real ternary indefinite quadratic forms of non-zero determinant $D(F)$. Here $F({\bf x}) = f_{11}x_1^2+f_{22} x_2^2+f_{33}x_3^2+2f_{12}x_1x_2+2f_{13}x_1 x_3 +2f_{23}x_2x_3 = {\bf x}^t A 
{\bf x} $ with symmetric $A =(f_{ij})$.   $F$ is \emph{integral} if $f_{ij} \in \mathbf{Z}$ and \emph{primitive} if $\text{GCD}(f_{ij})=1$. Since $D(-F)=-D(F)$, it suffices to limit our consideration to $F$'s with positive determinant.   The quantity $\mu(F)$ is invariant by real dilations of $F$, and also by the action of $\GL(3,\mathbb{Z})$ on the $F$'s induced by the linear change of the variable ${\bf x}$. The set of values $\mu(F)$ taken up to the above equivalence is the corresponding Markoff spectrum. The largest three points were determined by Markoff together with the $F$'s achieving them, namely
\begin{equation}\label{mu0}
\mu_1=\frac{2}{3}, \quad \mu_2=\frac{2}{5},\quad \mu_3=\frac{1}{3},
\end{equation}
with $\mu_1$ achieved by the form $F=2x_1^2-2x_2^2-2x_3^2+ 2x_1x_2+2x_2x_3+2x_1x_3$.

Oppenheim \cite{oppen} determined the next point $\frac{8}{25}$,\footnote{Venkov \cite{venkov} determined $\mu_5> \mu_6> \ldots >\mu_{11}=\frac{2}{9}$.} and provided the {insightful conjecture} that these points 
\begin{equation}\label{mu2}
\frac{2}{3}= \mu_1 > \mu_2 > \mu_3 > \ldots \mu_k > \ldots ,
\end{equation}
are discrete and converge to zero, and correspond to classes of $F$'s which are integral and primitive. The celebrated proof of the Oppenheim conjecture by Margulis \cite{MARG1987, Mar-Opp} established this {correspondence}, and that  for all $X\geq 1$
\begin{equation}\label{mu3}
\text{MAR}(X) = \#\left\{\mu_j \geq \frac{1}{X}\right\},
\end{equation}
is finite. In his 1990 ICM plenary address \cite[p.~208]{Margulis1990DynamicalAE}, Margulis asked for a bound for this count.\footnote{The $\mu_j$'s are counted with multiplicity. Mohammadi \cite{mohammadi}, using \cite{emv}, gave an upper bound of $O(X^A)$ with a large $A$ for $\text{MAR}(X)$. See also Li and Margulis~\cite{Mar-Li} for further applications of homogeneous dynamics to closely related questions.}

Following suggestions by Grunewald and Margulis, Martini \cite{Martini} in his thesis made a numerical study of the asymptotic behavior of $\text{MAR}(X)$ as $X\to \infty$. He goes as far as finding that $\mu_{145}=\frac{16}{245}$ and is led to conjecture that
\[
\text{MAR}(X) \approx 1.2\, X^2,
\]
(see also \cite{MARG1997} Sec.~2.5).

It turns out that his data did not go far enough, as the following shows:

\begin{theorem}\label{IThm1}\  
There is a positive constant $\gamma$ such that
\[
\textrm{\textup{MAR}}(X) \sim \gamma X\log X, \quad \text{as}\quad X  \to \infty.
\]
\end{theorem}

\begin{remark}
$\gamma$ is given by an explicit convergent infinite series of positive terms  (see Sec.~\ref{Mar-sec}, \eqref{mar-asym-form}) {which in principle can be used to estimate $\gamma$ numerically}. 
\end{remark}

The second question that we answer concerns the density of integral isotropic forms $F$, that is ones for which 
\begin{equation}\label{kappaF}
\kappa(F) = \inf_{\substack{{\bf x}\in \mathbb{Z}^3\backslash\{{\bf 0}\}}}|F({\bf x})|,
\end{equation}
is zero (i.e.  $F({\bf x})=0$ has an integral  non-zero solution). 

Let $\Sym(\R)\simeq\R^6$ denote the affine space of real symmetric $3\times 3$ matrices. Given a large set of $F$'s, say those in $X\Omega$, the dilate of a compact convex domain $\Omega \subset \Sym(\R)$ with piecewise smooth boundary containing the origin in its interior, Serre \cite{serre} (in the case $\Omega=[-1,1]^6$) asked to estimate how many are isotropic. Using sieve methods, he showed that 
\begin{equation}\label{iso1}
\textrm{\textup{ISO}}_{\prim}(X\Omega)= \#\left\{F\in X\Omega: F\  \text{isotropic and primitive}\right\} \ll_{\Omega} \frac{X^6}{\sqrt{\log X}}.\footnote{The addition of the primitivity condition in this counting is natural from our setup as well as in interpreting the count as the rational $F$'s of height at most $X$ when $\Omega = [-1,1]^6$.}
\end{equation}

Also using sieve methods, Hooley \cite{Hooley} established a matching lower bound
\begin{equation}\label{iso2}
\textrm{\textup{ISO}}_{\prim}(X\Omega)  \gg_{\Omega} \frac{X^6}{\sqrt{\log X}}.
\end{equation}

Using homogeneous dynamics and in particular the equidistribution results of Eskin and Oh \cite{eskinoh} (which in turn rely on Ratner's unipotent measure rigidity) we show that the isotropic $F$'s have a natural density.

\begin{theorem}\label{IThm2}\ 
We have
\[ \textrm{\textup{ISO}}_{\prim}(X\Omega) \sim \varpi\frac{\Vol(\Omega^{\iso})X^6}{\sqrt{\log X}}, \quad \text{as}\quad X\to \infty,
\]
where $\Vol(\Omega^{\iso})$ is the Euclidean volume of the locus $\Omega^{\iso}$ of real isotropic forms in $\Omega$ and
\[\varpi=\frac{2}{\Gamma(1/2)}\prod_p \lambda_p^{\iso}\cdot\left(1- p^{-1}\right)^{-\frac{1}{2}},\] where $\lambda_p^{\iso}$ for each prime $p$ is the $p$-adic probability  that a symmetric $3\times 3$ matrix with $\Z_p$-coefficients is primitive and isotropic (see Sec.~\ref{comp} for the explicit expression and comparison with conjectures).

\end{theorem}

\begin{remark}\label{rem1.4} Our proof of Theorem~\ref{IThm2} yields a stronger local density for the primitive integral isotropic $F$'s in the form: 
$$ d \lambda^{\text{iso}} = \varpi \frac{\prod_{i \leq j} d f_{ij}}{\sqrt{\log^{+} D(F)}}, $$ restricted to the real cone of isotropic $F$'s.  See Sec.~\ref{iso-revisited} for the exact statement and details.\end{remark}

The proofs of Theorems~\ref{IThm1} and~\ref{IThm2} rely on some weighted class number asymptotics which are of independent interest. Denote by $\Cl$ the  $\GL(3,\mathbb{Z})$ classes of integral forms $F$, which we also take to be primitive. The invariants $D(F)$ and $\kappa(F)$ descend to the class $C$ of $F$. Given $d\neq 0$, there are finitely many classes $C \in \Cl$ with $D(C)=d$. Their number is the class number $h(d)$. The collection of classes with  $D(C)=d$ partition  into genera denoted by $G$ with $D(G)=d$. Their number for a given $d$ is denoted by $g(d)$. It is well known \cite{meyer} that the number of classes in a genus for indefinite ternary forms is small, with the decisive result being that of Eichler \cite{eichler} that the number of classes in a spinor genus is one.

Denote by $\Gen$ the set of all genera of our primitive ternary forms. We show as $X\to \infty$:
\begin{enumerate}[label=\upshape(\Alph*)]\label{ResultA}
\item[{\bf(A)}.] \hfill$ \begin{aligned}[t] \quad\quad\quad \sum_{\substack{G\in \Gen\\ 1\leq D(G)\leq X}} 1= \sum_{1\leq d\leq X}g(d) \sim \frac{57}{4\pi^2}X\log{X},\end{aligned}$\hfill\hfill\mbox{}
\end{enumerate}
and
\begin{enumerate}[label=\upshape(\Alph*)]
\item[{\bf(B)}.] \hfill$ \begin{aligned}[t] \quad\quad\quad \sum_{\substack{C\in \Cl\\ 1\leq D(C)\leq X}} 1= \sum_{1\leq d\leq X}h(d) \sim \frac{57}{4\pi^2}X\log{X}.\end{aligned}$\hfill\hfill\mbox{}
\end{enumerate}

\remark {In Section \ref{clnbr} we show that almost all genera, in the sense of density when ordered by their determinants, consist of a single class. With this, {\bf (B)}\footnote{Very little is known about the similar sum for indefinite binary forms. Hooley \cite{hooley1984} conjectures that it is asymptotic to $\frac{25}{12\pi^2}X(\log{X})^2$, and shows that on average the class number of a genus exceeds one.} is deduced from {\bf (A)} and it is also the source of the asymptotic evaluations of various packet dependent (see Section \ref{packets-def}) sums in Sections \ref{local-rep} and \ref{Mar-sec}. }

What is perhaps surprising and more challenging to prove, is that the $\text{MAR}(X)$ sums are also of the same  size as {\bf (A)} and {(\bf B)}, namely
\begin{enumerate}[label=\upshape(\Alph*)]
\item[{\bf(C)}.] \hfill$ \begin{aligned}[t] \quad\quad\quad \sum_{\substack{C\in \Cl\\ 1\leq D(C)\leq \kappa(C)^3 X}} 1 \sim  \gamma X\log{X}.\end{aligned}$\hfill\hfill\mbox{}
\end{enumerate}
At the end of Section 7, it is shown that
\[\gamma>\frac{57}{4\pi^2}=\frac{19}{20}\frac{\zeta(2)}{\zeta(4)},\]
reflecting the fact that $\kappa(C)$ is often bigger than $1$.

For the proof of Theorem~\ref{IThm2} we make use of Siegel's class invariant $\rho(C)$ (\cite{siegel}). This is an archimedean density associated with $F\in C$ which also takes into account the global infinite group $\text{Aut}_{F}(\mathbb{Z})$, the unit group. More precisely, for $F\in C$, $F$ lies in the six dimensional space $\Sym(\R)$. The map $\psi_F: \text{Mat}(3\times 3,\mathbb{R}) \to\Sym(\R)$ is given by
\[ \psi_F(Z) = Z^{\intercal}FZ.
\]
Let $T$ be a small neighborhood of $F$ in $\Sym(\R)$ and let
\[ Y = \left\{ Z: \psi_F(Z)\in T\right\}.
\]
$\text{Aut}_{F}(\mathbb{Z})$ acts on $Y$ by multiplication on the left and stabilizes $Y$. If $Y_0=Y/\text{Aut}_{F}(\mathbb{Z})$, then
\begin{equation}\label{rho1}
\rho(F)= \lim_{\substack{T\to F}}\frac{\text{vol}(Y_0)}{\text{vol}(T)},
\end{equation}
where the volumes are normalised euclidean volumes in their corresponding spaces. One can also express $\rho(F)$ in terms of $\text{vol}(\text{Aut}_{F}(\mathbb{R})/\text{Aut}_{F}(\mathbb{Z}))$ with suitably normalized Haar measure on $\text{Aut}_{F}(\mathbb{R})$ and $D(F)$ (see \cite{siegelIAS} pg.121). $\rho(F)$ is a class invariant.\footnote{In the case of indefinite binary quadratic forms, $\rho(F)$ is the regulator.}

Siegel \cite{siegel} proves the following as a special case of general such averages (over all classes $C$, not only primitive ones), namely
\begin{equation}\label{siegel-sum}
\sum_{\substack{1\leq D(C)\leq X}} D(C)\rho(C)  \sim \frac{\zeta(2)\zeta(3)}{2}X.
\end{equation}

{These Siegel sums over classes $C$ of regulated weights $D(C)\rho(C)$ arise naturally when counting using orbits and can be dealt with quite generally and in fine detail using Sato-Shintani zeta-functions \cite{Shin}, \cite{IbuSai}.
}

 Let $\text{Cl}^{\text{iso}}$ denote the classes of isotropic primitive integral {ternary forms} with positive determinants. We show that as $X \to \infty$, {the Siegel sum over the thin subset of isotropics satisfies} 
\begin{enumerate}[label=\upshape(\Alph*)]
\item[{\bf(D)}.] \hfill$ \begin{aligned}[t] \quad\quad\quad \sum_{\substack{C\in \Cl^{\text{iso}}\\ 1\leq D(C)\leq X}} D(C)\rho(C) \sim \frac{\zeta(2)\zeta(3)}{2}\varpi\frac{ X}{\sqrt{\log X}},\end{aligned}$\hfill\hfill\mbox{}
\end{enumerate}
where $\varpi$ is the constant from Theorem \ref{IThm2}.

The paper is divided into two parts. The first proving Theorem~\ref{IThm1} concerning anisotropic ternaries and the second Theorem~\ref{IThm2} on isotropic ternary forms. Along the way we develop some tools that are used in both.

The primary difficulty in the $\textup{MAR}(X)$ asymptotics is to get a handle on $\kappa(F)$. This is already apparent for the simplest non-tamely ramified $F$'s of the form $n_px_1^2 -px_2^2+pn_px_3^2$ where $n_p$ is the least positive non quadratic residue modulo a prime $p$. Then for these $F$'s and $p \equiv 1 \pmod{4}$,  $\kappa(F) = n_p$ and these produce ``large'' points in the Markoff spectrum. 

Watson (\cite{watson54,watson1957}) proves the bound $\kappa(F) \ll_{\ve} |D(F)|^{\frac{1}{4}+\ve}$ for primitive $F$ and $\ve>0$ arbitrarily small. In Section \ref{spec}, we elaborate and extend his method, and also incorporate Burgess' bounds \cite{burgess63} for character sums, leading to $\kappa(F) \ll_{\ve} |P(D)|^{\frac{1}{4}}|D|^{\ve}$ (see Section~\ref{Watson}), where ${P(D) \ll \sqrt{D}}$ is the product of the odd primes $p$ such that $p^2 \! \mid \!  D\  (=D(F))$. With this and known bounds for the class numbers $h(D)$, we show  that $\textup{MAR}(X) \ll_{\ve}X^{1+\ve}$ (see Prop.~\ref{thm2}), {which according to {\bf (A)} is nearly the right order of magnitude}.

To go further and obtain sharp bounds and asymptotics, we need to inspect in detail the classes and genera of our forms. The genera can be quite complicated especially at the prime $p=2$, but since we are assuming that our forms are primitive, we are able to give a sufficiently explicit description (see Sec.~\ref{genera}) and use it, {together with the local nature of the genus,} to derive the asymptotic {\bf (A)}.

{ The sums {\bf (B)} and {\bf (C)} are exotic ones and to handle them, as well as {\bf (D)}, we introduce in Section~\ref{packets-def}  the notion of \emph{packets} which  deals with high ramification}. We factor the determinant $D$ as $rs$ where $s$ is odd and squarefree (so tamely ramified), and $r$ is coprime to $s$ and is highly ramified in that if $p \! \mid \!  r$ (with $p\neq 2$), then $p^2 \! \mid \!  r$. The prime $2$, when it occurs, is put into the $r$ factor due to the complexities associated with its invariants. A highly ramified or \emph{root packet} $H$ is a collection of local $\mathbb{Z}_p$ classes of forms for each $p \! \mid \!  r$.
A sketch of the use of packets for an asymptotic count is given in the beginning of Sec.~\ref{packets-def}.

 The set of root packets is denoted by  $\rPKT$ and any $H \in \rPKT$ has associated {local} invariants $D(H)$, $r(H)$ and $K(H)$. Here $K(H)$ is defined in an analogous manner as $\kappa(F)$, namely $K(H)$ is the least absolute value of the integers represented by all the $\mathbb{Z}_p$ forms  in $H$. The integrands defining {\bf(B)}, {\bf(C)} and {\bf(D)} are after normalization, functions $f_X^{{\bf(B)}}(H)$, $f_X^{{\bf(C)}}(H)$ and $f_X^{{\bf(D)}}(H)$ on $\rPKT$, and involve sums over the tame part $s$ with $(s,2r(H))=1$. Thus, to prove the existence and compute the limits, it suffices to show that as $X \to \infty$ the pointwise limit $f(H)$ exists for each $H\in \rPKT$, and to show that the $f_X$'s are dominated by a function in $\mathcal{l}^1(\rPKT)$, so as to apply the dominated convergence theorem.

For ${\bf (B)}$ we show in Section~\ref{clnbr} that for each $H\in \rPKT$ the $X$-limit $f^{{\bf (B)}}(H)$ exists and coincides with the genus limit. This follows by identifying the spinor kernel (see Prop.~\ref{propE1}) and studying the corresponding multiplicative function of $s$ using the robust version of the Landau-Selberg-Delange method formulated and proved in Granville-Koukoulopoulos \cite{GranK}. Specifically, Proposition~\ref{splitting} shows that the main contribution comes from genera $G$ whose class number $h(G)$ is equal to one. Together with the dominating bounds in Proposition~\ref{class-tail}, this leads to the class number asymptotics in ${\bf (B)}$.

The proof of the asymptotics of ${\bf (C)}$ is carried out in Sections~\ref{local-rep} and~\ref{Mar-sec}. It requires a much finer analysis of $\kappa(C)$. The first step is a {crude} but uniform bound for $\kappa(C)$ in terms of its root packet invariant $r(C)=r(D(C))$, namely $\kappa(C) \ll r(C)^\frac{7}{2}$ (see Prop.~\ref{kappa-finite}). With that {and with $H$ fixed (and with it the high ramification) one can exploit that almost all genera have class number one, which renders the corresponding $\kappa(C)=K(C)$ to be local. This allows for the determination of the pointwise limits $f^{{\bf (C)}}(H)$ for $H\in \rPKT$ in terms of the local (albeit complicated) data involving $H$ (see \eqref{MH}) }

{The most challenging step is to find an $\mathcal{l}^1(\rPKT)$  function dominating the $f_X^{{\bf (C)}}$'s. This requires controlling both the local obstructions to representability by the class $C$ with given root packet $H$, as well as possible local to global obstructions. For the latter, we use the sufficient condition for representability of Jones-Watson (see Prop. \ref{propD2}). In order to deal with both issues,}  we use a lower-bound sieve to produce many small (in terms of $r(H)$) integers free from small prime factors, and which are represented by many classes $C$ whose root packet is our given $H$ (see Prop.~\ref{brun}). This coupled with a combinatorial bootstrap (see Prop.~\ref{mechanical}) which pivots on the points produced in Proposition~\ref{brun} leads to an explicit $\mathcal{l}^1(\rPKT)$ dominating function (see Prop.~\ref{martini-tail}), and with it a proof of ${\bf (C)}$, completing Part I.

For Part II, we use homogeneous dynamics and in particular Oh \cite{heeoh} to localize the problem along the lines outlined in \cite{sarnakletter}. The problem treated there is that of the density of integral ternary quadratic $F$'s whose determinants are prime numbers, and it is executed fully in Kotsovolis-Woo \cite{KotWoo} including what is more difficult for that problem, namely the definite forms. In Section~\ref{serre-p} we review the reduction of Theorem~\ref{IThm2} to the asymptotics of ${\bf (D)}$. 

{The various Haar measure normalizations of volumes of quotients associated with classes $C$ coming from homogeneous dynamics and Siegel's $\rho(C)$ are related by a global normalization. We determine its value by comparing Siegel's $\eqref{siegel-sum}$ with a direct count of all forms.}

{To execute ${\bf (D)}$  we use Siegel's mass formula, which asserts that the sum over the isotropic genus of $\rho(C)$ is given by a product of local masses. Conway and Sloane \cite{CSMass} have explicit computations of these masses and are included for our use in Prop. \ref{massprop} and Appendix \ref{local-appendix}. At this point, we could proceed as we do in Section ~\ref{isotropic-siegel-b} to use our packet decomposition and determine the limit in {\bf (D)} as an infinite sum. The infinite product structure for these isotropic genera is connected with their Hasse invariants being equal to one, as well as an extra invariance of the corresponding masses. In fact, collecting these sums of isotropic classes of a given determinant $D$ yields a multiplicative function of $D$  (see Prop.~\ref{mult}). Forming the Euler product associated to the sums {\bf (D)}, as done in Section \ref{comp}, leads to the explicit Theorem \ref{exp-form}. This allows us to deduce the asymptotics of {\bf (D)}, the factorization of $\varpi$ and its interpretation, as well as the finer localization of the asymptotic, all from the Riemann zeta-function!\footnote{The tables in the Appendix computing the Conway-Sloane masses have been used in much of the paper. The fact that $\varpi$ factorizes and agrees with the local $p$-adic probabilities adds confidence in their correctness.} In particular, the reader who is interested only in Theorem \ref{IThm2} may read Part 2 of the paper which, using this treatment, is independent of Part 1.  }

To end, it is interesting to note that unlike the previous works on isotropic counts towards  Theorem~\ref{IThm2} which used sieve methods and in particular lower bound sieves, our proof of Theorem~\ref{IThm2} makes no use of the latter. On the other hand our proof of Theorem~\ref{IThm1} which on its face has nothing to do with prime numbers, makes crucial use of lower bound sieves which allow us to tame $\kappa(C)$.

\begin{acknowledgements}\ 

AGa, AGh and JPW  would like to thank the Institute for Advanced Study Summer Collaborators program for providing a productive research environment and financial support for this collaboration, begun in June, 2025.

AGa gratefully acknowledges support from a Simons Foundation grant No. 964948.

AGh thanks the IAS and Princeton University for making possible visits during parts of the years 2024-26. He  gratefully acknowledges support from a Simons Foundation grant No. 634846,  and additional support from his home department.

PS thanks both G. Margulis and Z. Rudnick for discussions years ago about related problems, and also W. Duke for related discussions.

JPW thanks Princeton University for hosting his sabbatical visit during this work. JPW was supported by the Samsung Science and Technology Foundation under Project Number SSTF-BA2201-03, and by the National Research Foundation of Korea (NRF) grant funded by the Korea government (MSIT) (RS-2025-02293115). JPW also thanks Jihye Oh for helpful comments on the manuscript.

Some of the computation for the tables in the paper were done using  SageMath, the Sage Mathematics Software System.

\end{acknowledgements}


\section{{\bf Part 1} : Spectrum of anisotropic forms}\label{spec}\

Following the discussion in the Introduction, we consider the set $\mathfrak{F}$ of primitive non-singular indefinite integral  ternary  quadratic forms with positive determinant.

Let $F({\bf x}) = f_{11}x_1^2 +f_{22}x_2^2 + f_{33}x_3^2 + 2f_{12}x_1x_2 + 2f_{13}x_1x_3 + 2f_{23}x_2x_3$ be such a form with the corresponding integral matrix  $A_F$ and determinant $D(F)=\det{A_F}$. It is isotropic if it represents zero non-trivially over the integers; otherwise it is anisotropic. The minimum of its absolute values  is $\kappa(F)$ given in \eqref{kappaF}. If $F$ is a diagonal form, we will write $F=<f_{11},f_{22},f_{33}>$.

 We define the following terms:
 \begin{definition}\label{pre-def}\ 
\begin{enumerate}[label=\upshape\arabic*.\ ]
\item $C(F)$ is the class of $F$ under the linear action of $\GL(3,\mathbb{Z})$. If $F$ is a primitive form, then  all forms in $C(F)$ are primitive,  and so $\Cl$ will denote the set of all such {\it primitive classes}. All the invariants of $F$ above descend to $\Cl$. So we can speak of {\it isotropic {\upshape (}anisotropic{\upshape)} classes}. The number of $C\in \Cl$ for a fixed ${\it D}(C) =d$ is finite and we denote that number by $h(d)$. For a class $C\in \Cl$,  $\mu(C)$ is defined in \eqref{mu}; the set of these numbers counted with multiplicity form the  {\it Markoff spectrum}.\footnote{That this set corresponds to the same set when allowing all real indefinite forms as in Theorem~\ref{IThm1} follows from the Oppenheim Conjecture proved by Margulis \cite{MARG1987,Mar-Opp}.}  We denote by $\Cl(d)$ the set of classes $C\in \Cl$ with ${\it D}(C)=d$.
\item For each prime $p\leq \infty$, let $\Gen_p$ be the set of $\GL(3,\Z_p)$-equivalence classes of elements in $\mathfrak F$, with $\Z_{\infty}=\R$. The genus ${\Ge}(F)$ of $F\in\mathfrak F$ is defined as the set of all forms $F'\in\mathfrak F$ having the same image as $F$ under the diagonal map
\[\Ge:\mathfrak F\to\prod_{p\leq\infty}\Gen_p;\]   two forms are in the same genus if they are $\mathbb{Z}_p$-equivalent for all $p$. We shall often identify a genus with the corresponding element of $\prod_{p\leq\infty}\Gen_p$. The genus of a class $C\in \text{Cl} $ is $\Ge(C)=\Ge(F)$ for any representative $F$ of $C$.  The set of genera of forms in $\mathfrak F$ will be denoted by  $\Gen=\Ge(\mathfrak F)\subseteq\prod_{p\leq\infty}\Gen_p$. Given $G\in\Gen$, the number of classes whose genus is $G$ is denoted by $h(G)$.  We denote the genera of classes of determinant $d$ by $\Gen(d)$ and the number of such  by $g(d)$.
\item ${\it SG}(F)$ is the spinor genus of $F$. By a theorem of Eichler, the spinor genus of an indefinite quadratic form $F\in\mathfrak F$ consists of a single class. See Sec.~\ref{spinor} for a detailed description.
\item If $\mathfrak{X}$ is a set containing classes, we will let $\mathfrak{X}^{(1)}$  denote the subset of all classes in $\mathfrak{X}$ whose genus has exactly one class, and use $\mathfrak{X}^{(2+)}$ to denote those with at least two classes. We will use similar superscripts for the counting functions (see Sec. \ref{Def-5.1} for an example).
\end{enumerate}
\end{definition}

\begin{notation} \ 
\begin{enumerate}[label=\upshape\arabic*.\ ] 
\item $\omega(n)$ denotes the number of distinct prime factors of $n$; while $\tau(n)$ is the number of positive divisors of $n$.
\item For $0\neq d\in\Z$ and integer $k\geq1$, let $\omega_k(d)=\{\text{$p$ prime}:p^k \! \mid \!  d\}$. Note $\omega_1=\omega$.
\item For $d\neq 0$ and $p<\infty$ a prime, $v_p(d)$ is the $p$-adic valuation of $d$, so $p^{v_p(d)}\parallel d$.
\item $\log^{+} |x| = \max(1, \log |x|)$.
\end{enumerate}
Also see Appendix \ref{notations} for a table of frequently used notations that are introduced later in the paper. \qed
\end{notation}
 

We will first consider preliminary estimates for $\textrm{\textup{MAR}}(X)$ defined in $\eqref{mu3}$, with the goal to derive {\bf (C)}. For this we need good estimates for $\kappa(F)$, which we consider next.


\subsection{Watson's papers   revisited}\label{Watson}\ 

As noted in the Introduction,  Watson (in \cite{watson54} and \cite{watson1957}) gave the first bounds for the key invariant $\kappa(F)$. In this subsection we review and refine his method, and also incorporate Burgess's bounds for character sums which are crucial in our work.

\begin{definition}\label{def2}\ \\
	For any $d\neq 0$, let $P_d \leq \sqrt{|d|}$ be the largest positive squarefree odd integer such that   $p \! \mid \! P_d$ implies $p^2 \! \mid \! d$. Note that $P_d =1$ if and only if $d=2^{v_2(d)}d_0$ where $d_0$ is squarefree.
\end{definition} 

\begin{theorem}\label{thmD1} 
	Let $f =\sum_{1\leq i\leq j\leq 3} f_{ij}x_ix_j$ be a primitive, non-singular and indefinite quadratic form  over the integers, with discriminant $d$. Let $n_f^{+} $ and $n_f^{-} $  denote the  least positive value and largest negative value  attained by $f$, respectively.  Then,
    \[\kappa(f) \leq |n_f^{\pm}| \ll_{\ve} P_{d}^{\theta}|d|^{\eps},\]
    for any $\eps >0 $ and  $d$ sufficiently large. Here 
     (i). $\theta = \frac{1}{2}$ using Polya-Vinogradov bounds, (ii). $\theta = \frac{1}{4}$ using Burgess' bounds,  and (iii). $\theta=0$ on GRH (or GLH).
\end{theorem}

\begin{remark}\label{remD1}
	Watson's proof and statement in  \cite{watson1957} is given for positive values of $f$, but the argument also holds for the negative values. We give the details for positive values.
\end{remark}
In what follows (only in this subsection), we consider forms in Watson's sense. This is relevant since he constructs odd integers $n$ satisfying $f({\bf x})=n$ with $n \ll |d|^{\frac{1}{4}+\eps}$ (using Polya-Vinogradov). If $F$ is a primitive form in the Gaussian definition that we use in this paper, then it is possible for the value set to be only even integers. However, in Watson's definition, such a $F$ must necessarily be imprimitive, and so is excluded from his theorem. However a $F$ such as this is necessarily of the form $2f$ with $f$ primitive in Watson's sense, and so the conclusion of his theorem still holds for $F$. 

A significant result used in Watson's proof is the following:

\begin{proposition}\label{propD2}\upn{({Therem~2 of Jones-Watson \cite{JW1956}})}\ \\
	Suppose $n$ is represented by at least one form  in the genus  $G(f)$. With $d$ the discriminant of $f$, suppose there exists $q$  an odd prime such that $q\nmid d$ but $q\parallel n$. Then $n$ is represented by all forms in  $G(f)$.
\end{proposition}
\begin{proof} We follow Watson \cite{watson1957}. Write $dn=n_1n_2^2$, where $n_1$ is squarefree. If $n$ is represented by at least one form  in the genus  $G(f)$ but not by all, then by Thm.~2\;(i)\&(ii) of \cite{JW1956}, $n_1 >1$ and  $n_1 \! \mid \! d_0$. Here, $d_0  \! \mid \! d$ is defined in that paper. But $q\parallel n$ and $q\nmid d$ implies $q \! \mid \!n_1$ so that $q \! \mid \! d$, a contradiction. 
\end{proof}

\vspace{6pt}

The proof of Theorem~\ref{thmD1} to construct "small" $n$'s such that $f({\bf x})=n$ is solvable over $\mathbb{Z}$ goes as follows: first determine conditions on $n$ such that $f({\bf x})=n$ is solvable over $\mathbb{Z}_p$ for all $p\geq 2$. Then, Prop.~\ref{propA1} shows that $n$ is represented by some form in $G(f)$. Restricting the $n$'s to satisfy Prop.~\ref{propD2} then shows that such $n$'s satisfy $f({\bf x})=n$. It then suffices to count the number of $n$'s satisfying the local and divisibility conditions.

If  $p\nmid 2d$, Prop.~\ref{propA3} implies $f$ is $\mathbb{Z}_p$-equivalent to a unit diagonal form of type $g=a_1x_1^2+a_2x_2^2+a_3x_3^2$ with $p$-adic units $a_j$. If $p\nmid n$, take $x_3=0$, and otherwise take $x_3=1$ so that Prop.~\ref{propA3}(1) implies $f({\bf x})=n$ is solvable over $\mathbb{Z}_p$ for all $n\in \mathbb{Z}$ with $p\nmid 2d$.

If $p\! \mid \! d$ with $p$ odd, since $f$ is primitive,   Prop.~\ref{propA3}(4) implies  $f$ is  $\mathbb{Z}_p$-equivalent to $h_1= a_1x_1^2 +p^{t_{11}}a_2x_2^2 + p^{t_{12}}a_3x_3^2$ or $h_2 = a_1x_1^2 + p^{t_2}(x_2^2 + a_3x_3^2)$ or $h_3 = (x_1^2 + a_2x_2^2) + p^{t_3}a_3x_3^2$, with units $a_j$. The first two cases occur only if  $p^{2} \! \mid \!  d$ while in the last, $p^{t_3}\parallel d$. If $p \! \mid \! n$, the equation  $f=n$ may not be  solvable over $\mathbb{Z}_p$, so we choose $n$ to satisfy $(n,2d)=1$. Then if $p\parallel d$, only the case $h_3 \equiv n$ occurs, and that is solvable with $x_3=0$. It remains to consider the cases $p^2 \! \mid \! d$ with $p>2$ and the case $p=2$, for $n$'s satisfying $(n,2d)=1$.

We will require
\begin{lemma}\label{lemD1}\ 
	If $f$ is primitive, there exists an integer $M$ with $(M,2d)=1$ such that $f = M$ is solvable over $\mathbb{Z}$. 
\end{lemma}
\begin{proof} This is a consequence of  \cite{Cassels} Thm.~8.1. and the chinese remainder theorem as follows.
	
	For $p$ odd, since $f$ is primitive, it is $\mathbb{Z}_p$ equivalent to a form of the type $a_1x_1^2 + p^t\phi(x_2,x_3)$ with $p\nmid a_1$. So $f$ represents $a_i$ over $\mathbb{Z}_{p_i}$ for each $p_i \! \mid \! d$, odd.
	For $p=2$,  $f$ is equivalent to a form $2^{t_1}a_1x_1^2 + 2^{t_2}\phi(x_2,x_3)$ with $\phi$ an odd form, $a_1$ odd and  $t_1t_2=0$. So $f$ represents an odd integer, say $b$ modulo 8. Then, there exists an integer $M_0$ such that $M_0 \equiv a_i \Mod{p_i}$ for all odd $p_i  \! \mid \! d$ and $M_0 \equiv b \Mod{8}$, so that $(M_0,2d)=1$. Thus $M_0$ is represented by $f$ in $\mathbb{Z}_p$ for all $p\geq 2$. Theorem 8.1 of  \cite{Cassels} then says that $f$ represents some integer $M=M_0u^2$ for some $u$ coprime to $2d$, as required.\end{proof}

\vspace{6pt}

Let $P_d$ be as in Def.~\ref{def2}. For any $p \! \mid \! P_d$, choose $n$ to satisfy $\genfrac(){}{0}{n}{p}= \genfrac(){}{0}{M}{p} $. Then if $f({\bf y})=M$, we have $f(t{\bf y})=t^2M = n$ is solvable over $\mathbb{Z}_p$. Finally, choose $n$ so that $n \equiv M \Mod{8}$, so that $f({\bf x})=n$ is solvable over $\mathbb{Z}_2$. 

We thus have to count the number of positive $n$'s satisfying (i) $(n,2d)=1$, (ii) $n \equiv M \Mod{8}$ and (iii) $\genfrac(){}{0}{n}{p}= \genfrac(){}{0}{M}{p} $ for all $p \! \mid \! P_d$. Such $n$ satisfy the equation $f({\bf x})=n$ over $\mathbb{Z}_p$ for any $p\geq 2$, and so by  Prop.~\ref{propA1} is represented over $\mathbb{Z}$ by some form in $G(f)$. Choose any prime $3\leq q \ll d^{\eps}$ such that $q \nmid d$ and restrict the $n$'s further to satisfy $q\parallel n$. By   Prop.~\ref{propD2}, such $n$ are represented by all forms in $G(f)$. Putting $n=qn_1$, it suffices to count the number of $n_1$'s satisfying (i) $(n_1,2dq)=1$, (ii) $n_1 \equiv M_1 \Mod{8}$ for some fixed odd $M_1$ and (iii) $\genfrac(){}{0}{n_1}{p}= \genfrac(){}{0}{Mq}{p} $ for all $p \! \mid \! P_d$.

The number of such $n$'s in the interval $1\leq n \leq \xi$ is
\begin{equation}\label{D.0.1}
\sum_{\substack{1\leq n \leq \xi\\ (n,2dq)=1}}\prod_{p \! \mid \! P_d}  \frac{1}{2}\left(1+ \genfrac(){}{0}{nMq}{p}\right) \prod_{\sigma=\pm 1}\frac{1}{2}\left(1+ \genfrac(){}{0}{2\sigma}{nM_1}\right).
\end{equation}
The main term comes from the principal character in the sum on the right, and is 
\begin{equation}\label{D.0.2}
2^{-\omega(P_d)-2}\sum_{\substack{1\leq n \leq \xi \\ (n,2dq)=1}} 1 = 2^{-\omega(P_d)-2}\frac{\phi(2dq)}{2dq} \xi + O(\tau(d)).
\end{equation}
The remainder is a sum over the non-principal real characters modulo $8P_d$. If $\chi$ denotes a generic such character, the remainder is bounded by 
\begin{equation}\label{D.0.3}
\max_{\substack{\chi \Mod{8P_d}}}\ |\sum_{\substack{1\leq n \leq \xi \\ (n,2dq)=1}} \chi(n)| \ll \max_{\substack{\chi\Mod{8P_d}\\ \delta  \mid  2dq}}\ d^{\eps}|\sum_{1\leq n \leq \frac{\xi}{\delta} } \chi(n)| \ll P_d^{\alpha}\xi^{\beta}d^{\eps},
\end{equation}
for suitable constants $\alpha$ and $\beta <1$. Note that $P_d^2 \! \mid \! d$ so that $P_d\ll d^\frac{1}{2}$.

We require

\begin{proposition}\label{propC1}\ \\ Let $\chi$ be a non-principal character of modulus $q>1$ and let \[S_{\chi}(N)= \sum_{M<n\leq M+N} \chi(n),\] for $M\geq 0$.
	\begin{enumerate}[label=\upshape\arabic*.\ ]
		\item \upn{(Polya-Vinogradov):} \quad $S_{\chi}(N) \ll \sqrt{q}\log{q}$.
		\item On GRH, $S_{\chi}(N) \ll \sqrt{N}q^{\ve}$.
		\item \upn{(Burgess):} \quad Let $q=kl$, where $l\geq 1$ is the maximal cube-free factor. Let $r\geq 1$ be any integer. Then
		\[ S_{\chi}(N) \ll_{r} N^{1-\frac{1}{r}}k^{\frac{1}{r}}l^{\frac{r+1}{4r^2}}(qN)^{\ve}.\]
	\end{enumerate}
	All implicit constants depend at most on $\ve>0$.
\end{proposition}
\begin{proof} See \cite{IwaK} Chapter 12 for details. The extension to non-primitive characters is straightforward, while the statement for the Burgess bound without the cube-free restriction is stated in equation (12.56).\end{proof}

\vspace{6pt}

By Prop.~\ref{propC1}, the Polya-Vinogradov estimate gives $\alpha= \frac{1}{2}$ and $\beta =0$, allowing us to choose $\xi$ of size $P_d^{\frac{1}{2}}d^{\eps}$. We may apply Burgess' estimate with $k=8$ in Prop.~\ref{propC1} 
giving $\alpha= \frac{r+1}{4r^2}$ and $\beta =1-\frac{1}{r}$ allowing one to choose $\xi$ of size $P_d^{\frac{r+1}{4r}}d^{\eps}$. We then choose $r$ arbitrarily large to get $\xi$ of size $P_d^{\frac{1}{4}}d^{\eps}$. Finally,  GRH (using Prop.~\ref{propC1}) gives $\alpha= 0$ and $\beta =\frac{1}{2}$ giving $\xi$ of size $d^{\eps}$. In general, one may choose $\theta = \frac{\alpha}{1-\beta}$, with $\xi = P_d^{\theta}d^{\eps}$.

\begin{remark}\label{remD2}\ 
	One can remove the $d^{\ve}$ in the estimates above as follows. The estimate in \eqref{D.0.3}, with a sum over divisors,  can be replaced with 
	\[
	\max_{\substack{\chi\!\Mod{8P_d}\\ \delta|2dq}}\ \tau(d)|\sum_{1\leq n \leq \frac{\xi}{\delta} } \chi(n)| \ll P_d^{\alpha}\xi^{\beta}\tau(d),
	\]
	say.
	The main term in \eqref{D.0.2} is
	\[ \gg 2^{-\omega(P_d)}(\log\log{d})^{-1}\xi - C\tau(d).
	\]
	Thus, it suffices to choose $\xi$ large enough to satisfy $\xi^{1-\beta} \gg L(d)P_d^{\alpha}$, with $L(d)=2^{\omega(P_d)}\tau(d)(\log\log{d})$. 
	
	Applying GRH (Prop.~\ref{propC1}) with $\beta = \frac{1}{2}$ and $\alpha = \frac{1}{2}\ve$ gives a non-empty count when  $\xi \approx L(d)^2P_d^{\ve}$.
	
	Applying the Burgess bound, with $\alpha= \frac{r+1}{4r^2} +\frac{\ve^2}{4} $, $\beta =1-\frac{1}{r}+\frac{\ve^2}{4}$ and $r = \frac{2}{\ve}$, allows us to choose $\xi \approx L(d)^{4\ve^{-1}}P_d^{\frac{1}{4}+ \ve}$ . 
	
	Let \[\kappa(d) = \max_{\substack{F\in\mathfrak F\\ D(F)=d}} \kappa(F).\]Using these bounds, and averaging over $d$ in some interval of length $X$, one gets crude upperbounds for the average of $\kappa(d)$  of the type (i). $O\!\left((\log{X})^{A}\right)$ to a sufficiently high power $A$, or (ii)  $ O\!\left((\log{X})^{3+\ve}\right)$, on GRH. However, much better bounds are obtained if one uses the Polya-Vinogradov inequality. This is because on average, $P_d$ is small and then the term $L(d)$ plays the dominant role. So taking $\alpha = \half +\ve$ and $\beta=0$ gives $\kappa(d) \ll L(d)P_d^{\half +\ve}$. Then
	\[
	\sum_{1\leq d\leq X} \kappa(d)\ll \log\log{X} \sum_{1\leq P^2r\leq X} 2^{\omega(P)}\tau(P^2r)P^{\half +\ve},
	\] 
	so that with $\tau(P^2r) \leq 3^{\omega(P)}\tau(r)$, since $P$ is squarefree, and summing over $r$ first gives the bound
	\[
	\sum_{1\leq d\leq X} \kappa(d) \ll \log\log{X} \sum_{1\leq P\leq \sqrt{X}} 6^{\omega(P)}P^{\half +\ve} \left(\frac{X}{P^2}\log\left(\frac{X}{P^2}\right)\right),
	\]
	giving us
	\begin{lemma}\label{avg-kappa}
		\[ \sum_{1\leq d\leq X} \kappa(d) \ll X(\log X)\log\log{X}. \]
	\end{lemma}
	It does not appear easy to remove the $\tau(d)$ term in the analysis above since complications arise if we remove the coprimality condition $(n,d)=1$ in \eqref{D.0.3}
\end{remark}

Estimates for ``large'' values of $\kappa(F)$'s are given in the following

\begin{lemma}\label{prop1}\ \\ 
\indent Let $p \equiv 1 \Mod{4}$ be a prime, and let $n_p$ be the least positive quadratic non-residue modulo $p$. Then, there is $F \in \mathfrak{F}$ of determinant $n_p^2p^2$ with $\kappa(F) = n_p$. In particular, the numbers $\frac{n_p}{p^2}$ form a subset of the Markoff spectrum.
\end{lemma}
\begin{proof} The indefinite form  $F = n_px_1^2 - p(x_2^2 -n_px_3^2)$ is anisotropic, as may be verified modulo $p$. The quadratic nature of residues modulo $p$ is symmetric about the origin. It suffices to verify that if $1\leq m <n_p$, then $F=m$ has no solutions.  But every such $m$ is a quadratic residue and the congruence $F\equiv m$ is unsolvable modulo $p$.
\end{proof}

It is known that $n_p \gg \log{p}$ for infinitely many $p \equiv 1 \Mod{8}$, and $n_p \ll (\log{p})^2$ on GRH (\cite{ankeny}). It then follows that 

\begin{corollary}\label{prop1cor}\ \\
	\indent There are infinitely many $d$'s for which there are   classes  $C^{(d)}$ with $D(C^{(d)})=d$ and $\kappa(C^{(d)}) \gg \log{d}$. Moreover, on GRH these classes also have the upper-bound $\kappa(C^{(d)}) \ll (\log{d})^2$.
\end{corollary}
\qquad

\subsection{\texorpdfstring{Upper and lower bounds for $\text{MAR}(X)$}{Upper and lower bounds for MAR(X)}}\

Following \eqref{mu3}, we write

\begin{equation}
\textrm{\textup{MAR}}(X) = \sum_{\substack{C\in \Cl \\ 0<D(C)\leq \kappa(C)^3 X}}1\quad \text{and} \quad   \textrm{\textup{MAR}}^\#(X) = \sum_{\substack{C\in \Cl \\ {\it D}(C)\leq \kappa(C)^3 X\\ {\it D}(C)\  \text{squarefree}}} 1 ,
\end{equation}
the goal being to obtain estimates for large $X$.

For any fixed pair $t\geq 0$ and $d>0$, define
\begin{equation}
\lambda(d,t) = \# \{C\in \Cl(d):  \kappa(C) =t \},
\end{equation}
and
\begin{equation}
\mathcal{S}(X,t)=\sum_{\substack{d\leq X}} \lambda(d,t) .\footnote{Note that $\mathcal{S}(X,0)$ counts the number of isotropic classes $C$ with ${\it D(C})\leq X$, which we consider in Sec.~\ref{serre-p}.}
\end{equation}

Then,
\begin{equation}\label{Mar-sum}
\textrm{\textup{MAR}}(X) = \sum_{t=1}^{\infty} \mathcal{S}(t^3X,t).
\end{equation}

The sum over $t$ is finite, with a more precise statement derived from Thm.~\ref{thmD1}. We have 

\begin{proposition}[Corollary to Thm.~\ref{thmD1}]\label{propWatson}\ \\
\indent Let $\eps >0$ be arbitrarily small. There is a constant $C(\eps) >0$ such that as $d\to\infty$
\begin{enumerate}[label=\upshape(\roman*)]
\item[\upshape{1.\ }]  $\lambda(d,t)=0$ if   $t \geq C_{\eps}P_d^{\frac{1}{4}}d^{ \eps}$ ,
\end{enumerate}
 and 
 \begin{enumerate}[label=\upshape(\roman*)]
\item[\upshape{2.\ }]  on GRH for any $d$, or unconditionally for bounded $P_d$, we have $\lambda(d,t)=0$ if $t \geq C_{\eps}d^{ \eps}$.
\end{enumerate}
\end{proposition}

\begin{remark}\label{rem1} In general, following the analysis in Sec.~\ref{Watson}, one has $\lambda(d,t) =0 $ if $t \geq C_{\eps}P_d^{\theta}d^{ \eps}$ where $\theta = \frac{\alpha}{1-\beta}$, with the constants obtained from bounds  for character sums of real characters of conductor $8P_d$.
\end{remark}

\begin{remark}\label{rem2} One sees from Corl.~\ref{prop1cor} that one may not improve the lower bound in Prop.~\ref{propWatson}(1) to below $\log d$. Moreover, in an interval $1\leq d\leq D$, the number of $d$'s with $P_d \gg d^{\ve}$ is at  most $O(D^{1-2\ve})$,  so that  $\lambda(d,t)=0$ for $t \gg d^{2\ve}$ for almost all $d$ in that interval. 
\end{remark}

We now can get an upper-bound for $\textup{MAR}(X) $ by estimating the finite sum in $\eqref{Mar-sum}$.  For this, we will require  estimates for the number of genera and classes of given determinant.

 \begin{proposition}\label{thm1}\  
 For any $\eps>0$, we have \[g(d) \leq h(d) \ll_{\ve}  d^{\ve}.\]
 \end{proposition}
 
\begin{proof}
The genera are determined by the $\mathbb{Z}_p$-inequivalent  forms for each prime $p$. Given the determinant $d\neq 0$, there is only one such class for each odd $p\nmid d$ (by Prop.~\ref{propA3}(3)). By Prop.~\ref{propA3}(4) and Props.~\ref{propB1} and~\ref{propB2}, there is an absolute constant $A>0$ such that the number of such classes is bounded by $Av_p(d)$, where $p^{v_p(d)}\parallel d$. This is because each such form, being primitive, is of the type $h_1+p^{e_2}h_2+p^{e_3}h_3$, with a bounded number $A$ of choices of forms $h_i$, and with $n_2e_2+n_3e_3 = v_p(d)$, with $1\leq n_i \leq 3$ the dimension of $h_i$ and $1\leq e_2< e_3$. Thus, the number of genera $g(d) \ll \prod_{p|d} Av_p(d) \ll \tau(d)^{1+\log_2A} \ll d^{\ve}$ (using an upper-bound for the divisor function). Trivially, we see that the average value of $g(d)$ is bounded above by a power of $\log{d}$. 

 The number of spinor genera in each genus is  $O(2^{\omega(d)}) = O(d^{\ve})$ (see for example \cite{JW1956}). Since there is only one class in each indefinite spinor genus \cite{eichler}, we see that the class number is  $  O(d^{2\eps})$.
 
 In Sec.~\ref{genera} and~\ref{clnbr}, a detailed analysis gives the asymptotics for the sums of $g(d)$ and $h(d)$ respectively.
 
\end{proof}
Using the estimates above from Prop.~\ref{thm1} and Prop.~\ref{propWatson}, we have the upper-bound 

\begin{proposition}\label{thm2} For any $\ve>0$
\[\textrm{\textup{MAR}}(X) \ll_{\ve} X^{1+ \eps},\] as $X \to \infty$. 
\end{proposition}

Next, we note that a  lower-bound for $\text{MAR}(X)$ can be obtained by restricting the determinants to be squarefree, giving us  $\text{MAR}^\#(X)$. Such a sum is particularly simple to analyse for two reasons: as we show $\kappa(F)$ is bounded by $2$, and all indefinite spinor genera contain only one equivalence class of forms (\cite{eichler}).  The following gives an asymptotic formula for $\text{MAR}^\#(X)$, with the constant determined by a detailed analysis of local classes of forms (using  Appendices \ref{prelim} and \ref{2classes}):

\begin{proposition}\label{Mar-sf}\ 
 As $X \to\infty$
\[ 
\text{\textup{MAR}}(X)\  \geq\  \text{\textup{MAR}}^\#(X) \sim  \frac{7}{4}cX\log{X} +O(X),
\]
where  $c = \prod_p (1+\frac{2}{p})(1-\frac{1}{p})^2 \approx 0.2867...$\ and  $\frac{7}{4}c \approx 0.502$.
\end{proposition}

\begin{remark}
The proof of the asymptotic formula in Theorem~\ref{IThm1} (proved in Secs. 4-7) utilises the notion of  packets and  root packets.  The squarefree case in Prop.~\ref{Mar-sf} is (essentially) the pointwise limit for the case of trivial root packets.
\end{remark}
\vspace{10pt}
\begin{proof}[Proof of Proposition~\ref{thm2}] \ 

Using the terminology in Remark \ref{rem1} above, the conditions on $d$ and $t$ in (1.0.4) give us  $P_d \ll X^{\frac{1}{2-3\theta}+\eps}$ and $d \ll X^{\frac{2}{2-3\theta}+\eps}$. Using Prop.~\ref{thm1}, write  (1.0.4) as $\textrm{\textup{MAR}}(X) \ll X^{\eps}\textrm{\textup{MAR}}_1$ where

\[\begin{aligned}
\textrm{\textup{MAR}}_1 = \sum_{\substack{1\leq d\leq t^3X\\ 1\leq t \ll P_d^{\theta}X^{ \eps}}} 1  &\ll \sum_{1\leq P  \ll  X^{\frac{1}{2-3\theta}+\eps}}\sum_{\substack{1\leq d\leq t^3X\\ 1\leq t\ll P^{\theta}X^{ \eps}\\P^2|d}} 1 \\ & \ll \sum_{1\leq P\ll  X^{\frac{1}{2-3\theta}+\eps}}\sum_{\substack{ 1\leq t\ll P^{\theta}X^{ \eps}}}\left(\frac{t^3X}{P^2} +1\right) ,\end{aligned}\]

giving the estimate
\[
\textrm{\textup{MAR}}_1 \ll X^{1+\eps} \sum_{1\leq P\ll  X^{\frac{1}{2-3\theta}+\eps}} P^{4\theta -2}\quad + \ \  X^{\frac{1+\theta}{2-3\theta}+\eps}\ .
\]
 If $0\leq \theta \leq\frac{1}{4}+{\eps}$, then $\textrm{\textup{MAR}}_1 \ll X^{1+\eps} $ and so $\textrm{\textup{MAR}}(X) \ll X^{1+ \eps}$ on taking $\theta = \frac{1}{4}$ from Prop.~\ref{propWatson}.

\end{proof}

 \begin{proof}[Proof of Proposition~\ref{Mar-sf}]\ 

Suppose $F$ is a given ternary form over $\mathbb{Z}$ with ${\it D}(F)=d \neq 0$ and $d$ is squarefree. 
We have by Prop.\,~\ref{propA3}

\begin{enumerate}[label=\upshape\arabic*.,itemsep=0.3cm]
\item If  $p>2$ and $p\nmid d$, then  $F$ is $\Zp$-equivalent to $<1,1,d>$, giving one $\Zp$-equivalent class that represents any $l\in \Zp$. All such forms satisfy $c_p(F)=1$.

\item If $p>2$ and $p \! \mid \!  d$, we must have $e_1=0$, $m_1=2$, $e_2=1$ and $m_2=1$. Then, $F$ is equivalent to forms of the type $x_1^2 + \eta_1x_2^2 + p\eta_2x_3^2$, with $\eta_j = 1$ or $r$, where $r$ is a fixed quadratic non-residue modulo $p$, and with $\eta_1\eta_2 \equiv \frac{d}{p}u^2 \Mod{p} $ for some unit $u$. Thus, there are two equivalence classes giving forms $F_1=<1,1,d>$ and $F_2=<1,r,rd>$. One checks that $c_p(F_1)=\left(\frac{-1}{p}\right)$  and $c_p(F_2)= -\left(\frac{-1}{p}\right)$. Both $F_1$ and $F_2$ represent all $l\in \mathbb{Z}$ with $p\nmid l$; in particular, they both represent $\pm 1$ and $\pm2$.
\end{enumerate}

For $p=2$, we use Props.\,~\ref{propB1}, ~\ref{propB2} and the tables in Appendix~\ref{local-appendix}:
\begin{enumerate}[label=\upshape\arabic*.,itemsep=0.3cm]
\item Suppose $d$ is odd. Then, $F$ is $\mathbb{Z}_2$-equivalent to a form with matrix $\{\mathcal{U}_1\}$ of dimension 3. We have the forms given in Prop.\,~\ref{propB2}\,(3) and have exactly two $\mathbb{Z}_2$-equivalence classes for each congruence class $d \Mod{8}$. Denoting these forms by $G_{i,j}$, with $i\in \{1,2\}$ and $j\in \{1,3,5,7\}$, one has the data given in the form   ``$G_{i,j}:\ \{c_2(G_{i,j});\ \text{Minimal}\ \Z-\text{value attained}\}$'' :

{\small \begin{align*}G_{1,1}=&<1,1,1>:\ \{-1; +1\}, \quad G_{2,1}=<1,3,3>:\ \{+1; \pm 1\},\quad  d\equiv 1 \Mod{8};\\
 G_{1,5}=&<1,1,5>:\ \{-1;\pm 1\}, \quad G_{2,5}=<1,3,7>:\ \{+1; \pm 1\},\quad d\equiv 5 \Mod{8};\\
 G_{1,3}=&<1,1,3>:\ \{+1; \pm 1\}, \quad G_{2,3}=<3,3,3>:\ \{-1; \pm 1\},\quad d\equiv 3 \Mod{8};\\
 G_{1,7}=&<1,1,7>:\ \{+1; \pm 1\}, \quad G_{2,7}=<3,3,7>:\ \{-1; -1\},\quad d\equiv 7 \Mod{8}.
\end{align*}}

 All these forms represent $1$ or $-1$ over $\mathbb{Z}_2$.

\item Suppose $2\!\parallel \! d$. Now $F$ is $\mathbb{Z}_2$-equivalent to a Type-1 form with matrix $\{\mathcal{U}_1,2\mathcal{U}_2\}$ of dimensions two and one respectively, or a Type-2 form with matrix $\{\mathcal{V_j},2\mathcal{U}_1,\}$, with $\dim(\mathcal{U}_1)=1$. 
 
 \begin{enumerate}[itemsep=0.3cm,label=\upshape\alph*.]
\item Type-1 : we have forms obtained from Prop.\,~\ref{propB2}\,(i. and ii.), giving us the six diagonal forms $a_1x_1^2+a_2x_2^2+2a_3x_3^2$ with  coefficients $(a_1,a_2,a_3)$ in the set 
\[\quad\quad\quad \{(1,1,\frac{d}{2}),\ (3,3,\frac{d}{2}),\ (1,5,\frac{5d}{2}),\ (3,7,\frac{5d}{2}),\ (1,3,\frac{3d}{2})\},\ (1,7,\frac{7d}{2})\] 
modulo 8. Denoting these forms by $H_{1,j} = <a_1,a_2,da_3>$ with appropriate values for the $a_i$, and with $1\leq j\leq 6$ respectively, we see that they all represent both $\pm 1$ over $\mathbb{Z}_2$. Then, by using the methods of \cite{Conway1993} with oddity fusion and sign walking, we see that there are two $\mathbb{Z}_2$- classes, with the classes given below: 
\begin{align*} c_2 = -1:& \quad H_{1,1} \approx H_{1,4} \approx H_{1,5} \quad \text{if}\quad d\equiv 2 \pmod{8},\\
c_2 = +1:& \quad H_{1,2} \approx H_{1,3} \approx H_{1,6} \quad \text{if}\quad d\equiv 2 \pmod{8}, \\
c_2 = +1:& \quad H_{1,1} \approx H_{1,4} \approx H_{1,6} \quad \text{if}\quad d\equiv 6 \pmod{8}, \\
c_2 = -1:& \quad H_{1,2} \approx H_{1,3} \approx H_{1,5} \quad \text{if}\quad d\equiv 6 \pmod{8}.
\end{align*}
So for each congruence class $d = \pm 2 \pmod{8}$, there are exactly two classes taking both the values $c_2 = \pm 1$, and both classes representing $\pm 1$ over $\mathbb{Z}_2$. 

\item Type-2 : Let $V_1(x,y)= 2xy$ and $V_2(x,y)= 2(x^2+xy+y^2)$, (see Prop.~\ref{propB1}). We obtain forms $H_{2,1}= V_1(x_1,x_2)+ 2s_1x_3^2$ with $s_1 \equiv 7\frac{d}{2} \Mod{8}$ and $H_{2,2}= V_2(x_1,x_2)+ 2s_2x_3^2$ with $s_2 \equiv 3\frac{d}{2} \Mod{8}$. They are equivalent and isotropic. Moreover, these forms  represent $2$ but not $\pm 1$ over $\mathbb{Z}_2$.\qed
\end{enumerate}
\end{enumerate}

We use Props.~\ref{jonesT29} and~\ref{jonesT46} to create indefinite forms over $\mathbb{Z}$ as follows. We choose the Hasse invariant data  with $c_{\infty}=1$ and index $i=1$ and treat the $d$ even and odd cases separately. We put $R=\omega(d)$.

For $d$ odd, we isolate one prime $q$, and choose the other $c_p$ values for odd $p \! \mid \! d$ randomly, giving $2^{R-1}$ possibilities, all representing both $\pm1$. We then  pick the forms  $G_{i,j}$'s. The product of the Hasse invariants of the chosen forms gives, say $v=\pm 1$. We finally choose the one form such that $c_qv$=1. This gives a total of $2^{\omega(d)}$  indefinite integral forms $F$ that represent $1$ (or $-1$) over $\Zp$ for all $p$.  By Prop.~\ref{propA1}, there is a form $F^{*}$ such that $|F^*|$ represents 1 over $\mathbb{Z}$, and thus 
by Prop.~\ref{propA2} all forms in that genus represent $1$ or $-1$. For our application, we need anisotropic forms. By Prop.~\ref{isotropics} we see that there is exactly one isotropic class, so that we obtain $2^{\omega(d)}-1$ anisotropic classes.

If $d\geq 3$ is even, the analysis is much like the above, except for the choice of the forms for $p=2$. We again distinguish an odd prime $q$, and have $2^{R-2}$ choices from odd primes $p\neq q$; all of them represent both $\pm 1$ and $\pm 2$. For $p=2$ with  forms of Type 1, there are 2 classes  giving $c_2= 1$ or $-1$ respectively, so that on choosing the one choice for $p=q$ gives us $2\times 2^{R-2}$ classes $F$ that represent 1 (or $-1$) over $\mathbb{Z}$. For classes of Type 2 when $p=2$, there is one choice with  $c_2=1$'s,  representing $ 2$ and not $\pm 1$, giving us a total of $2^{R-2}$ forms over $\Z$ representing $ 2$. To count the isotropic classes, there is one from Type 1  and one from Type 2. Thus, there are $\left(2^{\omega(d)-1}-1\right)$ anisotropic classes $C$ with $\kappa(C)=1$ and $\left(2^{\omega(d)-2}-1\right)$  with $\kappa(F)=2$.

Thus, if $d\geq 3$ is odd, $\lambda(d,1)= 2^{\omega(d)}-1$ and $\lambda(d,t)=0$ if $t\geq 2$. When $d$ is even, we have $\lambda(d,1)=\left(2^{\omega(d)-1}-1\right)$, $\lambda(d,2)= 2^{\omega(d)-2}-1$ and $\lambda(d,t)=0$ if $t\geq 3$. Let us write \[\mathcal S^\#(X,t)=\sum_{\substack{1\leq d\leq X\\d\textup{ squarefree}}}\lambda(d,t).\]
We thus have, $\mathcal{S}^\#(Y,t)=0$ if $t\geq 3$, while
\[\begin{aligned}
\mathcal{S}^\#(Y,1) &= \sum_{3\leq d\leq Y}\nolimits ^{\#o} \left( 2^{\omega(d)} -1\right) +\sum_{3\leq d\leq Y}\nolimits ^{\#e}\left( 2^{\omega(d)-1} -1\right) + O(1),\\  &= \frac{1}{2}\sum_{d\leq Y}\nolimits ^{*}2^{\omega(d)} + \frac{1}{2}\sum_{d\leq Y}\nolimits ^{*o}2^{\omega(d)} + O(Y),\end{aligned}\]
and
\[\begin{aligned}
\mathcal{S}^\#(Y,2) &= \sum_{3\leq d\leq Y}\nolimits ^{\#e}  \left( 2^{\omega(d)-2} -1\right),\\
&=  \frac{1}{4}\sum_{d\leq Y}\nolimits ^{\#}2^{\omega(d)} - \frac{1}{4}\sum_{d\leq Y}\nolimits ^{\#o}2^{\omega(d)} + O(Y) ,\end{aligned}\]
where all sums are over squarefree numbers, $\sum^{\#o}$ means a sum over odd, while $\sum^{\#e}$ means even integers.

The inner sums can be evaluated asymptotically using 
\[
\sum_{1\leq d\leq Y}\nolimits ^{\#}  2^{\omega(d)} = cY\log Y + O(Y),\]
and 
\[
\sum_{1\leq d\leq Y}\nolimits ^{\#o}  2^{\omega(d)} = \frac{1}{2}cY\log Y + O(Y),\]
with $c= \prod_p (1+\frac{2}{p})(1-\frac{1}{p})^2$. 

Then, 

\[\begin{aligned}
\textrm{\textup{MAR}}^\#(X) &= \mathcal{S}^\#(X,1) + \mathcal{S}^\#(8X,2),\\
 &= \frac{7}{4}cX\log{X} +O(X).
\end{aligned}\]
\end{proof}

\section{Genus count asymptotics}\label{genera}\ 

In this section, we obtain the asymptotic formula ${\bf (A)}$ in the Introduction.  The analysis proceeds by an explicit evaluation of the Hasse invariant for  inequivalent local forms, and by appealing to the classical theorems given in Appendix~\ref{prelim}, as well as a detailed description of the primitive $\Z_2$-classes, see Appendix \ref{local-appendix}.  The next Theorem is a refinement of {\bf (A)}.

	\begin{theorem}\label{avg-genera}
		\[ \sum_{1\leq d\leq X} g(d) =  \frac{19}{20}\frac{\zeta(2)}{\zeta(4)}X\log{X} + O(X),
		\]
\end{theorem}

\begin{proof}
	
	Let $d\geq 2$ and $F\in \mathfrak{F}$  with $D(F)=d$. Let $u$ denote a generic unit in $\Zp$.
	
	By Prop.\,~\ref{propA3}, we have
	
	\begin{enumerate}[itemsep=0.3cm,label=\upshape\arabic*.\ ]
		\item If  $p>2$ and $p\nmid d$, then  $F$ is $\Zp$-equivalent to $<1,1,d>$, giving one $\Zp$-equivalent class.  All such forms satisfy $c_p(F)=1$.
		
		\item If $p>2$ and $p\! \mid \! d$, let $p^{v_p(d)}\parallel d$ with $v_p(d)\geq 1$ and put $d=p^{v_p(d)}d_0$. Then, any $F$ is $\mathbb{Z}_p$-equivalent to one of three types: 
		
		\begin{enumerate}[itemsep=0.3cm,label=\upshape\alph*.]
			\item \underline{Type I} : If $v_p(d)$ is even, $F_p = ax_1^2 + p^{\frac{v_p(d)}{2}}(x_2^2 +bx_3^2)$, with $ab= d_0u^2$. For each $a$, $b$ is determined, giving us exactly 2 choices for such $F_p$'s. One has 
			\[ c_p(F_p)= \left[\left(\frac{a}{p}\right)\left(\frac{-d_0}{p}\right)\right]^{\frac{v_p(d)}{2}}. \]
			
			So if $4 \! \mid \! v_p(d)$, both class of $F_p$ have $c_p(F_p)=1$, while the two choices for $F_p$ give both values $c_p(F_p)=\pm 1$ when $2\parallel v_p(d)$. 
			\item \underline{Type II}: There are two choices $F_p = (x_1^2 + ax_2^2) + p^{v_p(d)}bx_3^2$, with $ab = u^2$.  Now $c_p(F_p) = \left[\left(\frac{-a}{p}\right)\right]^{v_p(d)}$. So if $2 \! \mid \! v_p(d)$, both classes of $F_p$ have $c_p(F_p)=1$, while the two choices for $F_p$ give both values $c_p(F_p)=\pm 1$ when $2\nmid v_p(d)$.
			\item \underline{Type III}: $F_p = ax_1^2 + p^{t_2}bx_2^2 + p^{t_3}cx_3^2$ with $t_2+ t_3=v_p(d) $, $1 \leq t_2 < t_3$, and $abc = d_0u^2$. Given $t_2$ and $t_3$, there are 4 choices of such forms. Thus, there are (i). $4(\frac{v_p(d)}{2}-1)$ such forms if $v_p(d)$ is even, and (ii) there are $4(\frac{v_p(d)-1}{2})$ such if $v_p(d)$ is odd.
			
			One verifies that \[ c_p(F_p)= \left(\frac{-1}{p}\right)^{t_2v_p(d)}\left(\frac{-ad_0}{p}\right)^{t_2}\left(\frac{-ab}{p}\right)^{v_p(d)}. \]
			Then, if $v_p(d)$ is odd, there are $2(v_p(d) -1)$ such classes, with half taking the value $c_p =1$. The case of $v_p(d)$ even leads to an uneven distribution of the $c_p$'s, and we do not list them here.
		\end{enumerate}
		Combining all the three types leads to the statement that for each $v_p(d)$, there are exactly $2v_p(d)$ classes. Moreover,  when $v_p(d)$ is odd, the classes are divided evenly giving values of $c_p=\pm 1$. 
		
	\end{enumerate}

We now consider the prime $p=2$. Given a nonsingular primitive ternary quadratic form $f$ with $\Z_2$-coefficients, there exist unique $u,v\geq0$ such that $f$ can be written up to equivalence as
		\[f(x_1,x_2,x_3)=\begin{cases}ax_1^2+2^ubx_2^2+2^{u+v}cx_3^3 ,\quad\text{or}\\
		 V_i(x_1,x_2)+2^vcx_3^2 ,\quad\text{or}\\
		 ax_1^2+2^uV_i(x_2,x_3),
		\end{cases}\]
		for some $a,b,c\in\Z_2$ with $(abc,2)=1$ and $i\in\{1,2\}$.
        
        The classification of $\mathbb{Z}_2$-inequivalent classes  of determinant $d=2^k\in\Z_2/(\Z^*)^2$ a power of $2$ is given in Appendix \ref{local-appendix}. We write $q=2^{u}$ and $r=2^{u+v}$. Given any nonzero square class $d\in\Z_2/(\Z_2^*)^2$, the number of form classes of determinant $d$ with fixed $(u,v)$ such that $2u+v=v_2(d)$ is given by the table below:
		\[\begin{tabular}{ |c|c c c c| }
			\hline
			$u,v$ & $0$ & $1$ & $2$ & $\geq3$ \\
			\hline
			$0$ & $2$ & $3$ & $5$ & $8$ \\
			$1$ & $3$ & $2$ & $4$ & $8$ \\
			$2$ & $5$ & $4$ & $4$ & $8$ \\
			$\geq3$ & $8$ & $8$ & $8$ & $16$ \\
			\hline
		\end{tabular}\\ \]
        
		Writing $h_2(d)$ for the number of $\Z_2$-equivalence classes of nonsingular primitive ternary $\Z_2$-quadratic forms, we have
		\[\sum_{k=0}^\infty h_2(2^k)x^k=\frac{2-x+4x^2-3x^3+5x^4-4x^5+5x^6-4x^7+4x^8}{(1-x)^2},\]
	
	so that \begin{equation}\label{g2}\sum_{k=0}^\infty h_2(2^{k})2^{-k} = \frac{19}{2}.\end{equation}
	
	Now, for each prime $p$, let $g_p(d)$ denote the cardinality of $\Gen_p(d)$. Using the notation above, if $p$ is odd then we have 
    \[g_p(d) = \begin{cases} 2v_p(d)\quad \text{if}\quad p \! \mid \!  d,\\ \quad 1 \ \quad\quad \text{if}\  p\nmid d,\end{cases}\]\\  
    while for $p=2$ we have
    \[g_2(d) \leq \begin{cases} 8v_2(d)\quad \text{if}\quad 2 \! \mid \!  d,\\ \quad 2 \ \quad\quad \text{if}\  2\nmid d. \end{cases}\]\\
    Hence we get the  bound 
		\begin{equation}\label{gbound1}
		g(d)\leq \left(3+(-1)^d\right)2^{\omega(d)}\prod_{p|d}v_p(d) \ll_{\ve} d^{\ve}.
		\end{equation}
	
	To get sharper results, we use Props.~\ref{jonesT29} and~\ref{jonesT46} as follows. We split the $d$'s into two sets, namely those that have no odd prime factors $p$ with $v_p(d)$ odd, and the complement, the generic $d$'s. The former are necessarily of the form $2^ad_1^2$ with $d_1$ odd. In an interval $1\leq d\leq X$, there are $O(\sqrt{X})$ of such. For the remaining, we pick an arbitrary odd prime $q \! \mid \! d$ with $v_q(d)$ odd. By the analysis above, we pick arbitrary forms for each finite prime $p\neq q$ with Hasse-invariants $c_p$, there being  $g_2(d)\prod_{\{p\neq q, \text{odd}\}}g_p(d)$ such. We choose the sign $c_{\infty}$ to be 1, and put $\nu = \prod_{\{p\neq q\}}c_p$. Then we choose half of the forms associated with $p=q$, chosen so that $c_{\infty}c_qv =1$ so as to satisfy all conditions needed in Props.~\ref{jonesT29} and~\ref{jonesT46}, giving us $\frac{1}{2}g_q(d)$ such forms. Choosing the index $i=1$ then gives us indefinite forms $F$ with the relevant rational and ring invariants. This then shows that for the generic $d$'s, we get
	\begin{equation}\label{gbound2}
		g(d)= \frac{1}{2}g_2(d)\prod_{p\geq 3}g_p(d)=
		\begin{cases}
		2^{\omega(d)}\prod_{p|d}v_p(d) & \text{if } 2\nmid d,\\
		\\
		2^{\omega(d)-2}g_2(d)\prod_{p|d, p>2}v_p(d) & \text{if } 2 \! \mid \!  d.                                
		\end{cases}
		\end{equation}
	
	Let $\theta(d)= 2^{\omega(d)}\prod_{p|d}v_p(d)$ if $d>1$ is odd, is $1$ if $d=1$ and is zero otherwise. Then \[ \sum_{d=1}^{\infty} \frac{\theta(d)}{d^{s}} = (1+2^{-2s})^{-1}(1-2^{-s})^{2}\frac{\zeta(2s)}{\zeta(4s)}\zeta(s)^2,
		\] 
		so that 
		\[\sum_{1\leq d \leq X} \theta(d) = \frac{\zeta(2)}{5\zeta(4)}X\log{X} + O(X).
		\]
		Since $g(d)\ll d^{\ve}$, we see from \eqref{gbound2} that
		\[
		\sum_{1\leq d\leq X} g(d) = \frac{\zeta(2)}{5\zeta(4)}\sum_{k\geq 0} \frac{g_2(2^k)}{2^{k+1}} X\log{X} + O(X).
		\]
		so that using \eqref{g2}, the result follows.
\end{proof}


\section{Preliminaries for  proofs of the Theorems and {\bf (B)}-{\bf(D)}}\label{packets-def}\

Our analysis in the previous sections relied on properties of the genera, and in the case of squarefree determinants the fact that there is only one class in a genus. To go further, we will have to explore the effect of the spinor genera on our sums. We will need some  results on the nature of the spinor genera, and this is given in the next subsection. We give first a brief sketch on how we will  proceed by considering problem $\bf{(B)}$, which then  motivates the notion of packets.

Consider the sum $I(X):=\sum_{1\leq d\leq X}h(d)$, and for simplicity here assume that the $d$'s are odd integers. We write $d=rs$, where $(r,s)=1$, with $r \in \mathcal{P}$ odd, powerful (or squarefull) and $s\in \mathcal{S}$ odd, squarefree. Fix a parameter $1\leq A\leq X$ and write $I(X)= I_A(X)+ I\!I_A(X)$, where $ I\!I_A(X)$ is the tail given by
\[
I\!I_A(X) = \sum_{\substack{r\in \mathcal{P}, A\leq r\leq X\\ rs\leq X, (r,s)=1}}h(rs).
\]
By using properties of the genera and spinor genera (see the proof of Prop.~\ref{thm1} and Prop.~\ref{propE1},(1)), one has $h(d)\ll \tau(s)2^{\omega_3(r)}$, so that
\begin{align*}
I\!I_A(X) &\ll \sum_{r\in \mathcal{P}, A\leq r\leq X}2^{\omega_3(r)}\sum_{s\in \mathcal{S},  s\leq X/r, (r,s)=1}\tau(s),\\
&\ll X\log{X}\sum_{r\in \mathcal{P}, r\geq A} \frac{2^{\omega_3(r)}}{r},\\
&\ll X\log{X}\cdot \frac{(\log{A})^{10}}{A},
\end{align*}
giving us control of the tail-sums. This uniform bound allows one to reduce ${\bf (B)}$  to studying $I_A(X)$ for $A$ fixed, as $X\to \infty$, so that uniformity in $r$ is not needed.

To analyze $I_A(X)$, we break it up into a sum over genera for each fixed $r$, namely $I_A(X)= \sum_{r\in \mathcal{P}, r\leq A}J_r(X)$, where
\[
J_r(X) = \sum_{\substack{s\in \mathcal{S}, rs\leq X\\ (r,s)=1}}\sum_{G\in \Gen(rs)} h(G),
\]
where $h(G)$ is the number of spinor classes in $G$.

For fixed $r$ as above, there is a projection map
\[ \pi_r : \Gen(rs) \to \Packet(r),
\]
where $\pi_r(G)$ is determined by the local classes of $G$ at the primes dividing $2r$ (we give the precise definitions of packets in Sec.~\ref{packs}). Since $s$ is squarefree, one sees that $|\pi_r^{-1}(H)|=2^{\omega(s)-1}$ for any $H\in \Packet(r)$ provided $s>1$. One also sees from the structure of the spinor genera (Prop.~\ref{propE1}(2)) that if $G,G'\in\Gen(rs)$ satisfy $\pi_r(G)=\pi_r(G')=H$, then $h(G)=h(G')$ is a function of $H$ and $s$, say $f(H,s)$. Then,
\[J_r(X) = \sum_{H\in \Packet(r)}\sum_{\substack{s\in \mathcal{S}, rs\leq X\\ (r,s)=1}}f(H,s)\sum_{\substack{G\in \Gen(rs)\\\pi_r(G)=H}} 1\sim\sum_{H\in \Packet(r)}\sum_{\substack{s\in \mathcal{S}, rs\leq X\\(r,s)=1}}f(H,s)2^{\omega(s)-1}.\]Suffice it to say here that the contribution of values of $s$ for which $f(H,s)>1$ turns out to be of lower order, so that that the sum over $s$ can be evaluated asymptotically as $c(r,H)X\log{X}$, giving us $I_A(X)\sim \sum_{r\leq A}\sum_{H}c(r,H)\times X\log{X}$. To execute this strategy, we modify our approach by using the dominated convergence theorem, rather than taking the direct approach in obtaining these asymptotics. The analysis of ${\bf (C)}$ is much more involved since the dependency of the minima function $\kappa(C)$ has to be taken into account when doing the factorization.


\subsection{Spinor genera}\label{spinor}\ 

We establish here a structural result (Prop.~\ref{propE1}) on the set of classes of ternary  forms in a given genus. Our analysis will bring forth the role of the decomposition of a positive integer $d\geq1$ into a product of factors $r(d)$ and $s(d)$ defined as follows:

\begin{definition}\ For each $d\in\Z_{\geq1}$, let $R(d)$ denote the set of prime numbers $p$ such that $p=2$ or $v_p(d)\geq2$. We write $r(d)=\prod_{p\in R(d)}p^{v_p(d)}$. Then  $s(d)=d/r(d)$ is an odd squarefree integer with $(r(d),s(d))=1$. 
\end{definition}

Given a commutative ring $\mathcal{A}$, we shall write $\mathcal{A}^*$ for the group of units in $\mathcal{A}$, and write $\Sq^*(\mathcal{A})=\mathcal{A}^*/(\mathcal{A}^*)^2$ for the group of unit square classes in $\mathcal{A}$. We note that $\Sq^*(\mathcal{A}_1\times \mathcal{A}_2)=\Sq^*(\mathcal{A}_1)\times\Sq^*(\mathcal{A}_2)$ for all commutative rings $\mathcal{A}_1$ and $\mathcal{A}_2$. For our application $\mathcal{A}$ is  $\mathbb{Z}$, $\mathbb{Z}_p$, $\Q_p$, or $\Z/n\Z$. We have
\begin{align*}
\Sq^*(\Z_2)&\simeq\Sq^*(\Z/8\Z)=(\Z/8\Z)^*,\quad\text{and}\\
\Sq^*(\Z_p)&\simeq\Sq^*(\Z/p\Z)\simeq\{\pm1\}\quad\text{for $p\geq3$ prime.}
\end{align*}
If $p\geq2$ prime, we have a group isomorphism 
$\Sq^*(\Q_p)\simeq\Z/2\Z\times\Sq^*(\Z_p)$
given by $p^\alpha a\mapsto ([\alpha],a)$ for all $\alpha\in\Z$ and $a\in\Z_p^*$.

In these paragraphs, following \cite{Conway1993} we treat $-1$ as a ``prime'' so that each nonzero $d\in\Z$ has a unique prime factorization $d=(-1)^a\prod_{p}p^{v_p(d)}$ with $a=v_{-1}(d)\in\{0,1\}$. Thus $-1$ is a prime divisor of a nonzero integer $d$ if and only if $d< 0$. We write $\Q_{-1}=\R$ so that $\Sq^*(\Q_{-1})\simeq\{\pm1\}$. We formally define the symbol $\Sq^*(\Z_{-1})$ to mean the trivial group \[\Sq^*(\Z_{-1})=\{1\}\leq\Sq^*(\Q_{-1}).\]Below, when we say $p$ is a prime, we shall include $p=-1$ unless otherwise specified.

Let $d\geq1$. Given a genus $G\in\Gen(d)$, writing $\Pi=\{-1,2\}\cup\{\textup{$q\mid d$ prime}\}$ we shall express $\Cl(G)$ as a suitable quotient of $\prod_{q\in\Pi}\Sq^*(\Z_q)$ by a subgroup generated by spinor operators, which we now define.

\begin{definition}[\emph{Spinor operators}]
Let $n\geq1$. Let $p$ be prime. Let $F$ be a nonsingular integral $n$-ary quadratic form.
\begin{enumerate}[label=\upshape\arabic*.\ ]
	\item The \emph{spinor norm} is the homomorphism $sn:O_F(\Q_p)\to\Sq^*(\Q_p)$ given by
	\[sn(B)=F(v_1)\cdots F(v_k)\in\Sq^*(\Q_p),\] where $v_1,\dots,v_k$ is any sequence of vectors in $\Q_p^n$ such that $B$ is a product of reflections $B=r_{v_1}\circ\cdots\circ r_{v_k}$ in said vectors.
	\item The group $\mathcal A_p(F)\leq\Sq^*(\Q_p)$ of \emph{$p$-adically automorphous numbers} for $F$ is
	\[\mathcal A_p(F)= sn(SO_F(\Z_p))\leq\Sq^*(\Q_p).\]
	\item Let $\Pi$ be a finite set of primes. We define the group homomorphism
\[\Delta_p^{(\Pi)}:\Sq^*(\Q_p)\to\prod_{q\in\Pi}\Sq^*(\Z_q),\quad\Delta_p^{(\Pi)}(A)_q=\left\{\begin{array}{l l} A/p^{v_p(A)} & q=p,\\p^{v_p(A)} & q\neq p.\end{array}\right.\]
If $A\in\mathcal A_p(F)$, we call $\Delta_p^{(\Pi)}(A)$ the \emph{spinor operator} associated to $(p,F,A,\Pi)$.
\end{enumerate}
\end{definition}

The content of the following lemma is found in \cite[Ch.~9]{Conway1993} Sections 9.4--9.6. See also \cite[Ch.3]{Montesinos}.

\begin{lemma}\label{p-autonum}
Let $n\geq3$ and $d\geq1$ be integers. Let $F$ be a nonsingular integral $n$-ary quadratic form of determinant $d$. Let $p \! \mid \!  -2d$ be prime satisfying $p\neq2$. Assume that $F$ is $\GL_n(\Z_p)$-equivalent to a diagonal form
\[a_1p^{t_1}x_1^2+\dots+a_np^{t_n}x_n^2,\] with $a_i\in\Z_p^*$ for all $i=1,\dots,n$. Then the group of $p$-adically automorphous numbers for $F$ is given by $\mathcal A_p(F)=\langle S_1\cup S_2\rangle$ where:
\begin{enumerate}[label=\upshape\arabic*.\ ]
	\item $S_1=\{a_ip^{t_i}/a_jp^{t_j}:1\leq i,j\leq n\}$, and
	\item $S_2=\Sq^*(\Z_p)$ if $t_i=t_j$ for some $i\neq j$, and $S_2=\emptyset$ otherwise.
\end{enumerate}
\end{lemma}

\begin{proposition}\label{propE1}
Let $d\geq1$ be an integer. Let $P=8\prod_{p^3\mid d\textup{ odd}}p$, and let $\Pi$ be the set of prime divisors of $P$. Then we have an identification $\prod_{q\in\Pi}\Sq^*(\Z_q)=\Sq^*(\Z/P\Z)$, and for each $G\in\Gen(d)$ there is a bijection
\[\Cl(G)\simeq\frac{\Sq^*(\Z/P\Z)}{\langle\Delta_p^{(\Pi)}(\mathcal A_p(G_p)):\textup{$p\mid -2r(d)$ prime}\rangle\cdot \langle p:p\mid s(d)\textup{ prime}\rangle}.\]
In particular, we have:
\begin{enumerate}[label=\upshape\arabic*.\ ]
	\item $h(\Ge(F))\leq 2^{2+\omega_3(d)}$, and
	\item if $G,G'\in\Gen(d)$ satisfy $G_p=G_p'$ for all primes $p \! \mid \!  2r(d)$, then $h(G)=h(G')$.
\end{enumerate}
\end{proposition}

\begin{proof}
First, by a theorem of Eichler we know that each spinor genus of nonsingular indefinite integral quadratic forms of $3$ or more variables contains a unique class. Let $\Pi'=\{-1,2\}\cup\{\text{$p$ prime}:p\mid d\}$. It follows by \cite[Ch.~9, Thm.~17 and Sec.~9.4]{Conway1993} (See also \cite[Sec.~3.5]{Montesinos}) that the set of spinor genera in the genus of $F$ is in bijection with
\[\frac{\prod_{q\in\Pi'}\Sq^*(\Z_q)}{\langle\Delta_p^{(\Pi')}(\mathcal A_p(F)):p\in\Pi'\rangle}.\]
Now, suppose that $p=-1$ or that $p\geq3$ is an odd prime satisfying $v_p(d)\in\{1,2\}$. Then $F$ is $\GL_2(\Z_p)$-equivalent to a diagonal form
$ax_1^2+bp^{t_2}x_2^2+cp^{t_3}x_3^2$
with $a,b,c\in\Z_p^*$ such that $t_1=0$ or $t_1=t_2$. By Lemma~\ref{p-autonum}, we have $\Delta_p^{(\Pi')}(\mathcal A_p(F))\supseteq\Delta_p(\Sq^*(\Z_p))=\Sq^*(\Z_p)$; in the terminology of \cite{Conway1993}, such primes are called \emph{tractable}. Moreover, if $p \! \mid \!  s(d)$, i.e., ~$p\geq3$ is an odd prime such that $v_p(d)=1$, then we have $\mathcal A_p(F)=\langle p,Sq^*(\Z_p)\rangle$ again by Lemma~\ref{p-autonum}. We thus have
\begin{align*}
\Cl(\Ge(F))&\simeq\frac{\prod_{p\in\Pi'}\Sq^*(\Z_p)}{\langle\Delta_p^{(\Pi')}(\mathcal A_p(F)):p\in\Pi'\rangle}\\
&=\frac{\prod_{p\in\Pi'}\Sq^*(\Z_p)}{\langle\Delta_p^{(\Pi')}(\mathcal A_p(F)):\textup{$p\mid -2r(d)$ prime}\rangle\cdot \langle\Delta_p^{(\Pi')}(p),\Sq^*(\Z_p):p\mid s(d)\rangle}\\
&\simeq\frac{\prod_{q\in\Pi}\Sq^*(\Z_q)}{\langle\Delta_p^{(\Pi)}(\mathcal A_p(F)):\textup{$p\mid -2r(d)$ prime}\rangle\cdot \langle p:p\mid s(d)\rangle}
\end{align*}
from which the desired result follows.
\end{proof}

\subsection{Sums over squarefree numbers}\label{sect:sqfree}\ 

As indicated above, our sums over $d$ are decomposed into subsums, where we will first sum over the corresponding squarefree $s(d)$'s. These sums are much easier to handle and lead to needing asymptotics and upper-bounds for sums of the type shown below.

\begin{proposition}\label{sqfree}
Let $r\geq1$ be an integer. For any set $\mathcal{C}_r\subseteq(\Z/r\Z)^*$ of invertible congruence classes modulo $r$, we have
\[ \sum_{\substack{1\leq s\leq X\\\textup{$s$ odd squarefree}\\p\in\mathcal C_r\:\forall \textup{prime}\:p\mid s}}2^{\omega(s)}\sim K(r,\mathcal{C}_r)\cdot X(\log X)^{2\frac{\lvert\mathcal{C}_r\rvert}{\varphi(r)}-1},\]
as $X\to\infty$, where
\[K(r,\mathcal{C}_r)=\frac{1}{\Gamma\left(2\frac{|\mathcal C_r|}{\varphi(r)}\right)}\prod_{\substack{p=2\; \textup{or}\\p\notin\mathcal C_r}}\left(1-\frac{1}{p}\right)^{2\frac{|\mathcal C_r|}{\varphi(r)}}\prod_{\substack{\textup{$p$ odd prime}\\p\in\mathcal C_r}}\left(1+\frac{2}{p}\right)\left(1-\frac{1}{p}\right)^{2\frac{|\mathcal C_r|}{\varphi(r)}}.\]
\end{proposition}

\begin{proof}
Our sum can be written as $\sum_{1\leq n\leq X} a_n$ where the terms $a_n$ are defined by the Dirichlet series
\[\sum_{n=1}^\infty\frac{a_n}{n^s}=\prod_{\substack{\textup{$p$ odd prime}\\p\in\mathcal C_r}}\left(1+\frac{2}{p^s}\right).\]
By the prime number theorem in arithmetic progressions, we have
\[ \sum_{p\leq X}a_p\log p=2\sum_{\substack{p\leq X,p\in\mathcal C_r,\\\textup{$p$ odd prime}}}\log p=2\frac{\lvert\mathcal{C}_r\rvert}{\varphi(r)}X+O\left(\frac{X}{\log X}\right),\]
and moreover  $|a_n|\leq\tau(n)$ for all $n$. The  result follows by an application of the Landau--Selberg--Delange method given in \cite[Theorem 1]{GranK}.
\end{proof}

\begin{proposition}\label{sqfree-twist}
Let $r\geq1$ be an integer. For a nontrivial character $\chi:(\Z/r\Z)^*\to\C^*$, as $X\to\infty$,
\[ \sum_{\substack{1\leq s\leq X\\\textup{$s$ odd squarefree}\\(s,r)=1}}2^{\omega(s)}\chi(s)=o(X\log X).\]
\end{proposition}
\begin{proof}
Our sum can be written as $\sum_{1\leq n\leq X} a_n$ where the terms $a_n$ are defined by the Dirichlet series
\[\sum_{n=1}^\infty\frac{a_n}{n^s}=\prod_{\substack{\textup{$p$ odd prime,}\\(p,r)=1}}\left(1+\frac{2\chi(p)}{p^s}\right).\]
We have
\[ \sum_{p\leq X}a_p\log p=\frac{1}{\varphi(r)}\sum_{a\in (\Z/r\Z)^*}2\chi(a)\cdot X +O\left(\frac{X}{\log X}\right)=O\left(\frac{X}{\log X}\right),\]
and   $|a_n|\leq\tau(n)$ for all $n$. The  result follows from Theorem 1 of \cite{GranK}.
\end{proof}

\begin{corollary}\label{useful}
For all real numbers $A>0$, and $X\geq2$, we have
 \[ \sum_{\substack{1\leq s\leq AX\\\textup{$s$ odd squarefree}}}2^{\omega(s)}\ll AX\log X\cdot\log^+\!A ,\]
where the implied constant is independent of $A$ and $X$.
\end{corollary}

\begin{proof}
If $AX<1$, then the left hand sum is zero, so that it suffices to assume $AX\geq 1$. By Prop.~\ref{sqfree} with $r=1$, there is a constant $c\geq1$ such that for any $A>0$ and $X\geq2$ we have
\[\sum_{\substack{1\leq s\leq AX\\\textup{$s$ odd squarefree}}}2^{\omega(s)}\leq c\cdot AX\log (AX)+1\leq c\cdot AX\log AX+ 2AX\log X\cdot\log^+\!A .\]
If $J(X)=\frac{\log AX}{\log X}$, then $0<J(X)\leq 1 \leq \log^+\!A$ if $A\leq 1$, and if $A>1$, then $J(X)$ is decreasing so that $J(X)\leq J(2)\leq 100\log^+\!A$. The result then follows.
\end{proof}


\subsection{Packets}\label{packs}\ 

 We give here a precise definition of what we mean by packets, show some properties and  consequences. While the definition makes sense in greater generality for integral quadratic forms in more variables, our focus here will be on forms from $\mathfrak F$, the set of all primitive non-singular integral ternary indefinite quadratic forms of positive determinant.

\begin{definition}[\emph{Packets}]\
\begin{enumerate}[label=\upshape\arabic*.\ ]
	\item If $S$ is a finite set of primes containing $2$, we define a \emph{packet} with \emph{support $S$} to be an element of $\prod_{p\in S}\Gen_p$. We write $\Supp(H)$ for the support of a packet $H$.
	\item We denote the set of all packets by $\Packet$.
	\item Let $S$ be a finite set of primes containing $2$.
	\begin{enumerate}[itemsep=0.3cm,label=\upshape\alph*.]
	\item If $f\in\mathfrak F$, we define the \emph{$S$-restriction} of $f$ to be the packet $f|_S=(f_p)_{p\in S}$ where $f_p$ denotes the class of $f$ in $\Gen_p$. The assignment $(-)|_S:\mathcal F\to\Packet$ descends to functions on $\Cl$ and $\Gen$.
	\item If $H\in\Packet$, we define the \emph{$S$-restriction} of $H$ to be the packet with support $\Supp(H)\cap S$ given by $H|_S=(H_p)_{p\in\Supp(H)\cap S}$.
\end{enumerate}
\end{enumerate}
\end{definition}

\begin{definition}[\emph{Determinant}]\ 
Given a packet $H$, the \emph{determinant} of $H$ is
\[D(H)=\prod_{p\in \Supp(H)}p^{v_p(D(H_p))}\in\Z ,\]
where $D(H_p)\in\Z_p/(\Z_p^*)^2$ denotes the square class of $D(F_p)$ for any choice of representative $F_p$ of the class $H_p$.
\end{definition}

\begin{remark}\
\begin{enumerate}[label=\upshape\arabic*.\ ]
	\item If $F$ is a form, class, genus, or a packet and if $\varphi:\Z\setminus\{0\}\to\mathbb R$ is any function, we shall often write $\varphi(F)=\varphi(D(F))$, so for example $D(F)=r(F)s(F)$.
	\item Given $H\in\Packet$, it is possible to have $p\in\Supp(H)$ such that $v_p(H)=0$. Thus, the determinant does not in general determine the support of $H$.
\end{enumerate}
\end{remark}

As Prop.~\ref{propE1} shows, given a genus $G\in\Gen$ much information is contained in the restriction of $G$ to the set of primes dividing $2r(D(G))$. This motivates the following definition.

\begin{definition}[\emph{Root packets and a partial order}]\label{partial}\
\begin{enumerate}[label=\upshape\arabic*.\ ]
	\item Let $F$ be a form, class, genus, or packet. Recall that $R(D(F))$ is the set of primes $p$ such that $p=2$ or $v_p(D(F))\geq2$. The \emph{root packet} of $F$ is the packet
	\[F_\rt=F|_{R(D(F))}.\]
	The set of all root packets will be denoted by  $\mathfrak R$. 
	\item We define a partial order $\leq$ on $\Packet$ as follows:  $H\leq H'$ if
\begin{enumerate}[label=\upshape\alph*.]
	\item $\Supp(H)\subseteq \Supp(H')$,
	\item $H_{\rt}=H'_{\rt}$, and
	\item $H=H'|_{\Supp(H)}$.
\end{enumerate}
	 More generally, if $H$ is a packet and $F$ is a form, class, or genus, then we write $H\leq F$ to mean that $H_{\rt}=F_{\rt}$ and $H=F|_{\Supp(H)}$.
\end{enumerate}
\end{definition}

The root packets are precisely the minimal elements in $\Packet$ with respect to the partial ordering. If $F$ is a form, class, genus, or a packet then there is a unique root packet $H\in\mathfrak R$ such that $F\geq H$. Thus, root packets and the relation $\leq$ provide a convenient way to partition form classes and genera (Lemma~\ref{trivlem}).

\begin{definition} Let $d\in\Z_{\geq1}$ and let $H\in\Packet$ be a packet. If $Y$ is $\Cl$ (resp.~$\Gen$, resp.~$\Packet$), we write $Y_H$  for the set of classes (resp.~genera, resp.~packets) $X$ of forms such that $X\geq H$. Then, we write $Y_H(d)= Y(d)\cap Y_H$. In particular, we write 
\[ h_H(d)=|\Cl_H(d)|\quad \text{and} \quad g_H(d)=|\Gen_H(d)|.\]
\end{definition}

\begin{lemma}\label{trivlem}
For $d\in\Z_{\geq1}$,
\begin{enumerate}[label=\upshape\arabic*.]
	\item We have the partitions
	\[\Cl(d)=\coprod_{H\in\mathfrak R}\Cl_H(d)\quad\text{and}\quad\Gen(d)=\coprod_{H\in\mathfrak R}\Gen_H(d),\]
so that \[h(d)=\sum_{H\in\mathfrak R}h_H(d)\quad  \text{and}\quad g(d)=\sum_{H\in\mathfrak R}g_H(d).\]
	\item Given a packet $H\in\Packet$, we have
\[\Cl_H(d)=\coprod_{G\in\Gen_H(d)}\Cl(G)\quad \text{and} \quad h_H(d)=\sum_{G\in\Gen_H(d)}h(G).\]
\end{enumerate}
\end{lemma}

\begin{proof}
These are all immediate from the definitions.
\end{proof}





\section{Class number asymptotics}\label{clnbr}

We obtain the asymptotic formula {\bf (B)} using the strategy outlined in the introduction. We use Lemma~\ref{trivlem} to decompose the sum over root packets $H$. For each such $H$, we give in Prop.~\ref{splitting}  a lower order   upper-bound for the number of classes $C\in\Cl_H$ with determinant $D(C)\leq X$ such that $h(\Ge(C))>1$. To prove ${\bf (B)}$ we use the dominated convergence theorem for a sum over the $H$'s, for which we need a dominating bound (Prop.~\ref{class-tail}), and  pointwise convergence (Prop.~\ref{avg-class-pointwise}). The proof is then completed in Thm.~\ref{avg-class}.

Recall that, if $H\in\Packet$ is a packet of forms and $d\ge1$, then $h_H(d)=|\Cl_H(d)|$ where $\Cl_H(d)=\{C\in\Cl(d):C\geq H\}$.
\begin{definition}\label{Def-5.1}\
	For any $d\in\Z$ and $H\in\Packet$, we define
        \begin{align*}
    &\Cl_H^{(1)}(d)=\{C\in \Cl_H(d):h(\Ge(C))=1\},\quad h_H^{(1)}(d)=|\Cl_H^{(1)}(d)|,\\
    \intertext{and}
    &\Cl_{H}^{(2+)}(d)=\{C\in \Cl_H(d):h(\Ge(C))\geq 2\},\quad h_H^{(2+)}(d)=|\Cl_H^{(2+)}(d)|.
    \end{align*}
\end{definition}
	It follows from the definition that, for any packet $H$ and $d\geq1$, we have
\begin{equation}\label{hbound}
    h_H^{(1)}(d)\leq g_H(d)\leq h_H^{(1)}(d)+h_H^{(2+)}(d) = h_H(d).\end{equation}

\begin{lemma}\label{lemgen1}
	Let $H\in\mathfrak R$ be a root packet and $d\geq1$ an integer. Then,\ 
	\begin{enumerate}[label=\upshape\arabic*.\ ]
		\item  Any two genera in $\Gen_H(d)$ contain the same number of classes, denoted $i_H(d)$.
	\end{enumerate}
	Moreover, we have
	\begin{enumerate}[label=\upshape\arabic*.\ ]
		\item[\upshape{2}.\ ]  $i_H(d)=h_H(d)/g_H(d)\leq 2^{2+\omega_3(r(d))}$,
		\item[\upshape{3}.\ ]  $g_H(d)\leq 2^{\omega(s(d))}$, \text{and}
		\item[\upshape{4}.\ ]  $h_H(d)\leq 2^{2+\omega_3(r(d))}\cdot 2^{\omega(s(d))}$.
	\end{enumerate}
\end{lemma}

\begin{proof}
	Parts (1) and (2) are consequences of Prop.~\ref{propE1}. Part (3) follows from Prop.~\ref{propA3}. Part (4) follows by combining parts (2) and (3).
\end{proof}

\begin{proposition}\label{splitting}
	Let $H\in\Packet$ be a packet. As $X\to\infty$,
	\[\sum_{1\leq d\leq X}h_H^{(2+)}(d)=O_H(X).\]
\end{proposition}

\begin{proof}
	We may assume without loss of generality that $H\in\mathfrak R$ is a root packet. Using the notation of Lemma~\ref{lemgen1}, note that
	\[\sum_{1\leq d\leq X}h_H^{(2+)}(d)=\sum_{\substack{1\leq d\leq X,\\i_H(d)>1}}g_H(d)i_H(d)\leq\sum_{\substack{1\leq d\leq X,\\i_H(d)>1},}2^{\omega(s(d))}\cdot 2^{2+\omega_3(d)}.\]
	Let $r=r(H)$. Since $i_H(d)=0$ if $r(d)\neq r$, it remains only to show that
	\[\sum_{\substack{1\leq s\leq X,(s,2r)=1\\\textup{$s$ squarefree},\\i_H(rs)>1}}2^{\omega(s)}=O_H(X).\]
	Let $P=8\prod_{p^3\mid r\text{ odd}}p$. Let $Q(P)$ denote the set of all subgroups of index $2$ in $(\Z/P\Z)^*$. By Prop.~\ref{propE1}, we have
	\[\sum_{\substack{1\leq s\leq X,(s,2r)=1,\\\textup{$s$  squarefree},\\i_H(rs)>1}}2^{\omega(s)}\leq\sum_{\mathcal{C}\in Q(P)}\sum_{\substack{1\leq s\leq X,(s,2r)=1,\\\textup{$s$  squarefree},\\p\in\mathcal C\:\forall\text{prime}\:p\mid s}}2^{\omega(s)}.\]
	Since $Q(P)$ is finite, we deduce by Prop.~\ref{sqfree} that the right hand side is $O_H(X)$.
\end{proof}

\vspace{4pt}

\begin{proposition}[\emph{Domination}]\label{class-tail}
For any $H\in\Packet$ and any $X\geq2$,
\[\frac{1}{X\log X}\sum_{1\leq d\leq X} h_H(d)\ll  \frac{2^{\omega_3(r(H))}}{D(H)},\]
where the implied constant is independent of $H$ and $X$.
\end{proposition}

\begin{proof}
Let us write $\delta=\prod_{p\in\Supp(H)}p$. We have by Props.~\ref{propA3} and~\ref{propE1} that
\begin{align*}
\sum_{1\leq d\leq X} h_H(d)&=\sum_{\substack{1\leq s\leq X/D(H),\\(s,2\delta(H))=1,\\\textup{$s$ squarefree}}} g_H(D(H)s)i_H(D(H)s)\\
&\leq \sum_{\substack{1\leq s'\leq X/D(H),\\\textup{$s'$ odd squarefree}}} 2^{\omega(s')}\cdot 2^{2+\omega_3(r)}\ll X\log X\cdot\frac{2^{\omega_3(r(H))}}{D(H)}
\end{align*}
where the last bound follows from Cor.~\ref{useful}.
\end{proof}

\vspace{4pt}

\begin{proposition}[\emph{Pointwise convergence}]\label{avg-class-pointwise}
Let $H\in\Packet$ be a packet. As $X\to\infty$, we have
\[\sum_{1\leq d\leq X} g_H(d)\sim \sum_{1\leq d\leq X}h_H(d)\sim \sum_{1\leq d\leq X}h_H^{(1)}(d)\sim A(H)\cdot X\log X ,\]
where
\[A(H)=\frac{1}{ 2^{|\Supp(H)|+2}}\cdot\frac{1}{D(H)}\prod_{p\in\Supp(H)}\left(1-\frac{1}{p}\right)^2\prod_{p\notin\Supp(H)}\left(1+\frac{2}{p}\right)\left(1-\frac{1}{p}\right)^2.\]
\end{proposition}

\begin{proof}
Combining \eqref{hbound}  with Prop.~\ref{splitting}, it suffices to show that
\[\sum_{1\leq d\leq X} g_H(d)\sim A(H)\cdot X\log X,\]
as $X\to\infty$. Let us write $\delta=\prod_{p\in\Supp(H)} p$. Then $g_H(d)\neq0$ precisely if $d$ is of the form $D(H)s$ where $s\geq1$ is a squarefree positive integer satisfying $(s,2\delta)=1$, such that $D(H)s=D(H_q)\mod(\Z_q^*)^2$ for every prime $q\in\Supp(H)$. For $s\neq1$ satisfying the requisite properties, we have $g_H(D(H)s)=2^{\omega(s)-1}$ by arguing as in the proof of Thm. \ref{avg-genera} by consideration of Hasse invariants.
Writing $X(q)$ for the set of all (quadratic) characters of $\Sq^*(\Z_q)$, by Props.~\ref{sqfree} and \ref{sqfree-twist} we have
\begin{align*}
&\sum_{1\leq d\leq X}g_H(d)=\sum_{\substack{1\leq s\leq X/D(H),(s,2\delta)=1\\ \textup{$s$ squarefree}}}g_H(D(H)s),\\
&\sim\sum_{\substack{1< s\leq X/D(H),(s,2\delta)=1\\\textup{$s$  squarefree}}}\frac{2^{\omega(s)-1}}{2\cdot 2^{|\Supp(H)|}}\prod_{q\in\Supp(H)}\sum_{\chi\in X(q)}\chi(D(H)s/D(H_q)),\\ &\sim\frac{1}{ 2^{|\Supp(H)|+2}}\prod_{p\in\Supp(H)}\left(1-\frac{1}{p}\right)^2\prod_{p\notin\Supp(H)}\left(1+\frac{2}{p}\right)\left(1-\frac{1}{p}\right)^2 \frac{X\log{X}}{D(H)},
\end{align*} with the main term coming from the principal character.
\end{proof}

\vspace{4pt}

\begin{theorem}
[\emph{Formula {\bf(B)}}]\label{avg-class}
As $X\to\infty$,
\[\sum_{1\leq d\leq X} h(d)\sim \frac{19}{20}\frac{\zeta(2)}{\zeta(4)}X\log X.\]
\end{theorem}

\begin{proof}
For each $X\geq2$, consider the function $f_X^{({\bf B})}(H):\mathfrak R\to\R_{\geq0}$ defined by
\[f_X^{({\bf B})}(H)=\frac{1}{X\log X}\sum_{1\leq d\leq X}h_H(d),\]
and let $N:\mathfrak R\to\R_{\geq0}$ be given by $N(H)=2^{\omega_3(r(H))}/r(H)$. We note that
\[\sum_{H\in\mathfrak R}N(H)=\sum_{H\in\mathfrak R}\frac{2^{\omega_3(r(H))}}{r(H)}\leq\sum_{r\in |R|}\frac{2^{\omega_3(r)}\cdot 4\cdot 2^{\omega(r)}\prod_{p\mid r}v_p(r)}{r},\]
where $|R|=\{r(d):d\in\Z_{\geq1}\}$. Because the numerator in each summand of the last series is $O_\epsilon(r^\epsilon)$, we conclude that $N\in L^1(\mathfrak R)$. It follows by Prop.~\ref{class-tail} that
$f_X^{({\bf B})}(H)\ll N(H)$
valid for all $H\in\mathfrak R$ and all $X\geq2$, where the implied constant is independent of $H$ and $X$. Thus $f_X^{({\bf B})}\in L^1(\mathfrak R)$ for every $X\geq2$. By the dominated convergence theorem the pointwise limit $f^{({\bf B})}=\lim_{X\to\infty} f_X^{({\bf B})}$ is in $L^1(\mathfrak R)$ and
\[\lim_{X\to\infty}\frac{1}{X\log X}\sum_{1\leq d\leq X}h(d)=\lim_{X\to\infty}\sum_{H\in\mathfrak R}f_X^{({\bf B})}(H)=\sum_{H\in\mathfrak R}f^{({\bf B})}(H)<\infty.\]
Now, we have
\begin{align*}
\sum_{H\in\mathfrak R}f^{({\bf B})}(H)&=\sum_{H\in\mathfrak R}\left[\lim_{X\to\infty}\frac{1}{X\log X}\sum_{1\leq d\leq X}h_H(d)\right],\\
&=\sum_{H\in\mathfrak R}\left[\lim_{X\to\infty}\frac{1}{X\log X}\sum_{1\leq d\leq X}g_H(d)\right]\quad\text{by Prop.~\ref{avg-class-pointwise}},\\
&=\lim_{X\to\infty}\frac{1}{X\log X} \sum_{H\in\mathfrak R}\sum_{1\leq d\leq X}g_H(d)=\lim_{X\to\infty}\frac{1}{X\log X} \sum_{1\leq d\leq X} g(d),
\end{align*}
where the first equality on the last line follows again by the dominated convergence theorem and the fact that $g_H(d)\leq h_H(d)$ for all $H\in\mathfrak R$ and $d\in\Z$. The desired result follows by Thm.~\ref{avg-genera}.
\end{proof}


\section{Local representations}\label{local-rep}

In preparation for the proof of ${\bf (C)}$, the asymptotic for $\textup{MAR}(X)$, we will need some results that control the size of $\kappa(C)$ when we factorize $D(C)$. The main result in this section is Prop.~\ref{kappa-finite}, while an analog of Prop.~\ref{splitting} used to prove ${\bf (B)}$ is given in Prop.~\ref{lambdas}. To obtain these, we need to understand  integers that are locally everywhere represented by a genus or a packet of forms. The partial order introduced (in Def.~\ref{partial}) on the set of packets plays an important role in our analysis. We begin with some definitions and some preliminary results.

\begin{definition}
Let $F$ be a form, class, genus, or packet, and for each prime $p$ let $F_p$ be its image in $\Gen_p$ .
\begin{enumerate}[label=\upshape\arabic*.\ ]
	\item Let $\mathcal{R}_p(F)$ denote the set of all integers represented by $F_p$ over $\mathbb Z_p^3\setminus\{{\bf 0}\}$.
	\item $\mathcal {R}(Z)$ will denote the set of all  integers represented by all the $F_p$'s with $p\in \Supp(F)$. When $F$ is a form, class, or genus, we take $\Supp(F)$ to mean the set of all primes, so that
	\[\mathcal {R}(F)=\bigcap_{p\in\Supp(F)}\mathcal R_p(F).\]  
	\item $K(F)$ (resp.~$K^*(F)$) will denote the least nonnegative (resp.~positive) integer in   $\mathcal{R}^{\textup{(abs)}}(F) := \left\{|l|\ : l\in \mathcal{R}(F)\right\}$.
    \item If $F$ is a form or class, let $\kappa(F)$ (resp.~$\kappa^*(F)$) denote the least nonnegative (resp.~positive) integer in $\left\{|t|:t\in F(\Z^3\setminus\{\bf{0}\})\right\}$.
\end{enumerate}
\end{definition}

\begin{remark}
If $F$ is a form or class, $\mathcal R(F)$ may be different from $R(F)=F(\Z^3\setminus\{{\bf 0}\})$. Moreover when $G$ is a genus, $K(G) = \min_{\Ge(C)=G}\kappa(C)$.
\end{remark}

It is easy to see that $K^*:\Packet\to\Z_{\geq0}$ is an increasing function where $\Packet$ is equipped with the partial order in Def.~\ref{partial} and $\Z_{\geq0}$ is equipped with the usual linear order. One might wonder if there exist arbitrarily long chains $H_1< H_2<\dots$ of packets with a strictly increasing sequence $K^*(H_1)<K^*(H_2)<\dots$ of $K^*$-values. We shall show in Prop.~\ref{upper-bound} that this cannot occur.

\begin{lemma}\label{triv}
Let $F\in\mathfrak F$ and let $p$ be an odd prime. If $p\nmid D(F)$, then $\mathcal R_p(F)=\Z$.
\end{lemma}

\begin{proof}
This follows from Hensel's lemma and the Chevalley--Warning Theorem.
\end{proof}

\begin{proposition}\label{upper-bound}
Let $H\in\Packet$ be a packet. Then there is a positive integer $N \! \mid \!  2 r(H)$ such that, for any odd $m\in\mathbb N$ coprime to $D(H)$, we have $Nm\in \mathcal{R}^{\textup{(abs)}}(H)$. In particular, \[K^*(H)\leq 2\,  r(H).\]
\end{proposition}

\begin{proof} 
Let $p\in \Supp(H)$ be an odd prime. We may assume $H_p$ is the class of the diagonal form
\[H_p(x_1,x_2,x_3)=ax_1^2+p^{t_2} bx_2^2+p^{t_3}cx_3^2, \]
for some integers $a,b,c$ coprime to $p$ and $0\leq t_2\leq t_3$ such that $t_2+t_3=v_p(H)$. We have the following cases.
\begin{enumerate}[label=\upshape\arabic*.\ ]
	\item If $t_2$ is even, then $ H_p(p^{\frac{1}{2}t_2}u,v,0)=p^{t_2}(au^2+bv^2)$.
	\item If $t_2$ is odd and $t_3$ is even, then $H_p(p^{\frac{1}{2}t_3}u,0,v)=p^{t_3}(au^2+cv^2)$.
	\item If $t_2$ and $t_3$ are odd, then $H_p(0,p^{\frac{1}{2}(t_3-t_2)}u,v)=p^{t_3}(bu^2+cv^2)$.
\end{enumerate}
Defining $e_p(H)=t_2$ if $t_2$ is even, and $e_p(H)=t_3$ otherwise, we deduce by Prop.~\ref{propA3}(1) that $p^{e_p(H)}m\in\mathcal R_p(H)$ for any integer $m$ coprime to $p$. In fact, if $v_p(H)=0$, then we have $\mathcal R_p(H)=\Z$ by Lemma~\ref{triv}.

We now consider the prime $p=2$. If $H_2$ is diagonalizable, then we define $e_2(H)$ as above and we see that $2^{e_2(H)}m\in \mathcal{R}^{\textup{(abs)}}_2(H)$ for every odd $m\in\mathbb N$ by Prop.~\ref{binquad-lem}. If $H_2$ is not diagonalizable, then by Appendix ~\ref{Z2} the form $H_2$ represents
$2^{t}V_{i}(x,y)$
for some $i\in\{1,2\}$ and $0\leq t\leq v_2(d)/2$, where $V_1(x,y)=2xy$ and $V_2(x,y)=2(x^2+xy+y^2)$ as defined in Prop.~\ref{propB1}. In this case, defining $e_2(H)=t+1$, we conclude by Prop.~\ref{binquad-lem} that $2^{e_2(H)}m\in \mathcal{R}^{\textup{(abs)}}_2(H)$ for every odd $m\in\mathbb N$.

Putting $N=\prod_{p\in \Supp(H)} p^{e_p(H)}$, it then  follows by the above that, if $m\in\mathbb N$ is odd and coprime to $D(H)$, then $Nm\in\mathcal{R}^{\textup{(abs)}}(H)$. We note that $e_2(H)\leq v_2(H)+1$ and $e_p(H)\leq v_p(H)$ for all odd $p\in \Supp(H)$, and moreover $e_p(H)=0$ for odd $p\in\Supp(H)$ if $v_p(H)\leq 1$. We conclude that $N \! \mid \!  2r(H)$, as stated.
\end{proof}

\begin{corollary}\label{minimal-ripe-new}
Let $H\in\mathfrak R$ be a root packet and let
\[T(H)=\{p:2\leq p\leq 2\,r(H)\}\supseteq\Supp(H).\]
If $F$ is a form, class, genus, or packet and $F\geq H$, then \[K^*(F)=K^*(F|_{T(H)}).\]
\end{corollary}
\begin{proof}
Note that we have
\[K^*(F|_{T(H)})\leq K^*(F)\leq 2r(F)=2r(H)\]
by Prop.~\ref{upper-bound}. In particular, every prime divisor of $K^*(F|_{T(H)})$ lies in $T(H)$. Thus, for any prime $p\notin T(H)$, we have $\pm K^*(F|_{T(H)})\in\mathcal{R}_p(F)$, and it follows that $K^*(F|_{T(H)})\in\mathcal R^{\textup{(abs)}}(F)$. It follows that $K^*(F|_{T(H)})=K^*(F)$ as desired.
\end{proof}

\vspace{4pt}

We now state and prove a result that is used in Sec.~\ref{Mar-sec} for the proof of ${\bf (C)}$ by the dominated convergence theorem.
 
\begin{proposition}\label{kappa-finite}
For each integer $q>1$, let $B(q)$ be the maximum of the minimal primes in each invertible congruence class modulo $q$. For any  $C\in\Cl$, we then have
\[\kappa^*(C)\leq 2r(C)B(8P(C))\ll |r(C)|P(C)^L\leq |r(C)|^{\frac{7}{2}},\]
where $P(C)=P(D(C))$ denotes the product of the odd prime divisors of $r(C)$, $L\leq 5$ denotes Linnik's constant, and the implied constant in the bound is independent of $C$.  In particular, if $H\in\mathfrak R$ is a root packet, then
\[\kappa_H:=\max\{\kappa(C):\kappa\in\Cl_H\}\ll |r(H)|^{\frac{7}{2}}.\]
\end{proposition}

\begin{proof}
Let $G$ be the genus of $C$. If $G$ contains a unique class, then \[\kappa^*(C)=K^*(G)\leq 2|r(G)|=2|r(C)|,\] by Prop.~\ref{upper-bound} as desired. If $G$ contains more than one class, then by Prop.~\ref{propE1} it follows that there exists an invertible congruence class $a\mod 8P(C)$ such that $s(C)$ is not divisible by any prime in this class. Choosing a minimal prime $p\equiv a\mod 8P(C)$, we see by Prop.~\ref{upper-bound} that $Np\in\mathcal{R}^{\textup{(abs)}}(G)$ for some $N \! \mid \!  2r(C)$. By Prop.~\ref{propD2}, we conclude that $Np\in\mathcal{R}^{\textup{(abs)}}(C)$ showing $\kappa^*(C)\leq 2r(C)B(8P(C))$. The remaining statements follow from Linnik's theorem and  Xylouris \cite{Xylouris2011} showing that $L\leq 5$ (any value of $L$ suffices for our application, but we have chosen $L=5$ to give an explicit rate of growth for the upper-bound in the proposition).
\end{proof}

\vspace{4pt}

As in Section~\ref{clnbr}, we shall next show that  genera containing more than one spin class  play a negligible role in the analysis of $\text{MAR}(X)$.

\begin{definition}\label{lambda-def}
Given any root packet $H\in\mathfrak R$, any $d\in\Z_{\geq1}$, and any $t\in\Z_{\geq0}$, let
\begin{align*}
\lambda_H(d,t)&=\lvert\{C\in\Cl_H(d):\kappa(C)=t\}\rvert,\\
\lambda_H^{(1)}(d,t)&=\lvert\{C\in\Cl_H^{(1)}(d):\kappa(C)=t\}\rvert,\\
\lambda_H^{(2+)}(d,t)&=\lvert\{C\in\Cl_H^{(2+)}(d):\kappa(C)=t\}\rvert,\quad\text{and}\\
\lambda_H^g(d,t)&=\lvert\{G\in\Gen_H(d):K(G)=t\}\rvert.
\end{align*}
\end{definition}
\begin{proposition}\label{lambdas}
Let $H$ be a root packet, let $d\in\Z_{\geq1}$, and let $t\in\Z_{\geq0}$.
\begin{enumerate}[label=\upshape\arabic*.\ ]
	\item We have \[\lambda_H^{(1)}(d,t)\leq\lambda_H^g(d,t)\leq \lambda_H(d,t)=\lambda_H^{(1)}(d,t)+\lambda_H^{(2+)}(d,t).\]
	\item Assume that $t>0$. If
    \[\mathcal F(H,t)=\left\{H'\in\Packet_H:\Supp(H')=T(H),K^*(H')=t\right\},\]
where $T(H)=\{p:2\leq p\leq 2r(H)\}$ is defined as in Corollary~\ref{minimal-ripe-new}, then
\[\left|\lambda_H^g(d,t)-\sum_{H'\in\mathcal F(H,t)}g_{_{H'}}(d)\right|\leq 1.\]
	\item We have $\lambda_H^{(2+)}(d,t)\leq h_H^{(2+)}(d)$. In particular,
	\[\sum_{1\leq d\leq X}\lambda_H^{(2+)}(d,t)=O_H(X).\]
\end{enumerate}
\end{proposition}
\begin{proof}\ 

1. The first inequality is clear. The second inequality follows from the fact that in each genus $G$ there is at least one class representing $\pm K(G)$. 

2.  By Corollary~\ref{minimal-ripe-new}, if $G\in\Gen_H(d)$ then $K^*(G)=K^*(G|_{T(H)})$. It follows that
\[\{G\in\Gen_H(d):K^*(G)=t\}=\{G\in\Gen_H(d):G|_{T(H)}\in\mathcal F(H,t)\}.\]
Since we have
\begin{align*}
    &\{G\in\Gen_H(d):K(G)=t\}\quad\\
    &\quad\quad =\{G\in\Gen_H(d):K^*(G)=t\}\setminus\{G\in\Gen_H(d):K(G)=0\},
\end{align*}
and the set being subtracted has cardinality at most $1$, the desired result follows.

3. This follows from Prop.~\ref{splitting}.
\end{proof}


\section{\texorpdfstring{Asymptotic formula for $\text{MAR}(X)$}{Asymptotic formula for MAR(X)}}\label{Mar-sec}

 We now give a proof of Theorem \ref{IThm1}.  Following the outline in the introduction we exhibit a dominating function on $\mathfrak R$ for the functions $f_{X}^{(\bf{C})}(H)$ and then prove that they have pointwise limits as $X \to \infty$.  The dominating function is given explicitly in Proposition \ref{martini-tail}.  The proof of the bounds in Proposition \ref{martini-tail} relies on two key Propositions as well as the uniform bound in Proposition \ref{kappa-finite}.  Proposition \ref{mechanical} shows that if sufficiently many small points (in terms of a parameter $M$), satisfying auxiliary divisibility conditions, can be produced in $\mathcal{R}(H)$, then a combinatorial sieve argument yields uniform bounds for $f_{X}^{(\bf{C})}(H)$.  The divisibility conditions serve to execute the sieving as well as to apply the local to global representation conditions of Jones and Watson \cite{JW1956} -- see Proposition \ref{propD2}.  Proposition \ref{brun} establishes  the existence of the $\kappa_j$'s using a lower bound Brun sieve.  The requisite level of distribution is established using the method developed in Sec.~\ref{Watson}.  The pointwise limit of $f_{X}^{({\bf C})}(H)$ is proven in Proposition \ref{avg-martini-pointwise} using the results from Sec.~\ref{local-rep}.

For any root packet $H\in\mathfrak R$ and $X\geq2$, we define
\[\textrm{\textup{MAR}}_{X}(H)=\lvert\{C\in\Cl_H:\mu(C)\geq1/X\}\rvert,\]
and
\[f_{X}^{(\bf{C})}(H)=\frac{1}{X\log X}\textrm{\textup{MAR}}_X(H). \]
\begin{proposition}
\label{mechanical}
Let $H\in\mathfrak R$. Suppose there exist  pairwise coprime positive squarefree integers $k_1,\dots,k_\nu \in\mathcal R(H)$, all coprime to $2r(H)$.  

 Let $b$ be the least prime  dividing  $\prod_{i=1}^\nu k_i$,  $M=\max\{k_1,\cdots,k_\nu\}$, and  $\kappa_H=\max\{\kappa(C):C\in\Cl_H\}$. Then
\[f_{X}^{(\bf{C})}(H) \ll \frac{2^{\omega_3(r(H))}}{r(H)}\left[M^3\log M+\kappa_H^3\log\kappa_H\cdot\left(\frac{\log M}{\log b}\frac{2}{b}\right)^\nu \ \right],
\]
valid for all $X\geq2$, with the implied constant  independent of $X$ and $H$.
\end{proposition}

\begin{proof}
We note that if $C\in\Cl_H$ satisfies $\kappa(C)>M$, then we must have $(s(C),k_i)\neq 1$ for every $i=1,\dots,\nu$. Otherwise, if say $(s(C),k_1)=1$ then the genus of $C$ must represent $k_1$, and by Prop.~\ref{propD2} we see that $C$ represents $k_1$, so that $\kappa(C)\leq k_1 \leq M$. 

Hence we have 
\begin{align*}
\textrm{\textup{MAR}}_X(H)&=\sum_{\substack{C\in\Cl(H)\\D(C)\leq \kappa(C)^3X\\\kappa(C)\leq M}}1+\sum_{\substack{C\in\Cl(H)\\D(C)\leq\kappa(C)^3X\\\kappa(C)> M}}1\leq\sum_{1\leq d\leq M^3 X}h_H(d)+\sum_{\substack{C\in\Cl(H)\\D(C)\leq \kappa_H^3 X\\(s(C),k_i)\neq 1\forall i}}1.
\end{align*}

For the first sum, by Corollary~\ref{useful} we have
\[\begin{aligned}
\sum_{1\leq d\leq M^3 X}h_H(d) &\leq \sum_{\substack{s\leq M^3X/r(H)\\\text{$s$ odd squarefree}}}2^{\omega(s)}\cdot 2^{2+\omega_3(r(H))}, \\ &\ll X\log X\cdot\frac{2^{\omega_3(r(H))}}{r(H)}\cdot M^3\log M,
\end{aligned}\]
where the implied constants are independent of $X$ and $H$. 

To analyze the second sum, let $\mathcal P$ be the set of all products of the form $\prod_{i=1}^\nu p_i$ where each $p_i$ is a prime divisor of $k_i$. Note that $b^\nu\leq\min(\mathcal P)$. Also, if $t_i$ is the number of prime divisors of $k_i$ then $b^{t_i}\leq k_i\leq M$ so we have $t_i\leq(\log M)/(\log b)$. Using Lemma~\ref{lemgen1}, one has $g_H(r(H)Ps)\leq 2^\nu 2^{\omega(s)}$ and $i_H(r(H)Ps)\leq 2^{2+\omega_3(r(H))}$ for any $P\in\mathcal P$ and any squarefree $s$ coprime to $2r(H)P$, so that
\begin{align*}
\sum_{\substack{C\in\Cl(H)\\D(C)\leq \kappa_H^3 X\\(s(C),k_i)\neq 1\forall i}}1
&\ll2^{\omega_3(r(H))}\left(2\cdot\frac{\log M}{\log b}\right)^\nu\sum_{\substack{s\in S(1,\kappa_H^3 X/(r(H)b^\nu))}}2^{\omega(s)},\\
&\ll2^{\omega_3(r(H))}\left(2\cdot\frac{\log M}{\log b}\right)^\nu\left[X\log X\cdot \frac{1}{r(H)b^\nu}\cdot\kappa_H^3\log\kappa_H\right],\end{align*}
where the implied constants are independent of $X$ and $H$. Combining these bounds, we obtain the desired result.
\end{proof}

\begin{remark}
Under the hypothesis of Prop.~\ref{mechanical}, we can also derive the bound
\begin{multline}
\frac{1}{X\log X}\textrm{\textup{MAR}}_H(X)\ll  \frac{2^{\omega_3(r(H))}}{r(H)} \\\times  \left[\sum_{m=1}^\nu k_m^3\log k_m  \left(\prod_{i=1}^{m-1}\sum_{p\mid k_i}\frac{1}{p}\right)  
 +   \kappa_H^3\log\kappa_H\left(\prod_{i=1}^{\nu}\sum_{p\mid k_i}\frac{1}{p}\right)\right],\end{multline}
which may be of independent interest.
\end{remark}

The existence in Prop.~\ref{mechanical} of the required  $k_1, k_2, \ldots ,k_{\nu}$, with appropriate properties, follows from an application of a one dimensional  lower-bound Brun sieve given as follows.

\begin{proposition}\label{brun}\ 
Given $\alpha > \frac{1}{4}$, there are constants $\delta=\delta(\alpha)$ and $\delta'=\delta^{(1)}(\alpha)$ such that for any odd squarefree $q=p_1\ldots p_l$ and $a_1, \ldots,\, a_l$ all coprime to $q$, one has:

For $y\ge q^{\alpha}$, and $y$ large enough in terms of $\alpha$ alone, there are squarefree integers $k_1,\ldots ,\, k_{\nu}$ all pairwise coprime and satisfying 
\[
\left(\frac{k_ja_1}{p_1}\right)=1,\ \left(\frac{k_ja_2}{p_2}\right)=1,\  \ldots,\ \left(\frac{k_ja_l}{p_l}\right)=1, \quad \text{for}\  j=1, \ldots, \nu\, ,
\]
 with
\[ k_j \leq y, \quad \text{and} \quad \nu = \lfloor{y^{\delta'}\rfloor},
\]
and such that if $p \! \mid \! k_j$, then $p\geq y^{\delta}$.
\end{proposition}
\begin{proof}

 The proof combines the method in Sec.~\ref{Watson} together with a lower-bound Brun sieve.

For $m\geq 1$, put
\[ b_m =2^{-l}\left[1+ \left(\frac{a_1m}{p_1}\right)\right]\left[1+ \left(\frac{a_2m}{p_2}\right)\right]\ldots \left[1+ \left(\frac{a_lm}{p_l}\right)\right],
\]
if $(m,q)=1$, and zero otherwise. Thus $b_m$ is zero unless 
\begin{equation}\label{br1}
\left(\frac{a_jm}{p_j}\right)=1 \quad \text{for all}\ \  1\leq j\leq \nu.
\end{equation}  
We seek many values of $m$ that satisfy \eqref{br1}, for which we consider the sieved sums
\[
\mathcal{S}(\mathcal{B},P'_z,y) = \sum_{\substack{m\leq y \\ (m,P'_z)=1}} b_m,
\]
where $\mathcal{B}$ denotes the sequence $m$ with $b_m\neq 0$, and $P'_z=\prod_{\substack{p\leq z , p\nmid q}}p$.

To apply the sieve, we need to study
\[ \sum_{\substack{ m\leq y \\ m \equiv 0 \Mod d }} b_m,\quad \text{for}\   (d,q)=1,
\]
for which we have

\[
\begin{aligned} \sum_{\substack{m\leq y \\ d\mid m }} b_m &= \sum_{\substack{m\leq y \\ d\mid m  \\ (m,q)=1 }}2^{-l}\left[1+ \left(\frac{a_1m}{p_1}\right)\right]\left[1+ \left(\frac{a_2m}{p_2}\right)\right]\ldots \left[1+ \left(\frac{a_lm}{p_l}\right)\right],\\
&= 2^{-l}\sum_{\substack{m\leq y \\ d\mid m }}\sum_{d_1|(q,m)}\mu(d_1)\left[1+ \left(\frac{a_1m}{p_1}\right)\right]\ldots \left[1+ \left(\frac{a_lm}{p_l}\right)\right],\\
&= 2^{-l} \sum_{d_1|q}\mu(d_1)\sum_{\substack{d_1m_1\leq y \\ d\mid m_1  }}\left[1+ \left(\frac{a_1d_1m_1}{p_1}\right)\right]\ldots \left[1+ \left(\frac{a_ld_1m_1}{p_l}\right)\right],\\
&=2^{-l} \sum_{d_1|q}\mu(d_1)\sum_{\substack{d_1dm_2\leq y}}\left[1+ \left(\frac{a_1d_1dm_2}{p_1}\right)\right]\ldots \left[1+ \left(\frac{a_ld_1dm_2}{p_l}\right)\right],
\end{aligned}
\]
since $(d_1,d)=1$. Expanding the product on the right, we have
\[
\sum_{\substack{m\leq y \\ d\mid m }} b_m = 2^{-l}\sum_{d_1|q}\mu(d_1)\sum_{\substack{\mathcal{V}\subset \{1,\ldots,l\}\\ \mathcal{V}=\{j_1,\ldots,j_s\} \\ p_{j_i}\nmid d_1}}\sum_{\substack{m_2\leq \frac{y}{d_1d}}}\left(\frac{a_{j_1}d_1dm_2}{p_{j_1}}\right)\ldots  \left(\frac{a_{j_s}d_1dm_2}{p_{j_s}}\right).
\]
When $\mathcal{V}$ is the empty set, we get a contribution of 
\[
2^{-l}\sum_{d_1|q}\mu(d_1)\left(\frac{y}{dd_1} + O(1)\right) = 2^{-l}\frac{\phi(q)}{q}\frac{y}{d} + O(1).
\]
We write
\[
\sum_{\substack{m\leq y \\ d\mid m}} b_m  = \frac{Y}{d} + R(d,\mathcal{B},y), \quad \text{with}\ Y=2^{-l}\frac{\phi(q)}{q}y\ .
\]
Applying Burgess' estimate to the sum over $m_2$ (see Prop.~\ref{propC1}, with $r\to \infty$) we get the bound
\[
 R(d,\mathcal{B},y) \ll \left(\frac{y}{d}\right)^{1-\frac{1}{r}}q^{\frac{r+1}{4r^2}}(yq)^{\ve}.
\]
Then, one has 
\[
\sum_{\substack{d\leq D \\ (d,q)=1}} |R(d,\mathcal{B},y)| \ll y^{1-\frac{1}{r}}D^{\frac{1}{r}}q^{\frac{r+1}{4r^2}}(yq)^{\ve}.
\]
If $y\geq q^{\alpha}$ with $\alpha>\frac{1}{4}$, then there are positive constants $C_1$, $C_2$ and $C_3$ depending at most on $\alpha$  such that with $r=C_1\ve^{-\half}$ and $\sigma =C_2\ve^{\half}$ we have 
\[\sum_{\substack{d\leq D \\ (d,q)=1}} |R(d,\mathcal{B},y)|\ll y^{1-C_3\ve}\quad  \text{for}\ \  D\ll (yq^{-\frac{1}{4}})^{1-\sigma}.\] This, together with a 1-dimensional Brun lower-bound sieve yields:\ 
 
for $\alpha>\frac{1}{4}$, there is a $\delta=\delta(\alpha) >0$ and a $C_{\alpha}>0$ such that
\[
\mathcal{S}(\mathcal{B},P'_z,y) \geq C_{\alpha}\frac{Y}{\log z}\quad  \text{for} \ \  z=y^{\delta}.
\]
In particular, there are at least 
\[C_{\alpha}2^{-l}\frac{\phi(q)}{q}\frac{y}{\delta\log y},
\]
solutions $m$ to \eqref{br1} with $m\leq y$, and if $p \! \mid \! m$, then $p\geq y^{\delta}$. Given this $\alpha$ and $\delta$, it follows that the number of solutions $m$ to \eqref{br1} is $\gg \delta^{-1}x^{1-\frac{\delta}{2}}$ for $x\geq \max(q^\alpha,C_{\delta})$ for some $C_{\delta}$ depending at most on $\delta$.

Now let $k_1$ be the least solution $m$ as above. Then, $k_1$ is squarefree, $(k_1,q)=1$, $k_1\leq y$ and if $p \! \mid \! k_1$, then $p\geq y^{\delta}$. This implies that there are at most $\delta^{-1}$ such prime factors $p$. The union of the at most $\delta^{-1}$ progressions $mp\leq y$ consists of at most $\delta^{-1}y^{1-\delta}$ of our solutions, so that we remove them and still have at least $\delta^{-1}y^{1-\frac{\delta}{2}} - \delta^{-1}y^{1-\delta}$ remaining. We now choose $k_2$ to be the least of these remaining solutions; it being squarefree and coprime to $k_1$. We now remove all multiples $pm\leq x$ with $p \! \mid \! k_2$, leaving us at least $\delta^{-1}y^{1-\frac{\delta}{2}} - 2\delta^{-1}y^{1-\delta}$ solutions  from which we choose $k_3$, squarefree, with $(k_3,k_2)=(k_3,k_1)=1$. Repeating in this way $\nu$-times with $\nu = \lfloor y^{\frac{\delta}{2}}\rfloor$ gives the result in the lemma.
\end{proof}

\begin{remark}
One may extend Prop.~\ref{brun} to include the prime $p=2$ by replacing $q$  with $2q$, and supplement with an odd integer $a_0$ with minor modifications. Then the definition of $b_m$ is modified to
\[ b_m =2^{-l-2}\left[1+ \left(\frac{2}{a_0m}\right)\right]\left[1+ \left(\frac{-2}{a_0m}\right)\right]\left[1+ \left(\frac{a_1m}{p_1}\right)\right]\ldots \left[1+ \left(\frac{a_lm}{p_l}\right)\right],
\]
so that $b_m=0$ unless $(m,2q)=1$. The rest of the argument follows without much change, with Burgess' estimate Prop.~\ref{propC1} applied with a character modulo $8q$. The integers $k_i$ obtained will be odd and squarefree.
\end{remark}

The following provides the key integrable dominating function of $H$.

\begin{proposition}\label{martini-tail}
For every root packet $H\in\mathfrak R$ and every $X\geq 2$,
\[\frac{1}{X\log X}\textrm{\textup{MAR}}_X(H)\ll \frac{2^{\omega_3(r(H))}\log^+\!r(H)}{r(H)^{9/16}},\]
where the implied constant is independent of $X$ and $H$.
\end{proposition}
 
\begin{proof} 
Let us choose $\alpha=7/24$. Then there exist $\delta(\alpha)>0$, $\delta^{(1)}(\alpha)>0$, and $B(\alpha)>1$ (if $y\geq B(\alpha)$ then $y$ is ``large enough in terms of $\alpha$ alone'') so that the conclusion of Prop.~\ref{brun} can be applied. Let $c\geq1$ be a constant such that \[\kappa_H=\max \{\kappa(C):C\in\Cl(H)\}\leq c r(H)^{\frac{7}{2}}, \]
for every $H\in\mathfrak R$; such $c$ exists by Prop.~\ref{kappa-finite}. Let $A\geq1$ be a large enough constant determined solely by $\delta$, $\delta'$, and $c$ chosen so that the inequality
\[(cr^{\frac{7}{2}})\log (cr^{\frac{7}{2}})\leq\left(\frac{\delta(A B(\alpha) r^{7/48})^\delta}{2}\right)^{\lfloor (A B(\alpha) r^{7/48})^{\delta'}\rfloor} ,\]
holds for all $r\geq1$.

Putting $Y(H)=AB(\alpha) r(H)^{\frac{7}{48}}$, we can then find $k_1,\dots,k_\nu$ as in Prop.~\ref{brun} where $\nu=\lfloor Y(H)^{\delta'}\rfloor$ and $k_i\leq Y(H)$ for all $i=1,\dots,\nu$ such that $Y(H)^{\delta}\leq p$ for all primes $p \! \mid \!  \prod_{i=1}^\nu k_i$. Using the notation of Prop.~\ref{mechanical}, since $M\leq Y(H)$ and $b\geq Y(H)^{\delta}$ we have
\begin{align*}
\frac{1}{X\log X}\textrm{\textup{MAR}}_X(H)&\ll \frac{2^{\omega_3(r(H))}}{r(H)}\left[M^3\log M+\kappa_H^3\log\kappa_H\cdot\left(\frac{\log M}{\log b}\frac{2}{b}\right)^\nu\right],\\
&\ll \frac{2^{\omega_3(r(H))}}{r(H)}\left(Y(H)^3\log Y(H)+1\right),\\
&\ll \frac{2^{\omega_3(r(H))}\log r(H)}{r(H)^{9/16}},
\end{align*}
for all $H\in\mathfrak R$ and all $X\geq2$, with the implied constants independent of $H$ and $X$.
\end{proof}

The next proposition establishes the pointwise limits as $X \to \infty$.

\begin{proposition}\label{avg-martini-pointwise}
For a fixed  $H\in\mathfrak R$, there is a constant $M(H)>0$ such that, as $X\to\infty$,
\[\textrm{\textup{MAR}}_{X}(H)\sim M(H)\cdot X\log X.\]
\end{proposition}

\begin{proof}
By Prop.~\ref{kappa-finite}, the set $L(H)=\{\kappa(C):C\in\Cl_H\}$ is finite. Under the notation of Def.~\ref{lambda-def}, we have
\[\textrm{\textup{MAR}}_X(H)=\sum_{0\neq t\in L(H)}\sum_{1\leq d\leq t^3 X}\lambda_H(d,t).\]
Let $L^g(H)=\{K(G):G\in\Gen_H\}\subseteq L(H)$. We note that if $t\in L(H)\setminus L^g(H)$ then
\[\sum_{1\leq d\leq t^3X}\lambda_H(d,t)=\sum_{1\leq d\leq t^3X}\lambda_H^{(2+)}(d,t)=O_{H,t}(X)\]
by Prop.~\ref{lambdas}.(3).
Thus, fixing $0\neq t\in L^g(H)$, it suffices to show there is a constant $c_{H,t}>0$ such that
\[\sum_{1\leq d\leq Y}\lambda_H(d,t)\sim c_{H,t}\cdot Y\log Y.\]
By parts (1) and (3) of Prop.~\ref{lambdas}, it is sufficient to show that there is a constant $c_{H,t}>0$ such that
\[\sum_{1\leq d\leq Y}\lambda_H^g(d,t)\sim c_{H,t}\cdot Y\log Y.\]
Using the notation of Prop.~\ref{lambdas}.(2), we have
\[\left|\lambda_H^g(d,t)-\sum_{H'\in\mathcal F(H,t)}g_{_{H'}}(d)\right|\leq 1,\]
where $\mathcal F(H,t)=\{H'\in\Packet_H:\Supp(H')=T(H),K^*(H')=t\}$ and where $T(H)=\{p:2\leq p\leq r(H)\}$ is defined as in Corollary~\ref{minimal-ripe-new}. Since $\mathcal F(H,t)$ is finite, by Prop.~\ref{avg-class-pointwise} it follows that
\begin{align*}
\sum_{1\leq d\leq Y}\lambda_H^g(d,t)&=\sum_{H'\in\mathcal F(H,t)}\left(\sum_{1\leq d\leq Y}g_{_{H'}}(d)\right) +O(Y),\\
&\sim \left(\sum_{H'\in\mathcal F(H,t)}A(H')\right)Y\log Y,
\end{align*}
as desired. Putting these results together, we have
\[M(H)=\sum_{0\neq t\in L^g(H)}t^3\sum_{H'\in\mathcal F(H,t)}A(H')\ ,\]
where 
\[A(H')=\frac{1}{ 2^{|\Supp(H')|+2}}\cdot\frac{1}{D(H')}\prod_{p\in\Supp(H')}\left(1-\frac{1}{p}\right)^2\prod_{p\notin\Supp(H')}\left(1+\frac{2}{p}\right)\left(1-\frac{1}{p}\right)^2.\]
Thus, we conclude that
\begin{equation}\label{MH} M(H)=\frac{c}{4\cdot 2^{|T(H)|}}\prod_{p\in T(H)}\left(1+\frac{2}{p}\right)^{-1}\sum_{\substack{H'\in\Packet_H,\\\Supp(H')= T(H)}}\frac{K^*(H')^3}{D(H')},\end{equation}
where $c=\prod_{p}(1+\frac{2}{p})(1-\frac{1}{p})^2$.
\end{proof}

We conclude this section with a proof of Theorem \ref{IThm1} asserting that there is a constant $\gamma >0$ such that, as $X\to\infty$,
\[\textrm{\textup{MAR}}(X)\sim \gamma \cdot X\log X.\]

\begin{proof}[Proof of Theorem~\ref{IThm1}]\

Recall that  the function $f_X^{({\bf C})}:\mathfrak R\to\R_{\geq0}$  was defined in Proposition \ref{mechanical} by
\[f_X^{({\bf C})}(H)=\frac{1}{X\log X}\textrm{\textup{MAR}}_X(H).\]
Let $N:\mathfrak R\to\R_{\geq0}$ be given by $N(H)=2^{\omega_3(r(H))}\log^+ r(H)/r(H)^{9/16}$.
We note that
\[\sum_{H\in\mathfrak R}N(H)\leq\sum_{r\in |R|}\frac{2^{\omega_3(r)}\cdot \log^+ r\cdot 4\cdot 2^{\omega(r)} \prod_{p\mid r}v_p(r)}{r^{9/16}}<\infty,\]
where $|R|=\{r(d):d\in\Z_{\geq1}\}$. So $N\in L^1(\mathfrak R)$ because the numerator in each summand of the second series is $O_\epsilon(r^\epsilon)$. It follows by Prop.~\ref{martini-tail} that $f_X^{({\bf C})}(H)\ll N(H)$
valid for all $H\in\mathfrak R$ and all $X\geq 2$, where the implied constant is independent of $H$ and $X$. Thus $f_X^{({\bf C})}\in L^1(\mathfrak R)$ for every $X\geq2$. By the dominated convergence theorem the pointwise limit $f^{({\bf C})}=\lim_{X\to\infty} f_X^{({\bf C})}$ lies in $L^1(\mathfrak R)$, and
\begin{equation}\label{mar-asym-form}\lim_{X\to\infty}\frac{1}{X\log X}\textrm{\textup{MAR}}(X)=\lim_{X\to\infty}\sum_{H\in\mathfrak R}f_X^{({\bf C})}(H)=\sum_{H\in\mathfrak R}f^{({\bf C})}(H)<\infty,\end{equation}
giving the desired result.
\end{proof}
\begin{remark}
    It follows from the proof of Prop.~\ref{martini-tail} that
    \[\gamma=\frac{c}{4}\sum_{H\in\mathfrak R}\frac{1}{2^{|T(H)|}}\prod_{p\in T(H)}\left(1+\frac{2}{p}\right)^{-1}\left(\sum_{\substack{H'\in\Packet_H,\\\Supp(H')= T(H)}}\frac{K^*(H')^3}{D(H')}\right)
    \]
    where $c=\prod_{p}(1+\frac{2}{p})(1-\frac{1}{p})^2$ and $T(H)=\{p:2\leq p\leq 2r(H)\}$.
\end{remark}
\begin{remark}\label{rmk-gamma}
That $\gamma>\frac{19}{20}\frac{\zeta(2)}{\zeta(4)}$ can be seen as follows. Recall that
\[\textup{MAR}(X)=\sum_{t=1}^\infty\mathcal S(t^3X,t),\quad\text{and}\quad \sum_{1\leq d\leq X}h(d)=\sum_{t=0}^\infty\mathcal S(X,t)\]
where $\mathcal S(X,t)=\sum_{1\leq d\leq X}\lambda(d,t)$ and $\lambda(d,t)=\#\{C\in\Cl(d):\kappa(C)=t\}$.
It is easy to show (cf.~proof of Prop.~\ref{avg-martini-pointwise}) that
\[\#\{C\in\Cl_H:1\leq D(C)\leq X,C\text{ isotropic}\}=O_H(X)\]
for each root packet $H$, so that $\mathcal S(X,0)=o(X\log X)$. It follows that
\begin{align*}
    \textup{MAR}(X)-\sum_{1\leq d\leq X}h(d)&=\sum_{t\geq2}(\mathcal S(t^3X,t)-\mathcal S(X,t))+o(X\log X)\\
    &=\sum_{t\geq2}\sum_{X<d\leq t^3X}\lambda(d,t)+o(X\log X).
\end{align*}
It follows that, by arguing as in the proof of Prop.~\ref{avg-martini-pointwise}, any occurrence of a root packet $H$ with $K(H)\geq2$ (of which there are many) will contribute a positive multiple of $X\log X$ to the right hand side. Alternatively, we may bound the last sum from below by
\begin{align*}
   \sum_{\substack{X<d\leq 8X\\d\text{ squarefree}}}\lambda(d,2)+o(X\log X)=\mathcal S^\#(8X,2)-\mathcal S^\#(X,2)+o(X\log X)\sim \frac{7}{8}cX\log X
\end{align*}
where $\mathcal S^\#(X,t)$ and $c>0$ are given as in the proof of Prop.~\ref{Mar-sf}. The desired lower bound on $\gamma$ then follows by comparing the above with formulas $\bf{(B)}$ and $\bf{(C)}$.
\end{remark}


\section{{\bf Part 2} : Isotropic forms}\label{serre-p}\

We turn to the proof of Theorem~\ref{IThm2}.  The first step is to use homogeneous dynamics to localize the counting of isotropic forms and to reduce the problem to the asymptotics ${\bf (D)}$.    There are subtle global normalizations that enter in the analysis so we begin by using this method to count all the indefinite integral forms and use Siegel's
\eqref{siegel-sum} to compute the constants.  With the normalizations in place we proceed to the isotropic counting problem which requires a direct evaluation of  ${\bf (D)}$.  That one can do so in terms of a product of local densities comes from special and fortuitous properties of isotropic forms (see Prop.~\ref{mult}).

In the remainder of this paper, let $\Sym\simeq\mathbb{A}^6$ be the affine space of $3\times 3$ symmetric matrices. Let $V^{+}(\mathbb{R}) \subset \Sym(\R)$ be the cone of the real symmetric $3\times3$ matrices of positive determinant and which are indefinite over $\mathbb{R}$, and let
\begin{equation} \label{81}
V^{+}_1(\mathbb{R}) := \{F \in V^{+}(\mathbb{R}) \, : \, D(F)=1 \}.
\end{equation}
$V^{+}_1(\mathbb{R})$ is a homogeneous space for the action of $\SL_3(\mathbb{R})$  given by $F \to g F g^t$.\footnote{In this section we consider equivalence in the ``narrow'' sense, working with $\SL_3(\Z)$ rather than $\GL_3(\Z)$. Note that the central element $-I$ has determinant $-1$, so this notion of equivalence agrees with the wider notion used in the rest of the paper.}    It  carries a (unique up to scalar multiple) $\SL_3(\mathbb{R})$ invariant measure $\mu_1$ which we normalize so that 
\begin{equation} \label{82}
d \lambda = d \mu_1 t^2 \frac{dt}{t},  
\end{equation}
where $d \lambda$ is the restriction  to $V^{+}(\mathbb{R})$ of Lebesgue measure on $\Sym(\mathbb{R})$, normalized to have volume one on the unit cube with $t=\det(F)$.

For $F\in V^{+}(\mathbb{R})$ denote by $\tilde{F}$ the projection of $F$ onto $V_1^{+}(\mathbb{R})$
\begin{equation} \label{83}
\tilde{F}=\frac{F}{(D(F))^{\frac{1}{3}}}.
\end{equation}

Fix a nice compact subset $\Omega$ of $V_1^{+}(\mathbb{R})$ (say the closure of a small neighborhood of a point) and for $C$ a class of integral forms\footnote{From here until \eqref{812} we allow imprimitive classes $C$.} in $V^{+}(\mathbb{R})$ set 
\begin{equation}\label{84}
N_{\Omega}(C):=\sum_{\substack{F\in C\\ \tilde{F}\in\Omega}}1=\sum_{\substack{\gamma \in \SL_3(\mathbb{Z}) / \text{Aut}_{F_0}(\mathbb{Z})\\ \widetilde{\gamma F_0 \gamma^{t}}\in \Omega}}1.
\end{equation}
Here $F_0$ is any member of $C$ and $\text{Aut}_{F_0}(\mathbb{Z}) = \{ \gamma \in \SL_3(\mathbb{Z}) : \gamma F_0 \gamma^{t} =F \}$,   which is a lattice in $H_{F_0}(\mathbb{R}) := \text{Aut}_{F_0}(\mathbb{R})  \simeq  \mathrm{SO}(2,1)$.

Now Theorem~$1.2$, together with Lemma~$2.2$ and Prop.$2.2$ of Eskin-Oh \cite{eskinoh}, and the results of Oh~\cite{heeoh} on symmetric matrices assert that these ``periodic'' normalized $H_{F_0}(\mathbb{R}) / H_{F_0}(\mathbb{Z})$ cycles equidistribute in $\SL_3(\mathbb{R})/ \SL_3(\mathbb{Z})$ as $D(F_0) \to \infty$. This in turn implies \cite{eskinoh} that as $D(C) \to \infty$ 
\begin{equation} \label{85}
N_{\Omega}(C) \sim \eta_1 \mu_1(\Omega) \frac{\text{vol}(H_{F_0}(\mathbb{R}) / H_{F_0}(\mathbb{Z}))}{\text{vol}(\SL_3(\mathbb{R})/ \SL_3(\mathbb{Z}))},
\end{equation}
where $\eta_1$  is a normalized constant relating our $\mu_1$ to the invariant Haar measure on $V_1^{+}(\mathbb{R})$ obtained from the Fubini formula (1.5) in \cite{eskinoh}, and where the Haar measures on $H_{F_0}(\mathbb{R})$ and $\SL_3(\mathbb{R})$ are normalized as in Siegel.  

The normalization constant $\eta_1$ can be computed but we prefer to determine the normalization directly with Siegel's $\rho(C)$ defined in \eqref{rho1}.   Using Siegel (page 121 \cite{siegelIAS}) relating $\rho(C)$ to  $\text{vol}(H_{F}(\mathbb{R}) / H_{F}(\mathbb{Z}))$, we see from \eqref{85} that as $D(C)\to \infty$ 
\begin{equation} \label{86}
N_{\Omega}(C) = \eta \mu_1(\Omega) D(C)^2 \rho(C) (1+o(1)), 
\end{equation}
where $\eta$ is a constant.

Let $W_{\Omega}(N)$ be the Euclidean cone subtended by $\Omega$
\begin{equation} \label{87}
W_{\Omega}(N) =\{F \in W^{+}(\mathbb{R}) \, : \, D(F) \leq N, \tilde{F} \in \Omega \}.
\end{equation}
We have normalized $\mu_1$ so that 
\begin{equation} \label{88}
\text{vol}(W_{\Omega}(N))=\mu_1(\Omega) \int_{0}^{N} t dt = \frac{\mu_1(\Omega) N^2}{2}.
\end{equation}
Now from \eqref{86} as $N \to \infty$
\begin{equation}\label{89}
\sum_{\substack{F \, \text{integral}\\ F \in W_{\Omega}(N)}}1=\sum_{\substack{1\leq D(C) \leq N}} N_{\Omega}(C)=\eta \mu_1(\Omega) \sum_{\substack{1\leq D(C) \leq N}}D(C)^2 \rho(C) (1+o(1)).
\end{equation}
On the other hand, the left-hand side of \eqref{89} is elementarily \begin{equation} \label{810}
\sim \text{vol}(W_{\Omega}(N)) =\frac{\mu_1(\Omega) N^2}{2}.\end{equation}
Hence
\begin{equation}\label{811}
\eta \sum_{\substack{1\leq D(C)\leq N}} D(C)^2 \rho(C) \sim \frac{N^2}{2}.
\end{equation}
From Siegel's   \eqref{siegel-sum}  (applied to positive $D$'s) we find that 
\begin{equation} \label{812}
\eta=\frac{2}{\zeta(2) \zeta(3)},
\end{equation}
which normalizes \eqref{86} for our application to isotropic forms. 

To that end we have that as $N \to \infty$, and restricting ourselves again to primitive $C$'s,
\begin{equation} \label{813}
\sum_{\substack{F\in W_{\Omega}(N)\\ F \text{integral} \\F \text{isotropic}}}1=\sum_{\substack{C\in \Cl^{\iso}\\D(C)\leq N}}N_{\Omega}(C)=\frac{2}{\zeta(2)\zeta(3)}\mu_1(\Omega) \sum_{\substack{C\in \Cl^{\iso}\\D(C)\leq N}} D(C)^2 \rho(C) (1+o(1)).
\end{equation}

This brings us to  ${\bf (D)}$.  In the next section we use the Siegel mass formula to evaluate the right-hand side of \eqref{813}:
\begin{equation} \label{814}
\sum_{\substack{C\in \Cl^{\iso}\\D(C)\leq N}} D(C)^2 \rho(C) \sim \frac{\zeta(2)\zeta(3)}{2}\varpi\cdot\frac{1}{2}\frac{N^2}{\sqrt{\log N}}, \end{equation}
where $\varpi$ is an infinite product of local densities.  Hence as $N \to \infty$,
\begin{equation} \label{815}
\sum_{\substack{F\in W_{\Omega}(N)\\ F \text{isotropic}}}1 \sim \varpi \frac{\text{vol}(W_{\Omega}(N))}{\sqrt{\log N}} \sim \varpi\int_{W_{\Omega}(N)} \frac{dx}{\sqrt{\log^{+}D(\mathbf{x})}},
\end{equation}
where 
\[
D(\mathbf{x})=D(x_1,\ldots,x_6) = \det \begin{psmatrix}
		x_1 & x_2 & x_3\\
		 x_2 & x_4 &  x_5	\\
		x_3 & x_5& x_6	
		\end{psmatrix}.
\]

This proves Theorem~\ref{IThm2} for the cones $W_{\Omega}(N)$.  The passage from \eqref{815} to Theorem~\ref{IThm2} follows the approximation of sets in $W^+(\mathbb{R})$ as outlined in \cite{sarnakletter} and executed in detail in \cite{KotWoo} on a similar problem.  It involves truncating small sets near the locus $D(F)=0$, for which upper bound sieves suffice.  In our present setting we use Serre's upper bound sieve.  Since the details are the same as in \cite{KotWoo} we omit them here.


\section{Isotropic Siegel asymptotics}\label{sie-asy}\

In this section we prove {\bf(D)} directly, which together with \eqref{813} yields Theorem \ref{IThm2}.  In what follows we write $\zeta_p(s)=(1-1/p^s)^{-1}$ for convenience. We begin by examining the masses associated with $\rho(C)$ for isotropic $C$'s, and especially the multiplicativity in $d$ of the corresponding function $\nu^{\iso}(d)$ below. Here and elsewhere we make crucial use of the mass computations by Conway and Sloane \cite{CSMass}, and the table of masses at $p=2$ in Appendix \ref{2classes}.

\begin{definition}
Let $\Sym$ be the space of symmetric $3\times 3$ matrices. Let $p$ be prime.
\begin{enumerate}[label=\upshape\arabic*.\ ]
	\item Let $\Sym^{\iso}(\Z_p)\subseteq\Sym(\Z_p)$ be the locus of isotropic matrices.
	\item Let $\Sym_{\prim}(\Z_p)\subseteq\Sym(\Z_p)$ be the locus of nonsingular primitive matrices.
	\item We write $\Sym_{\prim}^{\iso}(\Z_p)=\Sym^{\iso}(\Z_p)\cap \Sym_{prim}(\Z_p)$.
    \item For each $d\in\Z_p/(\Z_p^*)^2$ we write $\Sym(d)=\{A\in\Sym:\det(A)\in d\}$ and
	\[\Sym_{\prim}^{\iso}(d)(\Z_p)=\Sym_{\prim}^{\iso}(\Z_p)\cap\Sym_{\prim}^{\iso}(d)(\Z_p).\]
    \item We write $\mu_p$ for the standard $p$-adic probability measure on $\Sym(\Z_p)\simeq\Z_p^{6}$. That is, for each $S\subseteq\Sym(\Z_p)$ we set
\[\mu_p(S)=\lim_{r\to\infty}\frac{|\{a\textup{ mod }p^r:a\in S\}|}{|\Sym(\Z/p^r\Z)|}.\]
\end{enumerate}
\end{definition}
\begin{remark}
    If $G$ is the $\GL(3,\Z_p)$-orbit of $A\in\Sym(\Z_p)$ then
\begin{align*}
\mu_p(G)&=\lim_{r\to\infty}\frac{|\GL(3,\Z/p^r\Z)|}{|\Aut_A(\Z/p^r\Z)|}\frac{1}{|\Sym(\Z/p^r\Z)|}=\frac{1}{\prod_{n=1}^3\zeta_p(n)}\lim_{r\to\infty}\frac{p^{3r}}{|\Aut_A(\Z/p^r\Z)|}.
\end{align*}
\end{remark}
\begin{definition}
Let $A\in\Sym(\Z_p)$. Let $G$ be its $\GL(3,\Z_p)$-orbit.
\begin{enumerate}[label=\upshape\arabic*.\ ]
    \item The \emph{Siegel $p$-density} of $A$ is $\delta_p(A)=\prod_{n=1}^3\zeta_p(n)\cdot\mu_p(G)$.
    \item The \emph{$p$-mass} of $A$ is $m_p(A)=p^{2v_p(\det(A))}\delta_p(A)$.
\end{enumerate}
\end{definition}

\begin{remark}
The above terminology for the  $p$-mass was introduced by Conway-Sloane \cite{CSMass} for quadratic forms in any number of variables. It is a modification of Siegel's local density; it depends on the $\GL(n,\Zp)$-equivalence class of $A$ and is also invariant under scaling of $A$. As such we may assume that $A$ is primitive when using it in computations.
\end{remark}

In order to analyze the sum \eqref{813} we use Siegel's mass formula which localizes the $\rho$ sum:

\begin{theorem}\label{massform}\textup{(\emph{The Siegel Mass Formula} \cite{siegel,siegelIAS})}  \\ 
If $G\in\Gen$, one has 
\[\nu(G) :=\sum_{\Ge(C)=G}\rho(C)=2\prod_{p}2\delta_p(G_p).\]\qed
\end{theorem}

We shall denote by $\Gen^{\iso}$ the subset of isotropic genera in $\Gen$, and also write $\Gen^{\iso}(d)=\Gen^{\iso}\cap\Gen(d)$. In what follows, we shall write

\begin{align}\label{H-iso}
 \nu^{\iso}(d)& :=\sum_{G\in\Gen^{\iso}(d)}\nu(G)=\sum_{C\in\Cl^{\iso}(d)}\rho(C),\\     
\intertext{and}
          \quad & \tilde{\nu}^{\iso}(d) :=\frac{\nu^\iso(d)}{\nu^\iso(1)}. \nonumber
\end{align}

Theorem \ref{massform} allows us to prove the following multiplicativity:

\begin{proposition}\label{mult}\ 
We have
\[\sum_{d=1}^\infty \frac{\tilde{\nu}^{\iso}(d)}{d^{s}}=\prod_p\left(\sum_{k=0}^\infty\frac{\mu_p(\Sym_{\prim}^\iso(p^k)(\Z_p))}{\mu_p(\Sym_{\prim}^\iso(1)(\Z_p))}p^{-s}\right).\]
Equivalently, $\tilde{\nu}^{\iso}(d)$ is a multiplicative function of $d$, and for every prime $p$ and integer $k\geq0$ we have \[\tilde{\nu}^{\iso}(p^k)=\mu_p(\Sym_{\prim}^\iso(p^k)(\Z_p))/\mu_p(\Sym_{\prim}^\iso(1)(\Z_p)).\] In particular,
\[\sum_{k=0}^\infty \tilde{\nu}^{\iso}(p^k)=\frac{\mu_p(\Sym_{\prim}^\iso(\Z_p))}{|\Sq^*(\Z_p)|\mu_p(\Sym_{\prim}^\iso(1)(\Z_p))}.\]
\end{proposition}
\begin{proof}
Note first that $\nu^{\iso}(1)\neq0$. We show that $\nu^{\iso}$ is a nonzero scalar multiple of a multiplicative function.
Let $\Gen_p^{\iso}(d)$ denote the set of $\Zp$-equivalence classes of isotropic forms of determinant equal to $d\in\Zp/(\Zp^{*})^2$. Then the map $G\mapsto (G_p)_{p<\infty}$ gives us a bijection
\[\Gen^{\iso}(d)\simeq \prod_{p<\infty}\Gen_p^{\iso}(d).\]
It follows that
\begin{equation}\label{viso1}\nu^{\iso}(d)=\sum_{G\in\Gen^{\iso}(d)}2\prod_{p<\infty}2 \delta_p(G_p)=2\prod_{p<\infty} 2\delta_p(\Sym_{\prim}^{\iso}(d)(\Z_p)).\end{equation}
We claim that
$d\mapsto \delta_p(\Sym_{\prim}^{\iso}(d)(\Z_p))$
is a nonzero scalar multiple of a multiplicative function for each prime $p$. First, suppose that $d,e\geq1$ are positive integers such that $(e,p)=1$. Then
we have a bijection
\[\Gen_p^{\iso}(d)\simeq\Gen_p^{\iso}(de),\]
that preserves $\delta_p$-values, given by the scalar multiplication $f\mapsto e f$. It follows that $\delta_p(\Sym_{\prim}^{\iso}(d)(\Z_p))=\delta_p(\Sym_{\prim}^{\iso}(de)(\Z_p))$. Now, let us assume $a,b\geq1$ are positive integers such that $(a,b)=1$. Then without loss of generality $(b,p)=1$. Since we have $\delta_p(\Sym_{\prim}^{\iso}(1)(\Z_p))\neq0$, we see that
\[\frac{\delta_p(\Sym_{\prim}^{\iso}(ab)(\Z_p))}{\delta_p(\Sym_{\prim}^{\iso}(1)(\Z_p))}=\frac{\delta_p(\Sym_{\prim}^{\iso}(a)(\Z_p))}{\delta_p(\Sym_{\prim}^{\iso}(1)(\Z_p))}\cdot\frac{\delta_p(\Sym_{\prim}^{\iso}(b)(\Z_p))}{\delta_p(\Sym_{\prim}^{\iso}(1)(\Z_p))},\]
showing that $\delta_p(\Sym_{\prim}^{\iso}(-)(\Z_p))$ is a scalar multiple of a multiplicative function. Since $\delta_\ell(\Sym_{\prim}^\iso(p^k)(\Z_\ell))=\delta_\ell(\Sym_{\prim}^\iso(1)(\Z_\ell))$ for $(\ell,p)=1$, we observe that
\[
    \nu^{\iso}(p^k)=\nu^{\iso}(1)\cdot\frac{\delta_p(\Sym_{\prim}^{\iso}(p^k)(\Z_p))}{\delta_p(\Sym_{\prim}^{\iso}(1)(\Z_p))}=\nu^{\iso}(1)\cdot\frac{\mu_p(\Sym_{prim}^{\iso}(p^k)(\Z_p))}{\mu_p(\Sym_{\prim}^{\iso}(1)(\Z_p))},
\]
from which the desired result follows.
\end{proof}

In Section \ref{comp}, we shall give a detailed computation of an explicit formula for the Euler product
\[I(s)=\sum_{d=1}^\infty\frac{ d\, \tilde{\nu}^{\iso}(d)}{d^s}.\]
For now, combining Proposition \ref{mult} with elementary estimates of $p$-masses of isotropic forms (Lemma \ref{masslem} below), we establish the asymptotic formula {\bf(D)}.

\begin{proposition}\label{isotropic-siegel-2}
Let \[\varpi=\frac{2}{\Gamma(1/2)}\prod_p\mu_p(\Sym_{\prim}^{\iso}(\Z_p))\left(1-\frac{1}{p}\right)^{-1/2},\]
where $\mu_p(\Sym_{\prim}^{\iso}(\Z_p))$ is the $p$-adic probability measure of the locus of primitive isotropic matrices in the space of $3\times3$ symmetric $\Z_p$-matrices. Then, $0<\varpi <\infty$ and  as $X\to\infty$,
\[\sum_{\substack{C\in\Cl^{\iso}\\1\leq D(C)\leq X}}D(C)\rho(C)\ \sim\frac{\zeta(2)\zeta(3)}{2}\varpi\frac{X}{\sqrt{\log X}}.\]
\end{proposition}

\begin{proof}
We apply the Landau--Selberg--Delange method as in \cite{GranK}. As before, let us write $\tilde{\nu}^{\iso}(d)=\nu^{\iso}(d)/\nu^{\iso}(1)$, so that $\tilde{\nu}^{\iso}$ is a multiplicative function. By Lemma~\ref{masslem} below, \eqref{H-iso}, the prime number theorem, and Mertens' theorem, for every $A>0$ we have
\[\sum_{p\leq X\text{ odd}}p\,\tilde{\nu}^{\iso}(p)\log p=\sum_{p\leq X\text{ odd}}\frac{p+1}{2p}\log p=\frac{1}{2}X+O_A\left(\frac{X}{(\log X)^A}\right).\]
It can be seen that there exists $k>0$ such that $d\,\tilde{\nu}^{\iso}(d)\leq\tau_k(d)$ for all $d\geq1$, where $\tau_k(d)$ denotes the $d$-th coefficient of the Dirichlet series $\zeta(s)^k$. It then follows from \cite[Theorem~1, Sec.~7]{GranK} that
\begin{align*}
&\left(\frac{X}{\sqrt{\log X}}\right)^{-1}\sum_{1\leq d\leq X}d\,  \tilde{\nu}^{\iso}(d)\sim\frac{1}{\Gamma(\half)}\prod_p\left(\sum_{k=0}^\infty \tilde{\nu}^{\iso}(p^k)\right)\left(1-\frac{1}{p}\right)^{\half},\\
&\sim\frac{1}{\Gamma(\half)}\prod_p\frac{\mu_p(\Sym_{\prim}^{\iso}(\Z_p))}{|\Sq^*(\Z_p)|\mu_p(\Sym_{\prim}^{\iso}(1)(\Z_p))}\left(1-\frac{1}{p}\right)^{\half},\\
&\sim\frac{1}{\Gamma(\half)}\prod_p\frac{\mu_p(\Sym_{\prim}^{\iso}(\Z_p))}{|\Sq^*(\Z_p)|\delta_p(\Sym_{\prim}^{\iso}(1)(\Z_p))}\zeta_p(2)\zeta_p(3)\left(1-\frac{1}{p}\right)^{-\half},\\
&\sim\frac{1}{\nu^{\iso}(1)}\frac{\zeta(2)\zeta(3)}{\Gamma(\half)}\prod_p\mu_p(\Sym_{\prim}^{\iso}(\Z_p))\left(1-\frac{1}{p}\right)^{-\half},
\end{align*}
which yields the result.
\end{proof}

\begin{lemma}\label{masslem}
Let $G\in\Gen^{\iso}$.
\begin{enumerate}[label=\upshape\arabic*.\ ]
	\item Let $p$ be an odd prime.
	\begin{enumerate}[label=\upshape\alph*.\ ]
		\item If $v_p(G)=0$, then
		\[m_p(G)=\frac{1}{2}\left(1-p^{-2}\right)^{-1}.\]
		\item If $v_p(G)=1$, then
		\[m_p(G)=\frac{p}{4}\left(1-p^{-1}\right)^{-1}.\]
		\item If $v_p(G)\geq1$, then $m_p(G)\leq p^{v_p(G)}/2$.
	\end{enumerate}
	\item We have $m_2(G)\ll 2^{v_2(G)}$ where the constant is independent of $G$.
	\item We have $|D(G)|\nu(G)\ll1$ where the implied constant is independent of $G$.
\end{enumerate}
\end{lemma}

\begin{proof}
(1) and (2) follow from Prop.~\ref{massprop}, obtained by an examination of \cite{CSMass}, while (3) can be deduced from Theorem~\ref{massform} and the bounds obtained in parts (1) and (2).
\end{proof}

\begin{remark}\label{rem9.9}
The density constant $\varpi$ is given by a regularized product of local probabilities that a form is primitive and isotropic. We discuss this further in Remark~\ref{rem10.5} after evaluating these probabilities explicitly.
\end{remark}

\subsection{Alternative proof of the limit in Prop.~\ref{isotropic-siegel-2}}\label{isotropic-siegel-b}\ 

The proof of Prop.~\ref{isotropic-siegel-2}  exploited the particular feature of multiplicativity for  isotropic ternary forms stated in Prop.~\ref{mult}. We give here another proof using the method of packets (as was used to prove the asymptotics for the class number), as this may be useful in cases where such full multiplicativity fails. The method renders the constant as a sum over the packets.

\begin{proposition}\label{sqfree-mass}
Let $r\geq1$ be an integer. There is a constant $k(r)>0$ such that, as $X\to\infty$,
\[\sum_{\substack{1\leq s\leq X, (s,2r)=1 \\ s \textup{ squarefree}}}\prod_{p\mid s}\frac{1}{2}\left(1+\frac{1}{p}\right)\sim k(r)\cdot \frac{X}{\sqrt{\log X}}.\]
\end{proposition}

\begin{proof}
Let $\mathbb{P}_r$ denote the set of odd primes $p$ coprime to $r$. Then our sum can be written as $\sum_{1\leq n\leq X} a_n$ where the terms $a_n$ are defined by the Dirichlet series
\[\sum_{n=1}^\infty\frac{a_n}{n^s}=\prod_{p\in\mathbb{P}_r}\left(1+\frac{1}{2}\left(1+\frac{1}{p}\right)\frac{1}{p^s}\right).\]
By the prime number theorem and Mertens' theorem,
\[\sum_{p\leq X}a_p\log p=\sum_{\substack{p\leq X,\\p\in\mathbb P_r}}\frac{1}{2}\left(1+\frac{1}{p}\right)\log p=\frac{1}{2}X+O\left(\frac{X}{\log X}\right),\]
and moreover clearly $|a_n|\leq1$ for all $n$. The desired result follows by an application of  \cite[Theorem 1]{GranK}.
\end{proof}

\vspace{6pt}

\begin{corollary}\label{useful-2}
For all real numbers $A>0$ and $X\geq2$, we have
\[\sum_{\substack{1\leq s\leq AX\\\textup{$s$ odd}\\ \textup{ squarefree}}}\prod_{p\mid s}\frac{1}{2}\left(1+\frac{1}{p}\right)\ll\frac{X}{\sqrt{\log X}}A(\log^+\!A)^{1/2}, \]
where the implied constant is independent of $A$ and $X$.
\end{corollary}

\begin{proof}
Since the sum on the left hand side is empty if $AX<1$, we may assume without loss of generality that $AX\geq1$. By Prop.~\ref{sqfree-mass} with $r=1$, there is a constant $c\geq1$ such that for all $A>0$ and $X\geq2$ we have
\[\sum_{\substack{1\leq s\leq AX\\\textup{$s$ odd}\\  \textup{squarefree}}}\prod_{p\mid s}\frac{1}{2}\left(1+\frac{1}{p}\right)\leq c\cdot \frac{AX}{\sqrt{\log AX}}.\]
On the other hand,
\[\frac{\log AX}{\log X}\ll\log^+\!A,\]
for all $X\geq2$ and all $A>0$ with the implied constant independent of $A$ and $X$. Combining these two, we get the  result.
\end{proof}

\vspace{6pt}

\begin{proposition}\label{siegel-estimates}
Let $H\in\mathfrak R$ be the root packet of an isotropic genus.
\begin{enumerate}[label=\upshape\arabic*.\ ]
	\item There is a constant $S(H)>0$ such that, as $X\to\infty$,
\[\sum_{\substack{G\in\Gen^{\iso}_H\\1\leq D(G)\leq X}}D(G)\nu(G)\sim S(H)\frac{X}{\sqrt{\log X}}.\]
	\item For every $X\geq2$,
\[\sum_{\substack{G\in\Gen^{\iso}_H\\1\leq D(G)\leq X}}D(G)\nu(G)\ll \frac{X}{\sqrt{\log X}}\frac{1}{r(H)},\]
where the implied constant is independent of $X$ and $H$.
\end{enumerate}
\end{proposition}

\begin{proof}
First, note that $\Gen_H^{\iso}$ contains a unique isotropic genus of determinant equal to $D(H)$, which we shall also denote $H$ by abuse of notation.
By Prop.~\ref{massprop}, we have
\begin{align*}
\sum_{\substack{G\in\Gen^{\iso}_H\\1\leq D(G)\leq X}}D(G)\nu(G)&=D(H)\nu(H)\sum_{\substack{1\leq s\leq X/r(H)\\ (s,2r(H))=1 \\ \textup{s   squarefree} }}\prod_{p\mid s}\frac{1}{2}\left(1+\frac{1}{p}\right).
\end{align*}
Then (1) follows from Prop.~\ref{sqfree-mass}, and (2) from Cor.~\ref{useful-2} and Lemma~\ref{masslem}(3).
\end{proof}

\vspace{6pt}

\begin{proposition}\label{isotropic-siegel}(Alternative Prop.~\ref{isotropic-siegel-2})\  

There is a constant $\alpha>0$ such that, as $X\to\infty$,
\[\sum_{\substack{C\in\Cl^{\iso}\\1\leq D(C)\leq X}}D(C)\rho(C)\sim\alpha\frac{X}{\sqrt{\log X}}.\]
\end{proposition}

\begin{proof}
For each $X\geq2$, consider the function $S_X:\mathfrak R\to\R_{\geq0}$ defined by
\[f_X^{({\bf D})}(H)=\frac{1}{X(\log X)^{-1/2}}\sum_{\substack{G\in\Gen^{\iso}_H\\1\leq D(G)\leq X}}D(G)\nu(G),\]
and let $N:\mathfrak R\to\R_{\geq0}$ be given by $N(H)=1/r(H)$. We note that
\[\sum_{H\in\mathfrak R}N(H)=\sum_{H\in\mathfrak R}\frac{1}{r(H)}\leq\sum_{r\in |R|}\frac{4\cdot 2^{\omega(r)}\prod_{p\mid r}v_p(r)}{r},\]
where $|R|=\{r(d):d\in\Z_{\geq1}\}$. Because the numerator in each summand of the last series is $O_\epsilon(r^\epsilon)$, we conclude that $N\in L^1(\mathfrak R)$. It follows by Prop.~\ref{siegel-estimates}(2) that
$f_X^{({\bf D})}(H)\ll N(H)$
valid for all $H\in\mathfrak R$ and all $X\geq2$, where the implied constant is independent of $H$ and $X$. Thus $f_X^{({\bf D})}\in L^1(\mathfrak R)$ for every $X\geq2$. By the dominated convergence theorem the pointwise limit $f^{({\bf D})}=\lim_{X\to\infty} f_X^{({\bf D})}$ is in $L^1(\mathfrak R)$ and
\[\lim_{X\to\infty}\frac{\sqrt{\log X}}{X}\sum_{\substack{G\in\Gen^{\iso}\\1\leq D(G)\leq X}}D(G)\nu(G)=\lim_{X\to\infty}\sum_{H\in\mathfrak R}f_X^{({\bf D})}(H)=\sum_{H\in\mathfrak R}f^{({\bf D})}(H)<\infty.\]
The stated result follows.
\end{proof}


\section{Refined analysis of isotropic forms}
\subsection{Explicit Euler product}\label{comp}\

The multiplicative function $d\tilde{\nu}(d)$ from Sec.~\ref{sie-asy} governs the asymptotics in Prop.~\ref{isotropic-siegel-2}. In this section, we investigate the corresponding Euler product 
\[
I(s)=\sum_{d=1}^\infty\frac{d\,\tilde{\nu}^{\iso}(d)}{d^s}, 
\] as a function of $s$. 

We show that it is of finite degree (that is the Euler factors are rational functions of $p^{-s}$) and that it factors into elementary zeta functions times a degree two factor which has a natural boundary. We apply this to calculate $\varpi$ explicitly and to compare it with related calculations from the literature. We also use it to give a local density version for isotropic forms.

\begin{theorem}\label{exp-form}
    We have
    \[I(s)=\frac{\zeta(s+1)\zeta(2s)\zeta(2s+1)}{\zeta(3s+3)}\prod_{p}A_p(s),\]
    where
    \[A_p(s)=\begin{cases}1+\frac{1}{3}2^{-s}-\frac{1}{2}2^{-2s}, \quad\text{if} \ \ p=2,\quad\\ 1+ \frac{1}{2}p^{-s} - \frac{1}{2}p^{-s-1} -p^{-(2s+1)}, \quad\text{if} \ \ p>2. \end{cases}\]
\end{theorem}
To prove Prop.~\ref{comp} we compute the values of ${\tilde \nu}^{\iso}(p^{k})$, for which we need the values $m_p(f)$ of the $\Z_p$-equivalence classes of forms along with their multiplicity. The computation of $m_p(f)$ uses the algorithms given in \cite{CSMass}(Tables 1 and 2). We state the results without detailed proofs.

\begin{proposition}\label{massprop}
Let $p$ be an odd prime and $d\in\Z\backslash\{0\} $ and write $d=p^k d_0$ where $p\nmid d_0$  $($so $k=v_p(d)\geq 0$$)$.  Any form with $D(f)=d$ is $\mathbb{Z}_p$-equivalent to 
\[f_p({\bf x}) = f_p(x_1,x_2,x_3)=ax_1^2+p^{t_2}bx_2^2+p^{t_3}cx_3^2,\]
where $a,b,c\in\Zp$ are coprime to $p$, and $0\leq t_2\leq t_3$ with $t_2+t_3=k$.   Let $\epsilon$ be a fixed $p$-adic quadratic nonresidue.  Then up to $\Zp$-equivalence, exactly one of the following holds, and these are the only isotropic forms. We have
\begin{enumerate}[label=\upshape\arabic*.\ ]
	\item for $t_2=t_3=0$,  
	\[f_p({\bf x})=x_1^2+x_2^2+d_0x_3^2 \quad \text{with} \quad m_p(f)=\frac{1}{2}\left(1-p^{-2}\right)^{-1};\]
	\item for $0=t_2<t_3=k$,
	\[f_p({\bf x})=x_1^2-x_2^2-p^{k} d_0x_3^2 \quad \text{with} \quad m_p(f)=\frac{1}{4}\left(1-p^{-1}\right)^{-1}p^{k};\]
	\item for $0=t_2<t_3=k$ with $k$ even,  
	\[f_p({\bf x})=x_1^2-\epsilon x_2^2-p^{k}(d_0/\epsilon)x_3^2 \quad \text{with} \quad m_p(f)=\frac{1}{4}\left(1+p^{-1}\right)^{-1}p^{k};\]
	\item for $0<t_2=t_3=k/2$ with $k$ even, 
	\[f_p({\bf x})=-d_0 x_1^2+p^{k/2}(x_2^2-x_3^2) \quad \text{with} \quad m_p(f)=\frac{1}{4}\left(1-p^{-1}\right)^{-1}p^{k/2};\]
	\item for $0<t_2=t_3=k/2$ with $k/2$ even,
	\[ f_p({\bf x})=-d_0 x_1^2+p^{k/2}(x_2^2-\epsilon x_3^2)\quad \text{with} \quad m_p(f)=\frac{1}{4}\left(1+p^{-1}\right)^{-1}p^{k/2};\]
	\item for $0<t_2<t_3$, we have in all cases
	\[m_p(f)=\frac{1}{8}p^{t_3}.\]
	The inequivalent forms are as follows:
	\begin{enumerate}[label=\upshape\roman*.\ ]
	\item if $t_2$ is even and $t_3$ odd, then \[f_p({\bf x})=\epsilon^i(x_1^2-p^{t_2}x_2^2)+p^{t_3}d_0 x_3^3,\]
	with $i\in\{0,1\}$;
	\item if $t_2$ odd and $t_3$ even we have
	\[f_p({\bf x})=\epsilon^i(x_1^2-p^{t_3}x_3^2)+p^{t_2}d_0 x_2^3,\]
	with $i\in\{0,1\}$;
	\item if $t_2$ and $t_3$ both odd we have
	\[f_p({\bf x})=d_0x_1^2+\epsilon^i(p^{t_2}x_2^2-p^{t_3}x_2^2),\]
	with $i\in\{0,1\}$;
	\item if $t_2$ and $t_3$ both even we have
	\[f_p({\bf x})=\epsilon^ix_1^2+\epsilon^jp^{t_2}x_2^2+\epsilon^{i+j}d_0p^{t_3}x_2^2,\]
	with $i,j\in\{0,1\}$.
	\end{enumerate}
\end{enumerate}
\end{proposition}

For $p=2$, we will consider only the case $d=2^k$ for our computations (for if $d=2^kd_0$ with $d_0$ odd, then the corresponding  forms can be obtained from those with $d_0=1$ by the map $f\mapsto d_0f$).

\begin{proposition}\label{massprop3} For $p=2$, and $d=2^k$, the following are the only primitive isotropic forms $($see Props.~\ref{propB1} and \ref{propB2} for notation$)$. Given a quadratic form
\[f(x_1,x_2,x_3)=ax_1^2+2^ubx_2^2+2^{u+v}cx_2^2\]
with  $abc \equiv 1 \pmod{8}$ and $u,v\geq0$, we define $m_2^{(1)}(f)=m_2(f)/2^{u+v}$. 

The first Table below gives a breakdown for the number of classes $f$ multiplied with the corresponding value of $16m'_2(f)$, while the second is the composite sum, all with given $u$ and $v$:

\renewcommand{\arraystretch}{1.5}%
\begin{tabular}{ |c|c| c| c |c |c| }
\hline
u,\,v & $0$ & $1$ & $2$ & $\geq3$ odd & $\geq 3$ even\\
		\hline
$0$ & $1\cdot 4$ & $1\cdot 2$ & $1\cdot 1+1\cdot 2$ & $2\cdot \frac{1}{2}+1\cdot 1$ & $2\cdot\frac{1}{2}+2\cdot1$\\
$1$ & $1\cdot 2$ & $2\cdot 1$ & $2\cdot 1$ & $4\cdot\frac{1}{2}$ & $4\cdot\frac{1}{2}$\\
$2$ & $1\cdot1+1\cdot 2$ & $2\cdot1$ & $3\cdot1$ & $4\cdot\frac{1}{2}$ & $6\cdot\frac{1}{2}$\\
$\geq3$ odd & $2\cdot\frac{1}{2}+1\cdot 1$ & $4\cdot\frac{1}{2}$ & $4\cdot\frac{1}{2}$ & $8\cdot\frac{1}{4}$ & $8\cdot\frac{1}{4}$\\
$\geq3$ even & $2\cdot\frac{1}{2}+2\cdot1$ & $4\cdot\frac{1}{2}$ & $6\cdot\frac{1}{2}$ & $8\cdot\frac{1}{4}$ & $12\cdot\frac{1}{4}$\\
\hline
\end{tabular}\vspace*{.5cm}

\renewcommand{\arraystretch}{1.1}%
\begin{tabular}{ |c|c | c |c | }
\hline
u,\,v & $0$ & $\geq1$ odd & $\geq 1$ even \\
		\hline
$0$ & $4$ & $2$ & $3$ \\
$\geq1$ odd & $2$ & $2$ & $2$ \\
$\geq1$ even & $3$ & $2$ & $3$ \\
\hline
\end{tabular}

In addition, for forms of the type $ax_1^2 + 2^{u}V_i$,  we have 
	\begin{enumerate}[label=\upshape\arabic*.\ ]
		\item $f_2({\bf x}) = -x_1^2 + 2V_1$ with $m_2(f) = \frac{1}{6}$,
		\item $f_2({\bf x}) = -x_1^2 + 2^{u}V_1$ for all $u\geq 2$ with $m_2(f) = 2^{u-4}$,
		\item $f_2({\bf x}) = 3x_1^2 + 2^{u}V_2$ for all $u\geq2$ odd with  $m_2(f) = \frac{1}{3}\cdot  2^{u-4}$,
	\end{enumerate}
and for forms of the type  $V_i + 2^v\cdot cx_3^2 $ we have
\begin{enumerate}[label=\upshape\arabic*.\ ]
	\item  $f_2({\bf x}) = V_1 - 2 x_3^2 $ for $k=1$ with $m_2(f) = \frac{1}{6}$,
	\item  $f_2({\bf x}) = V_1 - 2^v x_3^2 $ for all $v\geq2$, with $m_2(f) =  2^{u-4}$,
	\item  $f_2({\bf x}) = V_2 + 2^v\cdot 3x_3^2 $ for $v\geq 2$ odd, with $m_2(f) = \frac{1}{3}\cdot  2^{u-4}$.
\end{enumerate} 
\end{proposition}

\vspace{6pt}

\begin{proof}[Proofs]
	These follow from the analysis in Sec.~\ref{genera}, together with results in Appendix~\ref{prelim}, and the Tables in Appendix~\ref{local-appendix} for the determination of the inequivalent forms, isotropy and the $p$-masses, the latter using the methods in \cite{CSMass} and \cite{AGM}.
\end{proof}

\qquad

\begin{proof}[Proof of Theorem.~\ref{exp-form}]\ 

Since $\tilde{\nu}^{\iso}(d)=\nu^{\iso}(d)/\nu^{\iso}(1)$ is multiplicative,  we write
\begin{equation}\label{Ips0} I(s) = \prod_{p\geq 2} \left(\sum_{k\geq 0}p^{k(1-s)}\tilde{\nu}^{\iso}(p^k)\right) = I_2(s)\prod_{p\geq 3}I_p(s),\end{equation}
where we note by Prop.~\ref{mult} that
\[\tilde\nu^\iso(p^k)=\frac{1}{p^{2k}}\frac{\sum_{G\in\Gen_p^\iso(p^k)}m_p(G)}{\sum_{G\in\Gen_p^\iso(1)}m_p(G)}.\]
If $p$ is odd, then by Prop.~\ref{massprop}(1) we have
\begin{equation}\label{Ips} I_p(s) = \left(1-\frac{1}{p^2}\right)\sum_{k=0}^\infty p^{-k(s+1)} \sum_{G\in\Gen^{\iso}(p^k)}2m_p(G).\end{equation} We apply Prop.~\ref{massprop}(2-6) to evaluate the sum over $k\geq 1$ taking into account the $m_p(G)$ values together with their multiplicities. From Prop.~\ref{massprop}(2-3) the contribution is
\[
\left(1-\frac{1}{p^2}\right)^{-1}\sum_{\substack{k\geq 1\\ 2\mid k}}p^{-ks} +\frac{1}{2}\left(1-\frac{1}{p}\right)^{-1}\sum_{\substack{k\geq 1\\ 2\nmid k}}p^{-ks}\ ;
\]
from (4-5) we get
\[\left(1-\frac{1}{p^2}\right)^{-1}\sum_{\substack{k\geq 1\\ 4\mid k}}p^{-k(s+\frac{1}{2})} +\frac{1}{2}\left(1-\frac{1}{p}\right)^{-1}\sum_{\substack{k\geq 1\\ 2\parallel k}}p^{-k(s+\frac{1}{2})}\ ;\]
and from (6), we have
\[
\frac{1}{2}\sum_{1\leq t_2<t_3}p^{-t_2(s+1)}p^{-t_3s} + \frac{1}{2}\sum_{\substack{1\leq t_2<t_3\\ 2\mid t_2,t_3}}p^{-t_2(s+1)}p^{-t_3s} .
\]
Combining these contributions into \eqref{Ips} and simplifying gives us for $p\geq 3$
\[ I_p(s) =\frac{1-p^{-3(s+1)}}{1-p^{-(s+1)}}\left[\left( 1-\frac{1}{p^{2s}}\right)  \left( 1-\frac{1}{p^{2s+1}}\right)\right]^{-1}A_p(s).   \]
To compute $I_2(s)$, we note by Appendix \ref{local-appendix} that $\sum_{G\in\Gen_2^{\iso}(1)}m_2(G)=1/4$ so
\[I_2(s)=4\sum_{k=0}^\infty\frac{1}{2^{k}}\left(\sum_{G\in\Gen_2^{\iso}(2^k)}m_2(G)\right)2^{-ks}.\]
Inputting the data from Prop~\ref{massprop3} in the same manner as we have done for $p$ odd above, we find that 
\[ I_2(s) = \frac{1-2^{-3(s+1)}}{1-2^{-(s+1)}}\left( 1-\frac{1}{2^{2s}}\right)^{-1}  \left( 1-\frac{1}{2^{2s+1}}\right)^{-1}A_2(s).\]
Thus, from \eqref{Ips0}, we have
\[
I(s) =\frac{\zeta(s+1)\zeta(2s)\zeta(2s+1)}{\zeta(3s+3)}\prod_{p}A_p(s),\]
as required.
\end{proof}

\begin{corollary}\label{exp-cor}\
\begin{enumerate}[label=\upshape\arabic*.\ ]
    \item We have $I(s)= \zeta(s)^{\frac{1}{2}}K(s)$, where $K(s)$ is regular for $\Re(s)> \frac{1}{2}$, and 
    \[ K(1) = \frac{25}{24}\cdot \frac{\zeta(3)}{\zeta(6)}\cdot\prod_{p\geq 2}\left[1 - \frac{p}{2(p+1)^2}\right]\left(1-\frac{1}{p}\right)^{-\frac{1}{2}}. \]
    \item For each odd prime $p$, we have
\[\mu_p(\Sym_{\prim}^\iso(\Z_p))=\frac{1}{\zeta_p(6)}\left(1-\frac{p}{2(p+1)^2}\right).\]
On the other hand, we have
\[\mu_2(\Sym_{\prim}^\iso(\Z_2))=\frac{25}{36}\cdot\frac{1}{\zeta_2(6)}.\]
\end{enumerate}
\end{corollary}
\begin{proof}\ \\
(1). By Thm.~\ref{exp-form}, we have $I(s) = \zeta(s)^{\frac{1}{2}}K(s)$ for $K(s)$ regular for $\Re(s)>\frac{1}{2}$ and explicitly given by
\[ \begin{split} K(s) = \frac{(1+\frac{1}{3}\cdot\frac{1}{2^{s}}-\frac{1}{2^{2s+1}})}{(1+\frac{1}{4}\cdot\frac{1}{2^{s}}-\frac{1}{2^{2s+1}})}\cdot & \frac{\zeta(s+1)\zeta(2s)\zeta(2s+1)}{\zeta(3s+3)} \\
& \times\prod_{p}\left(1+ \frac{p^{-s}}{2} - \frac{p^{-s-1}}{2} -\frac{1}{p^{2s+1}}\right)\left(1-p^{-s}\right)^{\frac{1}{2}},\end{split}\]
implying
\[ K(1) = \frac{25}{24}\cdot \frac{\zeta(3)}{\zeta(6)}\cdot\prod_{p\geq 2}\left[1 - \frac{p}{2(p+1)^2}\right]\left(1-\frac{1}{p}\right)^{-\frac{1}{2}}. \]
(2). Note that by Prop.~\ref{mult} we have \[\mu_p(\Sym_\prim^\iso(\Z_p))=\frac{|\Sq^*(\Z_p)| m_p(\Sym_\prim^\iso(1)(\Z_p))}{\zeta_p(1)\zeta_p(2)\zeta_p(3)}\cdot I_p(1),\]  where we write $\zeta_p(s)=(1-p^{-s})^{-1}$ as in Sect.~\ref{sie-asy}. Thus, for odd prime $p$, we have by direct computation
\[\mu_p(\Sym_\prim^\iso(\Z_p))=\frac{1}{\zeta_p(1)\zeta_p(3)} I_p(1)=\frac{\zeta_p(2)^2}{\zeta_p(1)\zeta_p(6)}A_p(1)=\frac{1}{\zeta_p(6)}\left(1-\frac{p}{2(p+1)^2}\right).\]
On the other hand, for $p=2$ we have
\[\mu_2(\Sym_{\prim}^\iso(\Z_2))=\frac{1}{\zeta_2(1)\zeta_2(2)\zeta_2(3)}I_2(1)=\frac{25}{24}\frac{\zeta_2(2)}{\zeta_2(1)\zeta_2(6)}=\frac{25}{36}\cdot\frac{1}{\zeta_2(6)},\]
which gives us the result.\end{proof}

\begin{remark}[$\varpi$ as a product and further interpretations]\label{rem10.5}\ 

		   The explicit determination  of the local factors $\mu_p(\Sym_{\prim}^\iso(\Z_p))$ in Corl.~\ref{exp-cor}(2) allows us to compare these probabilities with the calculations of Bhargava et al. (\cite{bhargavaetal}) of the same, namely that an integral form is (primitive and) $\mathbb{Z}_p$-isotropic. These agree for $p>2$. For $p=2$ they differ in view of the notion of being primitive integral considered there is in terms of the coefficients of $F$, while it is in terms of the entries of the matrix of $F$ in our work. As far as the archimedean factor in the asymptotics in Theorem~\ref{IThm2}, it depends on $\Omega$ (unlike the finite primes) and has a natural probabilistic interpretation, being the fraction $\text{VOL}(\Omega^{\iso})/\text{VOL}(\Omega)$ of the real isotropic forms in $\Omega$.
	
    In our local density formulation of Theorem.~\ref{IThm2} in Remark~\ref{rem1.4} and in Sec.~\ref{iso-revisited} below, the dependence on $\Omega$ is naturally absent. In the recent paper (\cite{LRS}) on the leading constant for rational points in families, the authors give a conjectured constant for the asymptotics in Thm.~\ref{IThm2} with $\Omega = [-1,1]^6$ (they also check and prove their conjecture for the family of diagonal isotropic ternaries, extending Guo \cite{Guo}). The non-archimedean factors in their conjecture are our $\mu_p(\Sym_{\prim}^\iso(\Z_p))(1-p^{-1})^{-1}$, being Peyre's local Tamagawa measure for the locus of isotropic forms in $\mathbb{P}^6(\mathbb{Q})$. Moreover, their convergence factors agree with ours after incorporating their $(1-p^{-1})^{\frac{1}{2}}$ factor (\cite{LRS} Sec.~3.6). Their archimedean, as well as extra global factors involve special integrals connected with their choice of an archimedean height, and we have not verified that these agree with Theorem~\ref{IThm2}.\qed
\end{remark}

\subsection{Local density of isotropic forms: Theorem \ref{IThm2} revisited}\label{iso-revisited}\

	Theorem \ref{IThm2} gives the asymptotics for the number of (primitive) isotropic forms $\textup{ISO}(\Omega X)$ that lie in the dilates of $\Omega$ by $X$ with $X\to\infty$. Our analysis provides a natural local density for such forms:
	\begin{equation}\label{101}
	d\lambda^{\iso}=\varpi\frac{dx_1\cdots dx_6}{\sqrt{\log^+D(x)}}\quad\text{on $V^+(\R)$.}
	\end{equation}
	Serre \cite{serre} points to the significance of the divisor $D(x)$ in the exponent $\frac{1}{2}$ of the log in (\ref{iso1}), and (\ref{101}) highlights its further significance in the asymptotic count. For $W\subset V^+(\R)$, (\ref{101}) gives the expected value for the number $\text{ISO}(W)$ of isotropic forms lying in $W$, namely
	\begin{equation}\label{102}
	\int_Wd\lambda^{\iso}.
	\end{equation}
	Our proof of Theorem \ref{IThm2} yields local versions for sets $W$ for which the expected value is valid asymptotically. We give one such family of $W$'s. For $X$ large and $H\leq X$ consider the ``annuli''
	\begin{equation}\label{103}
	W_\Omega(X,H)=\{F\in V^+(\R):\tilde F\in\Omega, X\leq D(F)\leq X+H\}.
	\end{equation}
	Then for $H\geq X^{17/30}$ we claim that
	\begin{equation}\label{104}
	\text{ISO}(W_{\Omega}(X,H))\sim\int_{W_{\Omega}(X,H)}d\lambda^{\iso},
	\end{equation}
	as $X\to\infty$. Note that $\log D(X)$ is asymptotically constant at $\log X$ on $W_\Omega(X,H)$, so that the right hand side is
	\[\int_{W_{\Omega}(X,H)}d\lambda^{iso}\sim\varpi\frac{\Vol(W_\Omega(X,H))}{\sqrt{\log X}}\sim \varpi\mu_1(\Omega)\frac{XH}{\sqrt{\log X}}.\]
	To deduce (\ref{104}) we use the local level asymptotics (\ref{86}) and (\ref{H-iso}) which give
	\begin{equation}
	\sum_{\substack{D(F)=d\\\tilde F\in\Omega \\\textup{$F$ primitive isotropic}}}1\sim\frac{2}{\zeta(2)\zeta(3)}\mu_1(\Omega)d^2\nu^{\iso}(d),
	\end{equation}
	as $d\to\infty$.
	While $d^2\nu^{\iso}(d)$ varies arithmetically with $d$ in that \[d^{1-\epsilon}\ll_\epsilon d^2\nu^{\iso}(d)\ll_\epsilon d^{1+\epsilon},\] the following gives its averages over short intervals.
	\begin{proposition}
		For $X^{17/30}\leq H\leq X$,
		\begin{equation}
		\sum_{X\leq d\leq X+H}d^2\nu^{\iso}(d)\sim\frac{\zeta(2)\zeta(3)}{2}\varpi\frac{XH}{\sqrt{\log X}},
		\end{equation}
		as $X\to\infty$.
	\end{proposition}
	\begin{proof}
		We have by Thm.~\ref{exp-form} that
		\begin{align*}\sum_{d=1}^\infty d^2\nu^{\iso}(d)d^{-s}&=\prod_p A_2(s-1)\cdot\frac{\zeta(s)\zeta(2s-2)\zeta(2s-1)}{\zeta(3s)},\\
		&=\sqrt{\zeta(s-1)}K_1(s),
		\end{align*}
		where $K_1(s)$ is analytic and uniformly bounded in $\Re(s)>3/2$. Now $\sqrt{\zeta(s-1)}$ has a square-root algebraic singularity coming from the pole at $s=2$, and square-root algebraic singularities at $1+\rho$ in $\Re(s)>3/2$ corresponding to the zeros $\rho$ of $\zeta(s)$ in $\Re(s)>1/2$. So if there are no such zeros then applying Perron's formula and Landau's contour shift \cite{Landau} gives the proposition with $H\gg X^{1/2+\epsilon}$. If there are zeros $\rho$ in $\Re(s)>1/2$ one uses a density theorem as noted by Hooley \cite[p.98]{Hooley-bk}. Indeed, any one such zero contributes a negligible amount to the right hand side of the proposition. So the issue is how many $\rho$ there may be and where are they located in terms of their real parts in $(1/2,1)$. That is controlled by a density theorem for $\zeta(s)$. The recent advance by Guth and Maynard \cite{GuthMaynard2024} allows us to take $H\geq X^{17/30}$.
\end{proof}


\part*{Appendix}
\begin{appendices}

We collect here in Appendix \ref{prelim} some classical results about ternary quadratic forms. In Appendix \ref{2classes}, we give  a detailed breakdown of classes, together with mass computations for $p=2$ local forms following work of Conway-Sloane.  

\section{Preliminaries on integral  quadratic forms}\label{prelim}\ 

\begin{proposition}\label{propA1} \upn{({\cite{Cassels} Thm.~1.3, Sec.~9.1})}\ \\
\indent Let $F$ be a non-singular form over $\mathbb{Z}$, and let $a\in \mathbb{Z}\backslash\{0\}$ such that it is represented by $F$ over $\mathbb{Z}_p$ for all $p\leq \infty$. Then $a$ is represented over $\mathbb{Z}$ by some form $F^*\in{\it G}(F)$. In particular, if  ${\it G}(F)$ has one equivalence class, then $a$ is represented by $F$.
\end{proposition} 

\begin{proposition}\label{propA2} \upn{({\cite{Conway1993} Thm.~21, Sec.~15})}\ \\
\indent For an indefinite form $F$ of dimension $n$ and determinant $d$, with more than one equivalence class in its genus, we must have
\[ k^{\binom{n}{2}}  \! \mid \!  4^{\left[\frac{n}{2}\right]}d ,\] for some non-square $k\in \mathbb{N}$ satisfying $k \equiv 0\ \text{or}\ 1  \Mod{4}$. In particular, if $d$ is squarefree, then all forms in ${\it G}(F)$ are equivalent.
\end{proposition} 

\begin{proposition}\label{propA3} \ Suppose $p$ is odd.
\begin{enumerate}[label=\upshape\arabic*.\ ]
\item The equation $ax^2+by^2=c$ is always solvable over $\Zp$ when $a$,$b$ and $c$ are units.
\item \upn{({\cite{jones-bk} Thm.~32, Sec.~4})}:\ \\ Every form $F \in \mathbb{Z}_p[{\bf x}]$ is $\mathbb{Z}_p$-equivalent to a diagonal form over $\mathbb{Z}_p$.
\item \upn{({\cite{Cassels} Lem.~3.4, Sec.~8.3})}:\ \\  Let $h\in \mathbb{Z}_p[{\bf x}]$ be a unit determinant diagonal form. Then, it is equivalent to a form of the type $x_1^2 + \ldots + x_{n-1}^2 + \det(h)x_n^2$. 
\item \upn{({\cite{Cassels} Thm.~3.1, Sec.~8.3})}:\ \\   Let $r$ be a fixed quadratic nonresidue of $p$. Then, every non-singular $F$ over $\mathbb{Q}_p$ is $\mathbb{Z}_p$-equivalent to a form of the type $g = p^{e_1}h_1 \oplus p^{e_2}h_2 \oplus \ldots \oplus p^{e_J}h_J$ with $e_1 <e_2 < \ldots < e_J$, with $\dim(h_j) = m_j$, where $m_1 + \ldots + m_J = \dim(F)$, and where the $h$-forms are of the type  $x_1^2 + \ldots + x_{m-1}^2 + \eta x_m^2$ with $\eta = 1$ or $r$. Moreover, the forms in the direct sum are inequivalent.
\end{enumerate}

\end{proposition} 

\begin{proposition}\label{binquad-lem}
Let $a,b,c\in\mathbb{Z}$ be odd integers.
\begin{enumerate}[label=\upshape\arabic*.\ ]
\item $x^2+xy+y^2=c$ is solvable over $\mathbb{Z}_2$.
\item At least one of $ax^2+by^2=c$ and $ax^2+by^2=-c$ is solvable over $\mathbb{Z}_2$.
\end{enumerate}
\end{proposition}

\begin{proof}\ 

1. Letting $y=1$, it suffices to show that $f(x)=x^2+x+c-1=0$ is solvable over $\mathbb{Z}_2$. Then $f(1)\equiv 0\mod2 $ but $f^{(1)}(1)\not\equiv0\mod2$, so the result follows by Hensel's lemma. 

2. Up to replacing $c$ by $-c$, we may assume that $a\equiv c\mod 4$. Then $ax^2+by^2\equiv c\mod 8$ is solvable with $(x_0,y_0)=(1,0)$ or $(1,2)$. Since $f(x,y)=ax^2+by^2-c$ satisfies $|f(x_0,y_0)|_2\leq\frac{1}{8}<\frac{1}{4}=|\partial_x f(x_0,y_0)|_2^2$, the result again follows by Hensel's lemma.
\end{proof}

\vspace{6pt}

\begin{definition}\label{def3.1}\cite{jones-bk}\ \\
\indent For each prime $p$, let $c_p(F)$ denote the Hasse invariant of $F$, where 
\[
c_p(F)= (-1,-d)_p(D_1,-D_2)_p(D_2,-d)_p ,
\]
with $D_1=f_{11}$, $D_2= f_{11}f_{22}-f_{12}^2$, for $F$ with matrix entries $f_{ij}$. Here $(\ ,\ )_p$ is the Hilbert symbol. Note that this differs from Cassels' definition by the factor $(-1,-d)_p$.

\end{definition}

\begin{proposition}\label{isotropics}\upn{(\cite{jones-bk}, {Thm.~14b})}\ \\
\indent $F$ is an isotropic form over $\Zp$ if and only if $c_p(F)=1$ for all primes $p$.
\end{proposition}

\begin{proposition}\label{jonesT29}\upn{(\cite{jones-bk}, {Thm.~29})} \ \\
\indent Let $d\geq 1$ and a set of values for $c_p(F) = \pm1$, for $p\leq \infty$. Then necessary and sufficient conditions for the existence of an integral form $F$ of determinant $d$ with specified Hasse invariants and index $i$ are:
\begin{enumerate}[label=\upshape\arabic*.\ ]
\item $c_p(F)=1$ for finite $p\nmid 2d$,
\item $\prod_{p\leq \infty} c_p(F) = 1$, and
\item $3-i  \equiv c_{\infty}(F)+ 1 \Mod{4}$.
\end{enumerate}
\end{proposition}

\begin{proposition}\label{jonesT46}\upn{(\cite{jones-bk}, {Thm.~46})} \ \\
\indent Given a set of field invariants satisfying Prop.~\ref{jonesT29},and a set of ring invariants satisfying the conditions given by the forms $F_p$ in Section~\ref{genera}, and consistent with the field invariants, there is a form of determinant $d$ with integral coefficients having all the given invariants. Consistency here means:
\begin{enumerate}[label=\upshape\arabic*.\ ]
\item the sum of the dimensions in the Jordan decomposition is 3,
\item the sum powers of $p$, with multiplity the dimensions,  match $v_p(d)$,
\item the product of the determinants of the forms in the Jordan decomposition is $d_0u^2$ for some unit in $\Zp$, with $d=v_p(d)d_0$, and
\item the improperly primitive forms $V_i$ over $\mathbb{Z}_2$ satisfy $\det(V_i) \equiv -1 \mod{4}$.
\end{enumerate}
All these consistency conditions are satisfied by the forms given in Section~\ref{genera}.
\end{proposition}

\begin{remark}
These last two propositions have their analogues in \cite{Conway1993} in Thms. 9 and 11.
\end{remark}


\section{Primitive forms and classes for p=2.}\label{2classes}

\subsection{Forms and matrices over \texorpdfstring{$\mathbb{Z}_2$}{{\bf Z}2}.}\label{Z2}\ 

For $p=2$, a form and its matrix are said to be
{\it improperly primitive} if the matrix has some unit element but has no unit element
on its diagonal. If a matrix or form has a unit diagonal element it is called
{\it properly primitive}.

\begin{proposition}\label{propB1}\upn{({\cite{Jones1944}, Thm.~1})}\ \\ 
\indent Every symmetric matrix  $\mathcal{U}$  over $\mathbb{Z}_2$ is $\mathbb{Z}_2$-equivalent   to a matrix given in block form
\[ \mathcal{U} \cong  \{2^{t_1}\mathcal{U}_1, \ldots, 2^{t_m}\mathcal{U}_m\}, \]
where for every $j$, $t_j<t_{j+1}$ and $\det(\mathcal{U}_j)$ is a unit. Furthermore, for ternary forms, each $\mathcal{U}_j$ is a diagonal matrix giving forms  in Prop.~\ref{propB2}, or is of the form  $\mathcal{V}_1 =  \left[\begin{smallmatrix} 0&1\\ 1&0 \end{smallmatrix}\right]$ or   $\mathcal{V}_2 =  \left[\begin{smallmatrix} 2&1\\ 1&2 \end{smallmatrix}\right]$. The latter give rise to forms we denote by $V_j(x,y)$.
\end{proposition}

\begin{proposition}\label{propB2}\upn{({\cite{Jones1944}, Lem.~3})}\ \\  
\indent Let $\mathcal{T} = \{1, 3, 5, 7\}$. A diagonal properly primitive form over $\mathbb{Z}$ in $r\leq 3$ variables and of odd determinant $d$ is $\mathbb{Z}_2$-equivalent to a unique form of the type $\sum_{i=1}^{r} a_ix_i^2$ where
\begin{enumerate}[label=\upshape\arabic*.\ ]
\item if $r=1$, $a_1 \equiv d \Mod{8}$ with $a_1 \in \mathcal{T}$;
\item if $r=2$, then $(a_1,a_2)$ takes the following set of values
 \begin{flalign*} 
   (1,1)\ \text{or}\ (3,3)\quad  & \text{if}\  d \equiv 1 \Mod{8}, \\
  (1,5)\ \text{or}\ (3,7)\quad  & \text{if}\  d \equiv 5 \Mod{8}, \\
  (1,3)\quad \quad   & \text{if}\  d \equiv 3 \Mod{8}, \\
  (1,7)\quad \quad   & \text{if}\  d \equiv 7 \Mod{8};
\end{flalign*}
\item if $r=3$, then $(a_1,a_2,a_3)$ takes the following set of values
\begin{flalign*} 
   (1,1,1)\ \text{or}\ (1,3,3)\quad  & \text{if}\  d \equiv 1 \Mod{8}, \\
  (1,1,5)\ \text{or}\ (1,3,7)\quad  & \text{if}\  d \equiv 5 \Mod{8}, \\
  (1,1,3)\ \text{or}\ (3,3,3)\quad    & \text{if}\  d \equiv 3 \Mod{8}, \\
  (1,1,7)\ \text{or}\ (3,3,7)\quad  & \text{if}\  d \equiv 7 \Mod{8}.
\end{flalign*}
\end{enumerate}
\end{proposition}

\subsection{Mass computations}\label{local-appendix}\ 

		Given a nonsingular primitive ternary quadratic form $f$ with $\Z_2$-coefficients, there exist unique $u,v\geq0$ such that $f$ can be written up to equivalence as
 \[f(x_1,x_2,x_3)=\begin{cases}ax_1^2+2^ubx_2^2+2^{u+v}cx_3^3 ,\quad\text{or}\\
		 V_i(x_1,x_2)+2^vcx_3^2 ,\quad\text{or}\\
		 ax_1^2+2^uV_i(x_2,x_3),
		\end{cases}\]
		for some $a,b,c\in\Z_2$ with $(abc,2)=1$ and $i\in\{1,2\}$.
		 We note however that the classes $a,b,c\in\Sq^*(\Z_2)$ are not $\GL(3,\Z_2)$-invariants of $f$ in general \cite[Sect. 7]{Conway1993}, due to the possibility of oddity fusion and sign-walking, (see \cite{AGM} for further details, where some minor inaccuracies in \cite{Conway1993} concerning canonical symbols are also corrected).
        In the tables below, we give a complete classification of the $\mathbb{Z}_2$-inequivalent classes of determinant $d=2^k\in\Z_2/(\Z_2^*)^2$, along with their normalized $2$-masses which are read off from \cite{CSMass}. The classes are listed by their $2$-adic symbols; in particular, $1_{I\!I}$ denotes the components $V_i$ mentioned above. The column labeled ``$\iso$'' indicates whether or not the listed classes are isotropic (Y) or anisotropic (N). We write $q=2^{u}$ and $r=2^{u+v}$.

\begin{table}[H]
\centering
    \renewcommand{\arraystretch}{2}%
    \scriptsize{\begin{tabular}{ |c|c|c|c||l|c| }
				\hline
				$u$ & $v$ & $\frac{64m_2}{2^{u+v}}$ & $\iso$& Classes & Total $\#$ \\  \hhline{|=|=|=|=||=|=|}
				$0$ & $0$ & $16$ & Y & $1_{-1}^3$  & $1$\\\hline
				$0$ & $0$ & $\frac{16}{3}$ & N & $1_3^3$ & $1$\\\thickhline
                $0$ & $1$ & $8$ & Y & ${[1^22^1]}_{-1}$ & $1$\\\hline
                $0$ & $1$ & $8$ & N & ${[1^22^1]}_{3}$ & $1$\\\hline
                $0$ & $1$ & $\frac{16}{3}$ & Y & $1_{I\!I}^{2}2_{-1}^1$ & $1$\\\thickhline
                $0$ & $2$ & $8$ & Y & $1_0^24_{-1}^1$ & $1$\\\hline
                $0$ & $2$ & $4$ & Y & $1_{-2}^{2}4_{1}^1,1_{I\!I}^24_{-1}^1$ & $2$\\\hline
                $0$ & $2$ & $4$ & N & $1_2^24_1^1$ & $1$\\\hline
                $0$ & $2$ & $\frac{4}{3}$ & N & $1_{I\!I}^{-2}4_{3}^{-1}$ & $1$\\ \hline
                \end{tabular}}
                \end{table}
                
\begin{table}[H]
\centering
    \renewcommand{\arraystretch}{2}%
    \scriptsize{\vspace*{1cm}\begin{tabular}{ |c|c|c|c||l|c| }
				\hline
                $0$ & $\geq3$ odd & $4$& Y & $1_0^2 r_{-1}^1,1_{I\!I}^2r_{-1}^1$&$2$\\\hline
                $0$ & $\geq3$ odd & $4$& N & $1_4^{-2}r_{3}^{-1}$ &$1$\\\hline
                $0$ & $\geq3$ odd & $2$& Y & $1_{-2}^{-2}r_{-3}^{-1},1_{-2}^{2}r_1^1$&$2$\\\hline
                $0$ & $\geq3$ odd & $2$& N & $1_2^{2}r_1^1,1_2^{-2}r_{-3}^{-1}$&$2$\\\hline
                $0$ & $\geq3$ odd & $\frac{4}{3}$& Y & $1_{I\!I}^{-2}r_{3}^{-1}$&$1$\\\thickhline
                $0$ & $\geq3$ even & $4$& Y & $1_0^2 r_{-1}^1,1_{I\!I}^2r_{-1}^1,1_4^{-2}r_{3}^{-1}$&$3$\\\hline
                $0$ & $\geq3$ even & $2$& Y & $1_{-2}^2r_1^1,1_2^{-2}r_{-3}^{-1}$ & $2$ \\ \hline
                $0$ & $\geq3$ even & $2$& N & $1_2^2r_1^1,1_{-2}^{-2}r_{-3}^{-1},$&$2$\\ \hline
                $0$ & $\geq3$ even & $\frac{4}{3}$& N & $1_{I\!I}^{-2}r_{3}^{-1}$&$1$\\ \thickhline
                $1$ & $0$ & $8$ & Y& ${[1^12^2]}_{-1}$& $1$\\\hline
				$1$ & $0$ & $8$ & N& ${[1^12^2]}_{3}$& $1$\\\hline
				$1$ & $0$ & $\frac{16}{3}$ & Y& $1_{-1}^12_{I\!I}^{2}$& $1$\\\thickhline
				$1$ & $1$ & $8$& Y & ${[1^{1}2^14^1]}_{-1}$ & $1$\\\hline
				$1$ & $1$ & $8$& N & ${[1^{1}2^14^1]}_{3}$ & $1$\\\thickhline
				$1$ & $2$ & $4$& Y & ${[1^12^1]}_08_{-1}^1,{[1^12^1]}_{-2}8_1^1,$ &$2$\\\hline
                $1$ & $2$ & $4$& N & ${[1^{-1}2^{-1}]}_08_{-1}^1,{[1^12^1]}_28_1^1$ &$2$\\\thickhline
                $1$ & $\geq 3$ odd & $2$& Y & ${[1^12^1]}_0r_{-1}^1,{[1^12^{-1}]}_4r_{3}^{-1},{[1^{1}2^{-1}]}_{2}r_{-3}^{-1}$, & $4$ \\
                & & & &   ${[1^12^1]}_{-2}r_1^1$ &\\\hline 
				$1$ & $\geq 3$ odd & $2$& N & ${[1^{-1}2^{-1}]}_0r_{-1}^1,{[1^{-1}2^1]}_4r_{3}^{-1},{[1^{1}2^{-1}]}_{-2}r_{-3}^{-1}$,  & $4$\\
                & & & & ${[1^12^1]}_2r_1^1$ & \\ \thickhline
				$1$ & $\geq 3$ even & $2$& Y & ${[1^12^1]}_0r_{-1}^1,{[1^12^1]}_{-2}r_1^1,[1^12^{-1}]_4r_3^{-1}$, & $4$\\
                & & & & ${[1^12^{-1}]}_{-2}r_{-3}^{-1}$ & \\ \hline
				$1$ & $\geq 3$ even & $2$& N & ${[1^{-1}2^{-1}]}_0r_{-1}^1,{[1^12^1]}_2r_1^1,{[1^{-1}2^1]}_4r_3^{-1}$,  & $4$\\
                & & & & ${[1^12^{-1}]}_2r_{-3}^{-1}$ & \\ \thickhline 
                \end{tabular}}
                \end{table}

\begin{table}[H]
\centering
                \renewcommand{\arraystretch}{1.8}%
    \scriptsize{\begin{tabular}{ |c|c|c|c||l|c| }
				\hline
				$u$ & $v$ & $\frac{64m_2}{2^{u+v}}$ & $\iso$& Classes & Total $\#$ \\  \hhline{|=|=|=|=||=|=|}
                $2$ & $0$ & $8$& Y&  $1_{-1}^14_0^2$ & $1$\\\hline
                $2$ & $0$ & $4$& Y&  $1_1^14_{-2}^{2},1_{-1}^14_{I\!I}^2$ & $2$\\\hline
                $2$ & $0$ & $4$ & N&  $1_1^14_2^2$ & $1$\\\hline
                $2$ & $0$ & $\frac{4}{3}$ & N&  $1_{3}^{-1}4_{I\!I}^{-2}$ & $1$\\\thickhline
				$2$ & $1$ & $4$ & Y& $1_{-1}^1{[4^18^1]}_0,1_1^1{[4^18^1]}_{-2}$ & $2$\\\hline
                $2$ & $1$ & $4$ & N&
                $1_{-1}^1{[4^{-1}8^{-1}]}_0,1_1^1{[4^18^1]}_2$ & $2$ \\\thickhline
                $2$ & $2$ & $4$ & Y & $1_1^14_{-1}^116_{-1}^1,1_{-1}^14_1^1 16_{-1}^1,1_{-1}^14_{-1}^116_1^1$ & $3$\\\hline
                $2$ & $2$ & $4$ & N & $1_1^14_1^116_1^1$& $1$\\\thickhline
				$2$ & $\geq3$ odd & $2$& Y  & $1_{-1}^14_{-1}^1r_1^1,1_{-1}^14_1^1r_{-1}^1,1_{1}^14_{-1}^1r_{-1}^1,1_{-3}^{-1}4_1^1r_{-3}^{-1}$ & $4$\\\hline
                $2$ & $\geq3$ odd & $2$ & N & $1_1^14_1^1r_1^1,1_3^{-1}4_1^1r_3^{-1},1_3^{-1}4_{-1}^1r_{-3}^{-1},1_{-3}^{-1}4_{-1}^1r_3^{-1}$ & $4$\\\hline
				$2$ & $\geq3$ even & $2$& Y  & $1_{-1}^14_{-1}^1r_1^1,1_3^{-1}4_{-1}^1r_{-3}^{-1},1_{-1}^14_1^1r_{-1}^1,1_{1}^14_{-1}^1r_{-1}^1$, & $6$\\
				& & & & $1_3^{-1}4_1^1r_{3}^{-1},1_{-3}^{-1}4_{-1}^1r_{3}^{-1}$ &\\\hline
                $2$ & $\geq3$ even & $2$ & N & $1_1^14_1^1r_1^1,1_{-3}^{-1}4_1^1r_{-3}^{-1}$ & $2$\\\hline
  \end{tabular}}
  \end{table}

\begin{table}[htb!]
\centering
     \renewcommand{\arraystretch}{1.8}%
     	\scriptsize{\begin{tabular}{ |c|c|c|c||l|c| }
     \hline
				$ \underset{\text{odd}}{u\geq 3}$ & $v$ & $\frac{64m_2}{2^{u+v}}$ & $\iso$& Classes & Total $\#$ \\  \hhline{|=|=|=|=||=|=|}
				 & $0$ & $4$ & Y & $1_{-1}^1q_0^2,1_{-1}^1q_{I\!I}^2$ & $2$\\\hline
				 & $0$ & $4$ & N & $1_{3}^{-1}q_4^{-2}$ & $1$\\\hline
                 & $0$ & $2$ & Y & $1_1q_{-2}^{2},1_{-3}^{-1}q_{-2}^{-2},$ & $2$ \\\hline
                 & $0$ & $2$ & N & $1_1^1q_2^{2},1_{-3}^{-1}q_2^{-2}$ & $2$\\\hline
                 & $0$ & $\frac{4}{3}$ & Y & $1_{3}^{-1}q_{I\!I}^{-2}$ & $1$\\\thickhline
				 & $1$ & $2$ & Y & $1_{-1}^1{[q^1r^1]}_0,1_{3}^{-1}{[q^1r^{-1}]}_4,1_{-3}^{-1}{[q^{-1}r^1]}_{-2},1_1^1{[q^1r^1]}_{-2}$ & $4$\\\hline
                 & $1$ & $2$ & N & $1_{-1}^1{[q^{-1}r^{-1}]}_0,1_{3}^{-1}{[q^{-1}r^1]}_4,1_1^1{[q^1r^1]}_2,1_{-3}^{-1}{[q^{-1}r^1]}_2$ & $4$\\\thickhline
                 & $2$ & $2$ & Y & $1_{-1}^1q_{-1}^1r_1^1,1_{-1}^1q_1^1r_{-1}^1,1_1^1q_{-1}^1r_{-1}^1,1_{-3}^{-1}q_1^1r_{-3}^{-1}$ & $4$\\\hline
                 & $2$ & $2$ & N & $1_1^1q_1^1r_1^1,1_3^{-1}q_{-1}^1r_{-3}^{-1},1_3^{-1}q_1^1r_{3}^{-1},1_{-3}^{-1}q_{-1}^1r_{3}^{-1}$ & $4$ \\\thickhline
				 &$\underset{\text{odd}}{v\geq 3}$  & $1$& Y& $1_1^1q_{-3}^{-1}r_{-3}^{-1},1_1^1q_{-1}^1r_{-1}^1,1_3^{-1}q_1^1r_3^{-1},1_3^{-1}q_{-1}^1r_{-3}^{-1},$& $8$\\
                &  & & & $1_{-3}^{-1}q_{-3}^{-1}r_1^1,1_{-3}^{-1}q_{-1}^1r_{3}^{-1},1_{-1}^1q_{-1}^1r_1^1,1_{-1}^1q_1^1r_{-1}^1$ & \\\hline
                 &$\underset{\text{odd}}{v\geq 3}$  & $1$& N& $1_1^1q_1^1r_1^1,1_1^1q_3^{-1}r_3^{-1},1_{-1}^1q_3^{-1}r_{-3}^{-1},1_{-1}^1q_{-3}^{-1}r_3^{-1}$ & $8$\\
                & & & & $1_3^{-1}q_3^{-1}r_1^1,1_3^{-1}1_{-3}^{-1}r_{-1}^1,1_{-3}^{-1}q_1^1r_{-3}^{-1},1_{-3}^{-1}q_3^{-1}r_{-1}^1$ & \\\hline
				 &$\underset{\text{even}}{v\geq 3}$  & $1$& Y& $1_1^1q_{3}^{-1}r_{3}^{-1},1_1^1q_{-1}^1r_{-1}^1,1_{-3}^{-1}q_{-3}^{-1}r_1^1,1_{-3}^{-1}q_1^1r_{-3}^{-1},$ & $8$\\
                & & & & $1_{-1}^1q_{-1}^1r_1^1,1_{-1}^1q_{-3}^{-1}r_3^{-1},1_{-1}^1q_{3}^{-1}r_{-3}^{-1},1_{-1}^1q_1^1r_{-1}^1$ & \\\hline
				 &$\underset{\text{even}}{v\geq 3}$  & $1$& N& $1_1^1q_1^1r_1^1,1_3^{-1}q_{-3}^{-1}r_{-1}^{1},1_1^1q_{-3}^{-1}r_{-3}^{-1},1_3^{-1}q_1^1r_3^{-1},$ & $8$\\
                & & & & $1_3^{-1}q_{-1}^1r_{-3}^{-1},1_3^{-1}q_3^{-1}r_1^1,1_{-3}^{-1}q_{-1}^1r_3^{-1},1_{-3}^{-1}q_{3}^{-1}r_{-1}^1$& \\\hline
                	\end{tabular}} \end{table}

\begin{table}[htb!]
\centering
  
                     \renewcommand{\arraystretch}{2}%
                		\scriptsize{\begin{tabular}{ |c|c|c|c||l|c| }
                			\hline
                			$ \underset{\text{even}}{u\geq 3}$ & $v$ & $\frac{64m_2}{2^{u+v}}$ & $\iso$& Classes & Total $\#$ \\  \hhline{|=|=|=|=||=|=|}
				 & $0$ & $4$ & Y & $1_{-1}^1q_0^2,1_{3}^{-1}q_4^{-2},1_{-1}^1q_{I\!I}^2$ & $3$\\\hline
                 & $0$ & $2$ & Y & $1_1^1q_{-2}^2,1_{-3}^{-1}q_{2}^{-2}$ & $2$ \\\hline
                & $0$ & $2$ & N & $1_1^1q_2^2,1_{-3}^{-1}q_{-2}^{-2}$ & $2$\\\hline
                & $0$ & $\frac{4}{3}$ & N & $1_{3}^{-1}q_{I\!I}^{-2}$ & $1$\\\thickhline
                 & $1$ & $2$ & Y & $1_{-1}^1{[q^1r^1]}_0,1_{3}^{-1}{[q^{-1}r^1]}_4,1_1^1{[q^1r^1]}_{-2},1_{-3}^{-1}{[q^{-1}r^1]}_2$ & $4$\\\hline
                 & $1$ & $2$ & N & $1_{-1}^1{[q^{-1}r^{-1}]}_0,1_{3}^{-1}{[q^1r^{-1}]}_4,1_1^1{[q^{1}r^1]}_2,1_{-3}^{-1}{[q^{-1}r^1]}_{-2}$ & $4$\\\thickhline
                 & $2$ & $2$ & Y & $1_{-1}^1q_{-1}^1r_1^1,1_3^{-1}q_{-1}^1r_{-3}^{-1},1_{-1}^1q_1^1r_{-1}^1,$ & $6$\\
                &  & &  & $1_1^1q_{-1}^1r_{-1}^1,1_{3}^{-1}q_1^1r_{3}^{-1},1_{-3}^{-1}q_{-1}^1r_{3}^{-1}$ & \\\hline
                 & $2$ & $2$ & N & $1_1^1q_1^1r_1^1,1_{-3}^{-1}q_1^1r_{-3}^{-1}$ & $2$\\\thickhline
				 &$ \underset{\text{odd}}{v\geq 3}$ & $1$& Y& $1_1^1q_{-3}^{-1}r_{-3}^{-1},1_1^1q_{-1}^1r_{-1}^1,1_3^{-1}q_3^{-1}r_1^1,1_3^{-1}q_{-3}^{-1}r_{-1}^1,$ & $8$\\
                & & & & $1_{-3}^{-1}q_1^1r_{-3}^{-1}, 1_{-3}^{-1}q_3^{-1}r_{-1}^1,1_{-1}^1q_{-1}^1r_1^1,1_{-1}^1q_1^1r_{-1}^1$ &\\\hline
                 &$ \underset{\text{odd}}{v\geq 3}$ & $1$& N& $1_1^1q_1^1r_1^1,1_1^1q_3^{-1}r_3^{-1},1_{-1}^1q_3^{-1}r_{-3}^{-1},1_{-1}^1q_{-3}^{-1}r_3^{-1}$ & $8$\\
                &  & & & $1_3^{-1}q_1^1r_{3}^{-1},1_3^{-1}q_{-1}^1r_{-3}^{-1},1_{-3}^{-1}q_{-1}^1r_3^{-1},1_{-3}^{-1}q_{-3}^{-1}r_1^1$&\\
                \hline
				 &$ \underset{\text{even}}{v\geq 3}$ & $1$& Y& $1_1^1q_3^{-1}r_3^{-1},1_1^1q_{-1}^1r_{-1}^1,1_3^{-1}q_3^{-1}r_1^1,1_3^{-1}q_1^1r_3^{-1}$, & $12$\\
                &  & & & $1_3^{-1}q_{-1}^1r_{-3}^{-1},1_3^{-1}q_{-3}^{-1}r_{-1}^1, 1_{-3}^{-1}q_{-1}^1r_3^{-1},1_{-3}^{-1}q_3^{-1}r_{-1}^1$, &\\
                & & & & $1_{-1}^1q_{-1}^1r_1^1,1_{-1}^1q_{-3}^{-1}r_3^{-1},1_{-1}^1q_3^{-1}r_{-3}^{-1},1_{-1}^1q_1^1r_{-1}^1$ &\\\hline
                 &$ \underset{\text{even}}{v\geq 3}$ & $1$& N& $1_1^1q_1^1r_1^1,1_1^1q_{-3}^{-1}r_{-3}^{-1},1_{-3}^{-1}q_1^1r_{-3}^{-1},1_{-3}^{-1}q_{-3}^1r_1^1$ & $4$ \\\hline
	\end{tabular}} \end{table}

\newpage
\section{Frequently used notations}\label{notations}

\begin{table}[H]
\renewcommand{\arraystretch}{1.2}%
\centering
\small{\begin{tabular}{|l| l |l|}
\hline
Notation & Definition & 1st occ. \\
\hline

$P_d$
& Product of distinct odd primes $p$ such that $p^2\mid d$
& Def.~2.2 \\

$r(d)$
& Highly ramified part of $d$
& Def.~4.1 \\

$s(d)$
& Odd squarefree part of $d$, so that $d=r(d)s(d)$
& Def.~4.1 \\

\hline

$\Packet$
& Set of packets
& Def.~4.8 \\

$\Packet_H$
& Packets above $H$
& Def.~4.12 \\

$\Supp(H)$
& Support of a packet $H$
& Def.~4.8 \\

$F|_S$
& $S$-restriction of class, genus, or packet $F$
& Def.~4.8 \\

$D(H)$
& Determinant of a packet $H$
& Def.~4.9 \\

$F_{\rt}$
& Root packet of $F$
& Def.~4.11 \\

$\mathfrak R$
& Set of root packets
& Def.~4.11 \\

\hline
\end{tabular} }
\end{table}

\begin{table}[H]
\renewcommand{\arraystretch}{1.4}%
\centering
\small{\begin{tabular}{|l| l |l|}
\hline
Notation & Definition & 1st occ. \\
\hline
$\Gen$
& Set of genera
& Def.\,\ref{pre-def} \\

$\Gen_H$
& Genera above $H$
& Def.\,4.12 \\

$\Gen_H(d)$
& Genera of determinant $d$ above $H$
& Def.\,4.12 \\

$g_H(d)$
& Cardinality $|\Gen_H(d)|$
& Def.\,4.12 \\

$h(G)$
& Number of classes in genus $G$
& Def. \ref{pre-def} \\

\hline

$\Cl$
& Set of classes
& \S1, Def.\,\ref{pre-def} \\

$\Cl_H$
& Classes above $H$
& Def.\,4.12 \\

$\Cl_H(d);\  h_H(d)$
& Classes above $H$ with determinant $d$; cardinality 
& Def.\,4.12 \\

$\Cl_H^{(1)}(d);\  h_H^{(1)}(d)$
& $C\in\Cl_H(d)$ with 1 class in genus; cardinality
& Def.\,5.1 \\

$\Cl_H^{(2+)}(d);\ h_H^{(2+)}(d)$
& $C\in\Cl_H(d)$ with $>1$ class in genus; cardinality
& Def.\,5.1 \\

$i_H(d)$
& Common number of classes in a genus in $\Gen_H(d)$
& Lem.\,5.2 \\

\hline

$\mathcal R_p(F)$
& Integers represented by $F$ over $\mathbb Z_p$
& Def.\,6.1 \\

$\mathcal R(F)$
& Integers represented by $F$ over $\mathbb Z_p$ for all $p$
& Def.\,6.1 \\

$\mathcal R^{(\mathrm{abs})}(F)$
& Set $\{|t|:t\in \mathcal R(F)\}$
& Def.\,6.1 \\

$K(F)$
& Least element of $\mathcal R^{(\mathrm{abs})}(F)$
& Def.\,6.1 \\

$K^\ast(F)$
& Least nonzero element of $\mathcal R^{(\mathrm{abs})}(F)$
& Def.\,6.1 \\

$\kappa(F)$
& Least absolute value of integers  $F$ represents
& (1.0.5) \\

$\kappa^*(F)$
& Least absolute value of $\neq0$ integers  $F$ represents
& Def.\,6.1 \\

$\lambda_H(d,t)$
& Number of classes $C\in \Cl_H(d)$ with $\kappa(C)=t$
& Def.\,6.7 \\

$\lambda_H^{(1)}(d,t)$
& Number of classes $C\in \Cl_H^{(1)}(d)$ with $\kappa(C)=t$
& Def.\,6.7 \\

$\lambda_H^{(2+)}(d,t)$
& Number of classes $C\in \Cl_H^{(2+)}(d)$ with $\kappa(C)=t$
& Def.\,6.7 \\

$\lambda_H^g(d,t)$
& Number of genera $G\in \Gen_H(d)$ with $K(G)=t$
& Def.\,6.7 \\

\hline

$\rho(C)$
& Siegel class invariant attached to a class $C$
& (1.0.8) \\

$\nu(G)$
& Siegel mass of a genus $G$
& Thm.\,\ref{massform} \\

$\nu^{\mathrm{iso}}(d)$
& Total Siegel mass of isotropic classes with Det $=d$
& (9.0.1) \\

$\widetilde{\nu}^{\mathrm{iso}}(d)$
& Normalized isotropic mass $\nu^{\mathrm{iso}}(d)/\nu^{\mathrm{iso}}(1)$
& (9.0.1) \\

$\delta_p(A)$
& Siegel $p$-density of $A$
& Def.\,9.3 \\

$m_p(A)$
& Local $p$-mass of $A$
& Def.\,9.3 \\

$I(s)$
& Dirichlet series associated with $d\,\widetilde{\nu}^{\mathrm{iso}}(d)$
& Sec.\,10.1 \\

$d\lambda^{\mathrm{iso}}$
& Local density measure for primitive isotropic forms
& (10.2.1) \\

\hline
\end{tabular} }
\end{table}

\end{appendices}

\footnotesize
\setlength{\parskip=0.0pt}
\setlength{\lineskip=0.0pt} 


\printbibliography

\end{document}